\documentclass[a4paper,leqno]{amsart}

\usepackage{amsmath}
\usepackage{amsthm}
\usepackage{amssymb}
\usepackage{amsfonts}
\usepackage[applemac]{inputenc}
\usepackage{mathrsfs}
\usepackage{pdfsync}
\usepackage{color} 
\usepackage{mathtools}

\def\C{{\mathbb C}}
\def\R{{\mathbb R}}
\def\N{{\mathbb N}}

\def\le{\leqslant}
\def\ge{\geqslant}

\newcommand{\re}{\mathrm{Re}}
\newcommand{\im}{\mathrm{Im}}
\newcommand{\eps}{\varepsilon}

\DeclareMathOperator{\bd}{\mbox{\boldmath $\delta$}}

\theoremstyle{plain}
\newtheorem{theorem}{Theorem}[section]
\newtheorem{lemma}[theorem]{Lemma}

\newtheorem{proposition}[theorem]{Proposition}

\theoremstyle{definition}

\newtheorem{remark}[theorem]{Remark}
\newtheorem*{remark*}{Remark}

\numberwithin{equation}{section}
\DeclareMathOperator{\sech}{sech}
\DeclareMathOperator{\sgn}{sgn}
\begin{document}

\title[Derivative NLS type equations in two spatial dimensions]{On a class of derivative Nonlinear Schr\"odinger-type equations in two spatial dimensions}

\author[J. Arbunich]{Jack Arbunich}

\author[C.Klein]{Christian Klein}

\author[C. Sparber]{Christof Sparber}

\address[J.~Arbunich]
{Department of Mathematics, Statistics, and Computer Science, M/C 249, University of Illinois at Chicago, 851 S. Morgan Street, Chicago, IL 60607, USA}
\email{jarbun2@uic.edu}

\address[C.~Klein]{Institut de Math\'ematiques de Bourgogne, 9 avenue 
Alain Savary, 21000 Dijon, France}
\email{christian.klein@u-bourgogne.fr}

\address[C.~Sparber]
{Department of Mathematics, Statistics, and Computer Science, M/C 249, University of Illinois at Chicago, 851 S. Morgan Street, Chicago, IL 60607, USA}
\email{sparber@uic.edu}

\begin{abstract}
We present analytical results and numerical simulations for a class of nonlinear dispersive equations in two spatial dimensions. These 
equations are of (derivative) nonlinear Schr\"odinger type and have recently been obtained in \cite{DLS} in the context of nonlinear optics. In contrast to the usual nonlinear Schr\"odinger equation, this new model incorporates the additional effects of self-steepening and partial off-axis variations of the group velocity of the laser pulse. 
We prove global-in-time existence of the corresponding solution for various choices of parameters, extending earlier results of \cite{AAS}. 
In addition, we present a series of careful numerical simulations concerning the (in-)stability of stationary states and the possibility of finite-time blow-up.
\end{abstract}

\date{\today}


\thanks{This publication is based on work supported by the NSF through grant no. DMS-1348092. Additional 
support by the NSF research network KI-Net is also acknowledged. C.K. 
also acknowledges partial support by the ANR-FWF project ANuI and by the Marie-Curie 
RISE network IPaDEGAN}

\maketitle

 \tableofcontents


\section{Introduction}\label{sec:intro}

This work is devoted to the analysis and numerical simulations for the following class of nonlinear dispersive equations in two spatial dimensions:
\begin{equation}\label{PDNLS}
i P_{\eps}\partial_{t}u + \Delta u + (1+i \mbox{\boldmath $\delta$} \cdot \nabla )\big(|u|^{2\sigma}u\big) =0, \quad u_{\mid t=0}=u_0(x),
\end{equation}
where $x=(x_1, x_2)\in \R^2$, $\mbox{\boldmath $\delta$}=(\delta_1,\delta_2)^\top \in \R^{2}$ is a given vector with $|\mbox{\boldmath $\delta$}| \le 1$, and $\sigma> 0$ is a 
parameter describing the strength of the nonlinearity. 
In addition, for $0< \eps \le 1$, we denote by $P_\eps$ the following linear differential operator,
\begin{equation}\label{P}
P_{\eps} = 1 - \eps^2 \sum_{j =1}^k \partial_{x_j}^2, \quad k\le 2.
\end{equation}
Indeed, we shall mainly be concerned with \eqref{PDNLS} rewritten in its {\it evolutionary form}:
\begin{equation}\label{PDNLSevol}
i \partial_{t}u + P^{-1}_{\eps}\Delta u + P_\eps^{-1}(1+i \mbox{\boldmath $\delta$} \cdot \nabla )\big(|u|^{2\sigma}u\big) =0, \quad u_{\mid t=0}=u_0(x_1,x_2). 
\end{equation}
Here and in the following, $P_\eps^{s}$, for any $s\in \R$, is the non-local operator defined through multiplication in Fourier space using the symbol
$$\widehat{P}^s_\eps(\xi)= \Big(1+ \eps^2\sum_{j =1}^k \xi_j^2\Big)^s, \quad k\le 2,$$
where $\xi = (\xi_1,\xi_2)\in \R^2$ is the Fourier variable dual to $x=(x_1,x_2)$. For $s=-1$ this obviously yields a bounded operator $P^{-1}_\eps: L^2(\R_x^2)\to  L^2(\R_x^2)$. 
In addition, $P_\eps$ is seen to be uniformly elliptic provided $k=2$.  Moreover when $\eps =1$ and $k=2$, note we can define the $L^{2}(\R_{x}^{2})$-based Sobolev spaces for $s \in \R$ via the norm 
$$
 \|f \|_{H_{x}^{s}} = \big \|P_1^{s/2}f \big \|_{L^{2}_{x}}:=\left(\int_{\R^{2}}|\widehat{P}_1^{s/2}\widehat{f}(\xi)|^2\;d\xi \right)^{\frac12}.
$$

The inclusion of $P_\eps$ implies that \eqref{PDNLS}, or equivalently \eqref{PDNLSevol}, shares 
a formal similarity with the well-known {\it Benjamin-Bona-Mahoney equation} for uni-directional shallow water waves \cite{BBM, BMR}. 
However, the physical context for \eqref{PDNLS} is  rather different. 
Equations of the form \eqref{PDNLS} have recently been derived in \cite{DLS} as an effective description for the propagation of high intensity laser beams. 
This was part of an effort to remedy some of the shortcomings of the classical 
(focusing) {\it nonlinear Schr\"odinger equation} (NLS), which is obtained from \eqref{PDNLS} when $\eps=\delta_{1}=\delta_2=0$, i.e.
\begin{equation}\label{NLS}
i \partial_{t}u + \Delta u + |u|^{2\sigma}u =0, \quad u_{\mid t=0}=u_0(x_1,x_2). \\
\end{equation}
The NLS is a canonical model for slowly modulated, self-focusing wave propagation in a weakly nonlinear 
dispersive medium. The choice of $\sigma =1$ thereby corresponds to the physically most relevant case of a {\it Kerr nonlinearity}, cf. \cite{Fi, Sul}. 
Equation \eqref{NLS} is known to conserve, among other quantities, the {\it total mass}
\[
M(t)\equiv\| u(t,\cdot)\|^2_{L_{x}^2} = \| u_0 \|^2_{L_{x}^2}.
\]
A scaling consideration then indicates that \eqref{NLS} is $L^2$-critical for $\sigma =1$ and $L^2$-super-critical for $\sigma> 1$. 
It is well known that in these regimes, solutions to \eqref{NLS} may not exist for all $t\in \R$, due to the possibility of {\it finite-time blow-up}. The latter 
means that there exists a time $T<\infty$, depending on the initial data $u_0$, such that 
\[
\lim_{t\to T_-} \| \nabla u(t,\cdot) \|_{L_{x}^2} = + \infty.
\]
In the physics literature this is referred to as {\it optical collapse}, see \cite{Fi}. 

In the $L^2$-critical case,  there is a sharp dichotomy characterizing the possibility 
of this blow-up: Indeed, one can prove that the solution $u$ to \eqref{NLS} with $\sigma =1$ exists for all $t\in \R$, provided its total mass is below that of the {\it nonlinear ground state}, i.e., 
the least energy (nonzero) solution of the form $$u(t,x) = e^{it} Q(x).$$ Solutions $u$ whose $L^2$-norm exceeds the norm of $Q$, however, will in general exhibit a self-similar blow-up 
with a profile given by $Q$ (up to symmetries), see \cite{MR1, MR2}. In turn, this also implies that stationary states of the form $e^{it} Q(x)$ are {\it strongly unstable}. 
For more details on all this we refer the reader to \cite{Cz, Fi, Sul} and references therein.

In comparison to \eqref{NLS}, the new model \eqref{PDNLSevol} includes two additional physical effects. 
Firstly, there is an additional nonlinearity of derivative type which describes the possibility of {\it self-steepening} of the laser pulse in the direction $\mbox{\boldmath $\delta$}\in \R^2$. 
Secondly, the operator $P_\eps$ describes {\it off-axis variations} of the group velocity of the beam. 
The case $k=2$ is thereby referred to as {\it full off-axis dependence}, whereas for $k=1$ the model incorporates only a {\it partial off-axis variation}. 
Both of these effects become more pronounced for high beam intensities (see \cite{DLS}) 
and both are expected to have a significant influence on the possibility of finite-time blow-up. 
In this context, it is important to note that \eqref{PDNLSevol} does not admit a simple scaling invariance analogous to \eqref{NLS}. Hence, 
there is no clear indication of sub- or super-critical regimes for equation \eqref{PDNLSevol}.  At least formally, though, equation \eqref{PDNLSevol} admits the following conservation law,
\begin{equation} \label{CLM}
M_\eps(t)\equiv \| P^{1/2}_{\eps}u(t, \cdot) \|^2_{L_{x}^2} =\| P^{1/2}_{\eps}u_{0}\|^2_{L_{x}^2},
\end{equation}
generalizing the usual mass conservation.
In the case of full-off axis dependence, \eqref{CLM} yields an a-priori bound on the $H^1$-norm of $u$, ruling out 
the possibility of finite-time blow-up. However, the situation is more complicated in the case with only a  
partial off-axis variation. 

The latter was studied analytically in the recent work \cite{AAS}, but only for the much simpler case {\it without} self-steepening, i.e., only for $\delta_1=\delta_2=0$.  
It was rigorously shown that in this case, even a partial off-axis variation (mediated by $P_\eps$ with $k=1$) can arrest the blow-up for all $\sigma <2$. 
In particular, this allows for nonlinearities larger than the $L^2$-critical case, cf. Section \ref{sec:anpartial} for more details. 
One motivation for the present work is to give numerical evidence for the fact that these results are indeed sharp, and that one can expect 
finite-time blow-up as soon as $\sigma \ge 2$. 

The current work aims to extend the analysis of \cite{AAS} to situations with additional self-steepening, i.e., $\mbox{\boldmath $\delta$}\not =0$, 
and to provide further insight into the qualitative interplay between this effect and the one stemming from $P_\eps$. 
From a mathematical point of view, the addition of a derivative nonlinearity makes the question of global well-posedness versus finite-time blow-up 
much more involved. Derivative NLS and their corresponding ground states are usually studied in one spatial dimension only, see e.g. 
\cite{AmSi, CO, HaOz, GNW, LSS1, LSS2, TF, Wu} and references therein. For $\sigma =1$, the classical one-dimensional derivative NLS 
is known to be completely integrable. Furthermore, there has only very recently been a breakthrough in the proof of global-in-time existence for this case, see \cite{JLPS1, JLPS2}. In contrast to that, 
\cite{LSS2} gives strong numerical indications for a self-similar finite-time blow-up in derivative NLS with $\sigma >1$. 
The blow-up thereby seems to be a result of the self-steepening effect in the density $\rho=|u|^2$, which generically 
undergoes a time evolution similar to a dispersive shock wave formation in Burgers' equation. To our knowledge, however, no 
rigorous proof of this phenomenon is currently available. 

In two and higher dimensions, even the local-in-time existence of solutions to derivative NLS type equations seems to be largely unknown, let alone any further qualitative properties of their solutions. 
In view of this, the present paper aims to shine some light on the specific variant of two-dimensional derivative NLS given by \eqref{PDNLSevol}. Except for its physical significance, this 
class of models also has the advantage that the inclusion of (partial) off-axis variations via $P_\eps$ are expected to have a strong 
regularizing effect on the solution, and thus allow for several stable situations without blow-up. 

The organization of our paper is then as follows:
\begin{itemize}
\item In Section \ref{sec:ground}, we shall numerically construct nonlinear stationary states to \eqref{PDNLS}, or equivalently \eqref{PDNLSevol}. 
These also include the well-known ground states for the classical NLS.
For the sake of illustration, we shall also derive explicit formulas for the one-dimensional case and compare them with the 
well-known formulas for the classical (derivative) NLS. 
\item Certain perturbations of these stationary states will form the class of initial data considered in the numerical time-integration of \eqref{PDNLSevol}. 
The numerical algorithm used to perform the respective simulations is detailed in Section \ref{sec:numtime}. In it, we also 
include several basic numerical tests which compare the new model \eqref{PDNLSevol} to the classical (derivative) NLS. 
\item Analytical results yielding global well-posedness of \eqref{PDNLSevol} with either full or partial off-axis variations are given in Sections \ref{sec:full} and \ref{sec:anpartial}, respectively. 
\item In the former case, the picture is much more complete, which allows us to perform a numerical study of the (in-)stability properties of 
the corresponding stationary states, see Section \ref{sec:numfull}. 
\item In the case with only partial off-axis variations, the problem of global existence is more complicated and one needs to distinguish between the cases where the action of $P_\eps$ 
is either parallel or orthogonal to the self-steepening. Analytically, only the former case can be treated so far (see Section \ref{sec:anpartial}). 
Numerically, however, we shall present simulations for both of these cases in Section \ref{sec:numpartial}. 
\end{itemize}


\section{Stationary states}\label{sec:ground}

In this section, we focus on {\it stationary states}, i.e., time-periodic solutions to \eqref{PDNLS} given in the following form:
\begin{equation}\label{gstate}
u(t,x_1,x_2)= e^{it}Q(x_1,x_2).
\end{equation} 
The function $Q$ then solves 
\begin{equation}
	 P_\eps Q=\Delta Q+(1+i\mbox{\boldmath $\delta$}\cdot 
	 \nabla)(|Q|^{2\sigma}Q)
	 \label{PDNLSsol},
\end{equation}
subject to the requirement that $Q(x)\to 0$ as $|x|\to \infty$. Every non-zero solution $Q(x)\in \C$ gives rise to a 
solitary wave solution (with speed zero) to \eqref{PDNLS}. These solitary waves will be an important benchmark for our numerical simulations later on. 
Note that in \eqref{gstate} we only allow for a simple time-dependence 
$\exp(i\omega t)$ with $\omega=1$ in \eqref{gstate}. This is not a restriction for the usual 2D NLS, given its scaling invariance, 
but it is a restriction for our model in which this invariance is broken (see also \cite{GNW, LSS1} for the connection between $\omega$ and 
the speed of stable solitary waves).

For the classical NLS, i.e., $\eps=0$ and $|\bd | =0$, there exists a particular solution $Q$, called {\it the nonlinear ground state}, which is 
the unique radial and positive solution to \eqref{PDNLSsol}, cf. \cite{Fi, Sul}. Recall that in dimensions $d=2$ the NLS is already $L^2$-critical and thus, ground states, 
in general, cannot be obtained as minimizers of the associated energy functional (which is the same for both $\eps=0$ and $\eps >0$, see \cite{DLS}).
As we shall see below for $\eps>0$, the regularization via $P_\eps$  
yields a natural modification of the ground state $Q$ by smoothly widening its profile (while conserving positivity). We shall thus also refer to these solutions $Q$ as the 
ground states for \eqref{PDNLSsol} with $|\bd| = 0$ and $\eps>0$. At present, there are 
unfortunately no analytical results on the existence and uniqueness of such modified ground states available. However, our 
numerical algorithm indicates 
that they exist and are indeed unique (although, in general no longer radially symmetric).

The situation with 
derivative nonlinearity $|\bd|\not =0$ is somewhat more complicated, since in this case, solutions $Q$ to \eqref{PDNLSsol} are always complex-valued 
and hence the notion of a ground state does not directly extend to this case (recall that uniqueness is only known for positive solutions). 
At least in $d=1$, however, explicit calculations (see below) show, that there is a class of smooth $\bd$-dependent stationary solutions to \eqref{PDNLSsol}, which 
for $|\bd|=0$ yield the family of $\eps$-ground states.


\subsection{Explicit solutions in 1D}\label{sec:formulas}
In one spatial dimension, equation \eqref{PDNLSsol} allows for explicit formulas, 
which will serve as a basic illustration for the combined effects of self-steepening and off-axis variations. 
Indeed, in one spatial dimension, equation \eqref{PDNLS} simplifies to
\begin{equation}\label{eq:dNLS1d}
i (1-\eps^{2}\partial_{x}^{2})\partial_{t}u + \partial_{x}^{2}u + (1+ i\delta\partial_{x})(|u|^{2\sigma}u) =0.
\end{equation}
Seeking a solution of the form \eqref{gstate} thus yields the following ordinary differential equation:
\begin{equation}\label{eq:Q}
(1+\eps^{2})Q'' + (|Q|^{2\sigma}-1)Q +i\delta(|Q|^{2\sigma}Q)' = 0.
\end{equation}
To solve this equation, we shall use the polar representation for $Q(x)\in \C$
$$Q(x)=A(x)e^{i\theta(x)}, \quad A(x), \theta(x)\in \R$$ 
where we impose the requirement that 
$A(x)\ge 0$ and $\lim_{x\to\pm\infty}A(x)=0$.
Plugging this ansatz into \eqref{eq:Q}, factoring $e^{i\theta}$ out and isolating 
the real and imaginary part yields the following coupled system:
\begin{align*}
      \big(1+\eps^{2}\big)A'' + (A^{2\sigma}-1)A  -A\theta'\Big( (1+\eps^{2})\theta^{\prime} + \delta A^{2\sigma} \Big)=0,\\
      \big(1+\eps^{2}\big)\big(A\theta'' + 2\theta'A' \big) +(2\sigma +1)\delta A^{2\sigma}A'=0 .
\end{align*}
Multiplying the second equation by $A$ and integrating from $-\infty$ to $x$ gives
$$
(1+\eps^{2})\theta' = -\frac{(2\sigma+1)\delta A^{2\sigma}}{2(\sigma +1)},
$$
where here we implicitly assume that $A^{2}\theta'$ vanishes at infinity.
Using the above, we infer that the amplitude solves
\begin{equation}\label{Amp}
\big(1+\eps^{2}\big)A'' + (A^{2\sigma}-1)A  +\frac{(2\sigma+1)\delta^{2}}{4(1+\eps^{2})(\sigma+1)^{2}} A^{4\sigma+1} =0,
\end{equation}
while the phase is given a-posteriori through
\begin{equation}\label{phase}
\theta(x) =  -\frac{(2\sigma+1)\delta}{2(1+\eps^{2})(\sigma+1)}\int_{-\infty}^{x} A^{2\sigma}(y)\;dy.
\end{equation}

After some lengthy computation, similar to what is done for the usual NLS, cf. \cite{Fi}, the 
solution to (\ref{Amp}) can be written in the form
\begin{equation}\label{Ampex}
    A (x) = \left(\frac{2(\sigma+1)}{1+K_{\eps,\delta}\cosh 
   \left( \frac{2 \sigma x}{\sqrt{1+\eps^{2}}}\right)}\right)^{1/(2\sigma)},
\end{equation}
where 
$K_{\eps,\delta}=\sqrt{1+\frac{\delta^{2}}{1+\eps^{2}}}>0$.
In view of (\ref{phase}), this implies that the phase function $\theta$ is given by
\begin{equation}\label{phaseex}
    \theta (x) = 
    -\sgn(\delta)(2\sigma+1)\arctan\left(\frac{\sqrt{1+\eps^{2}}}{|\delta|}\left(1+K_{\eps,\delta}e^{\frac{2\sigma x}{\sqrt{1+\eps^{2}}}}\right)\right),
\end{equation}
where we omitted a physically irrelevant constant in the phase (clearly, $Q$ is only unique up to multiplication by a constant phase).

Note that in the case with no self-steepening $\delta=0$, the phase $\theta$ is zero. Thus, $Q(x)\equiv A(x)$ and we find
$$
Q(x)=(\sigma+1)^{1/(2\sigma)}\sech^{1/\sigma}\left(  \frac{\sigma x}{\sqrt{1+\eps^{2}}}\right).
$$
For $\eps=0$, this is the well-known ground state solution to \eqref{NLS} in one spatial dimension, cf. \cite{Fi, Sul}. 
We notice that adding the off-axis dispersion ($\eps >0$) widens 
the profile, causing it to decay more slowly as $x\to \pm \infty$ as can be 
seen in Fig.~\ref{fig1d} on the left. On the right of 
Fig.~\ref{fig1d}, it is shown that the maximum of the ground state 
decreases with $\sigma$ but that the peak becomes more compressed. 

\begin{figure}[htb!] 
  \includegraphics[width=0.49\textwidth]{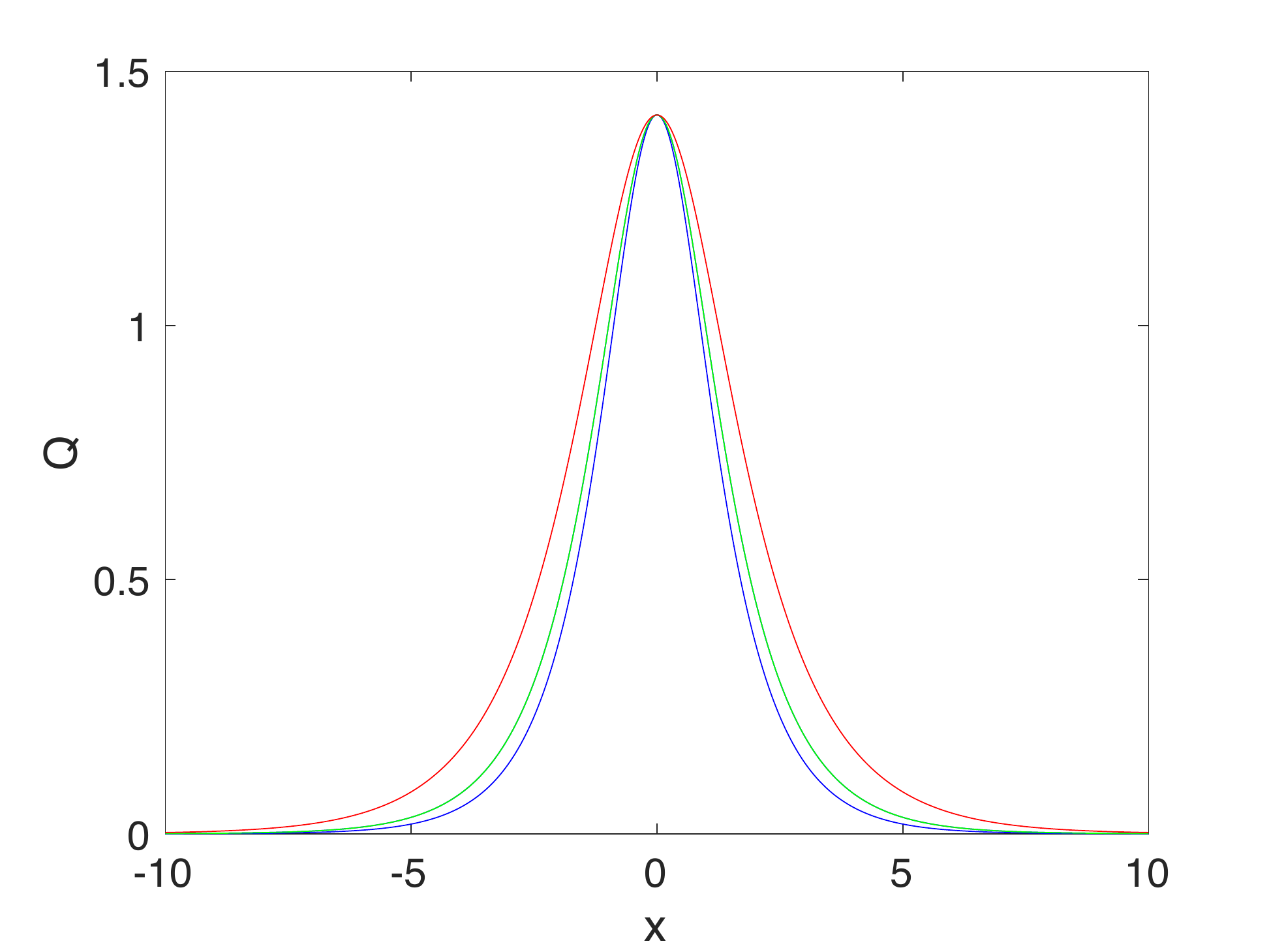}
  \includegraphics[width=0.49\textwidth]{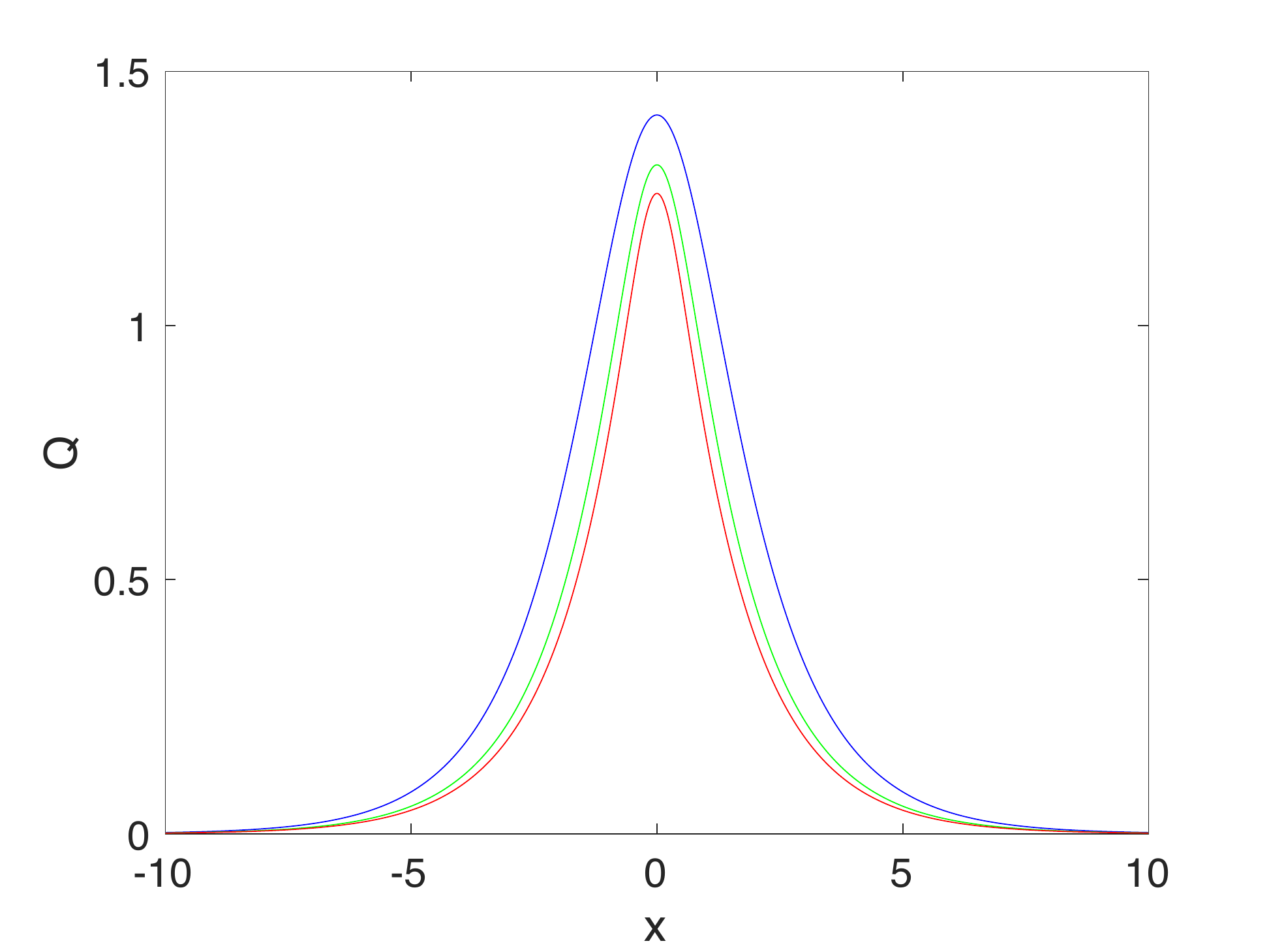}
 \caption{Ground state solution to (\ref{Ampex}) with $\delta=0$: On the 
 left for $\sigma=1$ and $\eps=0$ (blue), $\eps=0.5$ (green) and 
 $\eps=1$ (red). On the right for $\eps=1$ and $\sigma=1$ (blue), 
 $\sigma=2$  (green) and $\sigma=3$ (red). }
 \label{fig1d}
\end{figure}

\begin{remark}
The ($\sigma$-generalized) one-dimensional derivative NLS can be obtained from \eqref{eq:dNLS1d} by putting $\eps =0$, rescaling $$u(t,x)= \delta^{-1/(2\sigma)}\tilde{u}(t,x),$$ and letting $\delta\to\infty$. Note that $\tilde{u}$ solves
$$
i \partial_{t}\tilde{u} + \partial_{x}^{2}\tilde{u} + (\delta^{-1}+ i\partial_{x})(|\tilde{u}|^{2\sigma}\tilde{u}) =0.
$$
Denoting $\tilde{Q}=\tilde{A}e^{i\tilde \theta(x)}$, we get from 
(\ref{Ampex}) and (\ref{phaseex}) the well-known zero-speed solitary wave solution
of the derivative NLS, i.e.,
$$
\tilde{A}(x) = \big(2(\sigma+1)\sech (2\sigma x)\big)^{1/(2\sigma)}, 
\quad \tilde \theta (x) = -(2\sigma+1) \arctan (e^{2\sigma x}).$$ 
The stability of these states has been studied in, e.g., \cite{CO, LSS1, GNW}.
\end{remark}


\subsection{Numerical construction of stationary states}\label{sec:numground} 

In more than one spatial dimension, no explicit formula is known for $Q$. 
Instead, we shall numerically construct $Q$ by following an approach similar to those in \cite{KS, KMS}. 
Since we can expect $Q$ to be rapidly decreasing, we use a Fourier spectral method and approximate
$$
\mathcal{F}(Q)\equiv \widehat{Q}(\xi_1,\xi_2) = \frac{1}{2\pi}\iint_{\R^{2}}Q(x_1,x_2)e^{-ix_{1}\xi_1}e^{-ix_{2}\xi_2}\;dx_1 dx_2,
$$
by a discrete Fourier transform which can be efficiently computed via the 
Fast Fourier Transform (FFT). 
In an abuse of notation, we shall in the following use the same symbols for the discrete and continuous Fourier transform.
To apply FFTs, we will use a {\it computational domain} of the form
\begin{equation}\label{domain}
\Omega=[-\pi,\pi]L_{x_1}\times[-\pi,\pi]L_{x_2},
\end{equation}  
and choose $L_{x_1}$, $L_{x_2}>0$ sufficiently large so that the obtained Fourier coefficients of $Q$ decrease to 
machine precision, roughly $10^{-16}$, which in practice is slightly larger due to unavoidable rounding errors. 

Now, recall that for a solution of the form (\ref{gstate}) to satisfy (\ref{PDNLS}), the function $Q$ needs to solve \eqref{PDNLSsol}.
In Fourier space, this equation takes the simple form
\begin{equation*}\label{PDNLSsolf}
	\widehat{Q}(\xi_1,\xi_2)= \widehat \Gamma_\eps \mathcal F{(|Q|^{2\sigma}Q)}(\xi_1, \xi_2) ,
\end{equation*}
where 
\begin{equation*}
\widehat \Gamma_\eps(\xi_1, \xi_2)= 
\frac{(1-\delta_{1}\xi_{1}-\delta_{2}\xi_{2} )}{1+\xi_1^{2} + \xi_2^{2} + \eps^{2}\sum_{i=1}^{k}\xi_{i}^{2}}.
\end{equation*}
For $\delta_1=\delta_2=0$, the solution $Q$ can be chosen to be real, but this 
will no longer be true for $\delta_{1,2}\not =0$. In the latter situation, 
we will decompose $$Q(x_1,x_2)=\alpha(x_1, x_2)+i\beta(x_1, x_2),$$ 
and separate (\ref{PDNLSsol}) 
into its real and imaginary part, yielding a coupled nonlinear system for $\alpha, \beta$.  
By using FFTs, this is equivalent to the following system for $\widehat{\alpha}$ and $\widehat{\beta}$:
\begin{equation*}
\left\{  
	\begin{split}
		\widehat{\alpha}(\xi_1, \xi_2) - \widehat \Gamma_\eps \mathcal F \Big( \big(\alpha^2 +  \beta^2\big)^\sigma \alpha\Big)(\xi_1, \xi_2)=0, \\
		\widehat{\beta}(\xi_1, \xi_2) - \widehat \Gamma_\eps \mathcal F \Big( \big(\alpha^2 + \beta^2\big)^\sigma \beta\Big)(\xi_1, \xi_2)=0.
	\end{split}
\right.
\end{equation*}
Formally, the system can be written as $M(\widehat q)=0$ where $\widehat q=(\widehat{\alpha},\widehat{\beta})^\top$ and 
solved via a {\it Newton iteration}. One thereby starts from an initial 
iterate $\widehat q^{(0)}$ and computes the $n$-th iterate via the well known formula
\begin{equation}
    \widehat q^{(n)}=\widehat q^{(n-1)}-J\Big(\widehat q^{(n-1)}\Big)^{-1}M\Big(\widehat q^{(n-1)}\Big) \quad n\in\mathbb{N}
    \label{newton},
\end{equation}
where ${J}$ is the Jacobian of $M$ with respect to $\widehat q$. Since our required numerical resolution makes it 
impossible to directly compute the action of the inverse Jacobian, we instead 
employ a Krylov subspace approach as in \cite{GMRES}. Numerical experiments show that when the initial iterate $\widehat q^{(0)}$ 
is sufficiently close to the final solution, we obtain the expected quadratic convergence of our scheme and reach a precision of 
order $10^{-10}$ after only 4 to 8 iterations. 

As a basic test case, we compute the ground state 
of the standard two-dimensional focusing NLS with $\sigma =1$,   
using the initial iterate
$$q^{(0)}(x_1,x_2)=\mbox{sech}^{2}\Big(\sqrt{x_1^{2}+ x_2^{2}}\Big)$$
on the computational domain \eqref{domain} with $L_{x_1}=L_{x_2}=5$.
By choosing $N_{x_1}=N_{x_2}=2^{9}$ many Fourier modes, we have after seven 
iterations of \eqref{newton} a residual smaller than $10^{-12}$. The obtained
solution is given on the left of Fig.~\ref{NLS2dsol}. As expected, the solution is radially symmetric.
\begin{figure}[htb!]
  \includegraphics[width=0.32\textwidth]{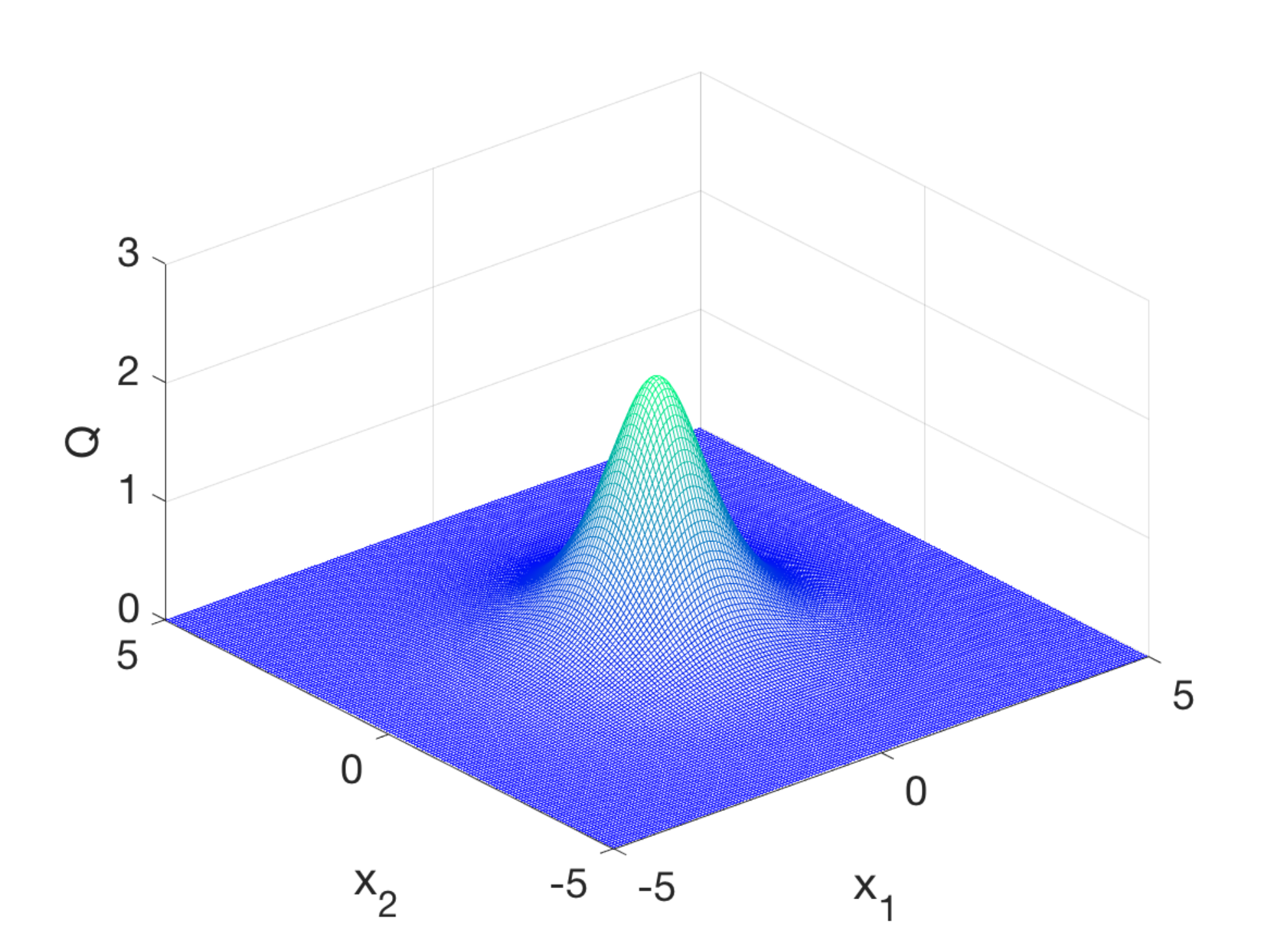}
  \includegraphics[width=0.32\textwidth]{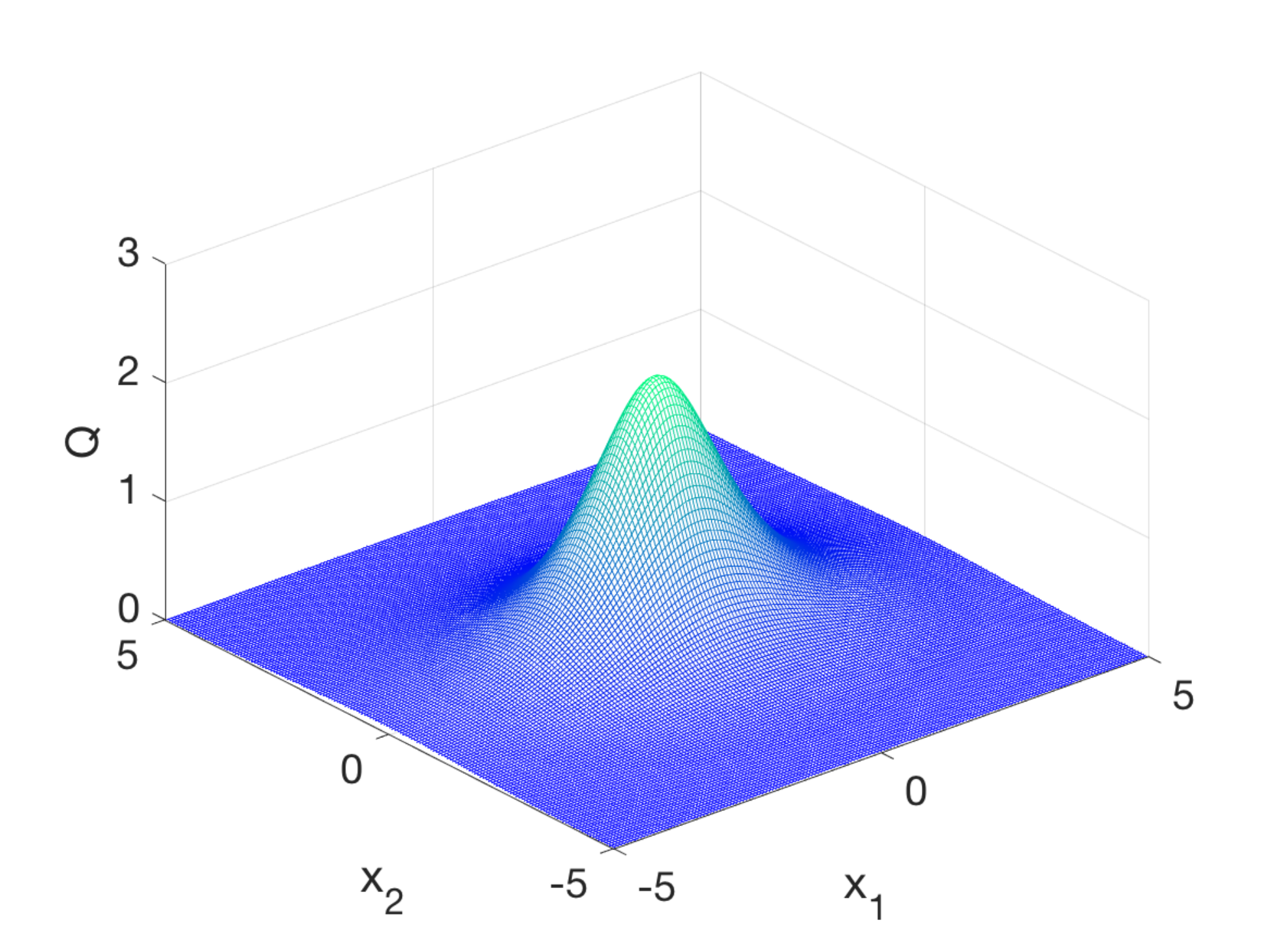}
  \includegraphics[width=0.32\textwidth]{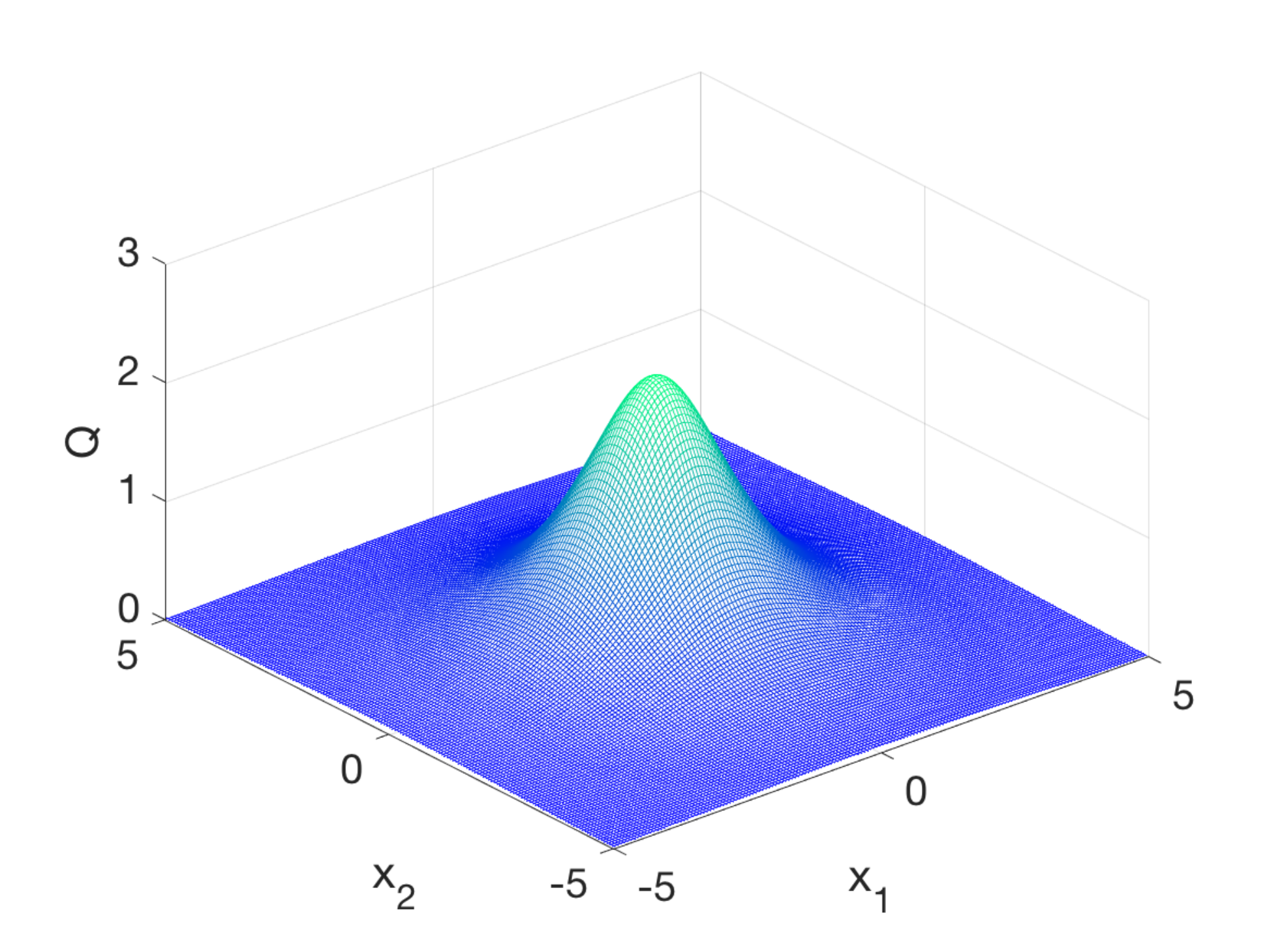}
 \caption{Ground state solution to equation (\ref{PDNLS}) with $\sigma =1$ and 
 $\mathbf{\delta}=0$: On the 
 left for $\eps=0$, in the middle for $\eps=1$ and $k=1$ (partial off-axis dependence), on 
 the right for $\eps=1$ and $k=2$ (full off-axis dependence). }
 \label{NLS2dsol}
\end{figure}

The numerical ground state solution hereby obtained will then be used as an initial iterate for the situation with non-vanishing $\eps$ and $\mbox{\boldmath $\delta$}$, as follows:

{\it Step 1:} In the case without self-steepening $\delta_1=\delta_2=0$, the iteration is straightforward even for relatively large values such as $\eps=1$. 
It can be seen in  the middle of Fig.~\ref{NLS2dsol}, that the ground state for $\eps=1$ and $k=1$ is no longer radially symmetric. 
As an effect of the partial off-axis variation, the solution is elongated in the $x_1$-direction. In the case of full off-axis dependence, 
the ground state for the same value of $\eps=1$ can be seen in Fig.~\ref{NLS2dsol} on the right.  
The solution is again radially symmetric, but as expected less localized than the ground state of the classical standard NLS. 
This is consistent with the explicit formulas for $Q$ found in the one dimensional case above.

{\it Step 2:} In the case with self-steeping $\delta_1=\delta_2=1$, smaller intermediate steps have to be used in the iterations: 
We increment $\mbox{\boldmath $\delta$}$, by first varying only $\delta_{2}$ in steps of $0.2$, always using the last computed value for $Q$ as an initial 
iterate for the slightly larger $\mbox{\boldmath $\delta$}$. The resulting solution $Q$ can be seen in Fig.~\ref{NLS2dsold1}. 
Note that the imaginary part of $Q$ is of the same order of magnitude as the real part. 
\begin{figure}[htb!] 
  \includegraphics[width=0.49\textwidth]{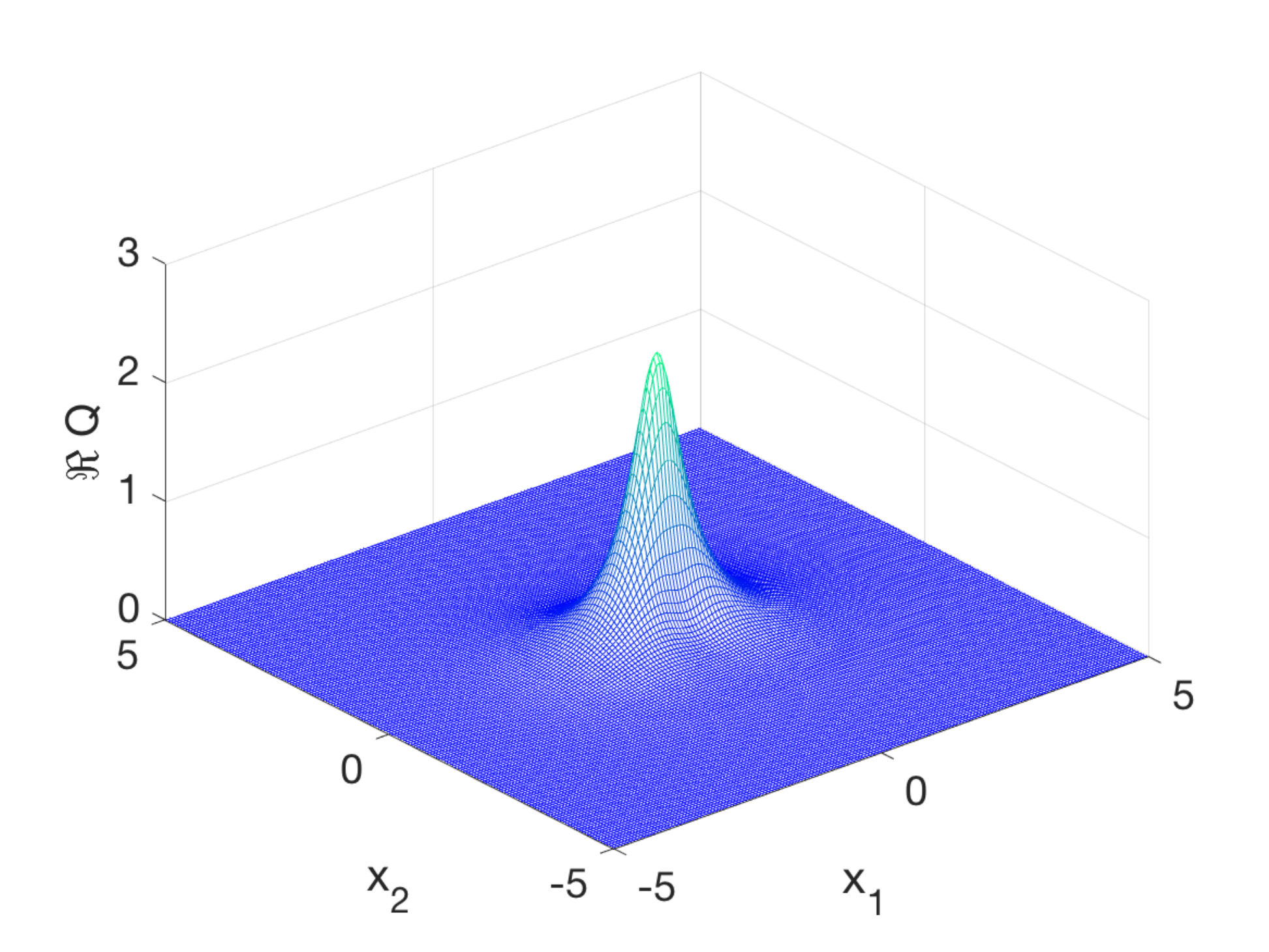}
  \includegraphics[width=0.49\textwidth]{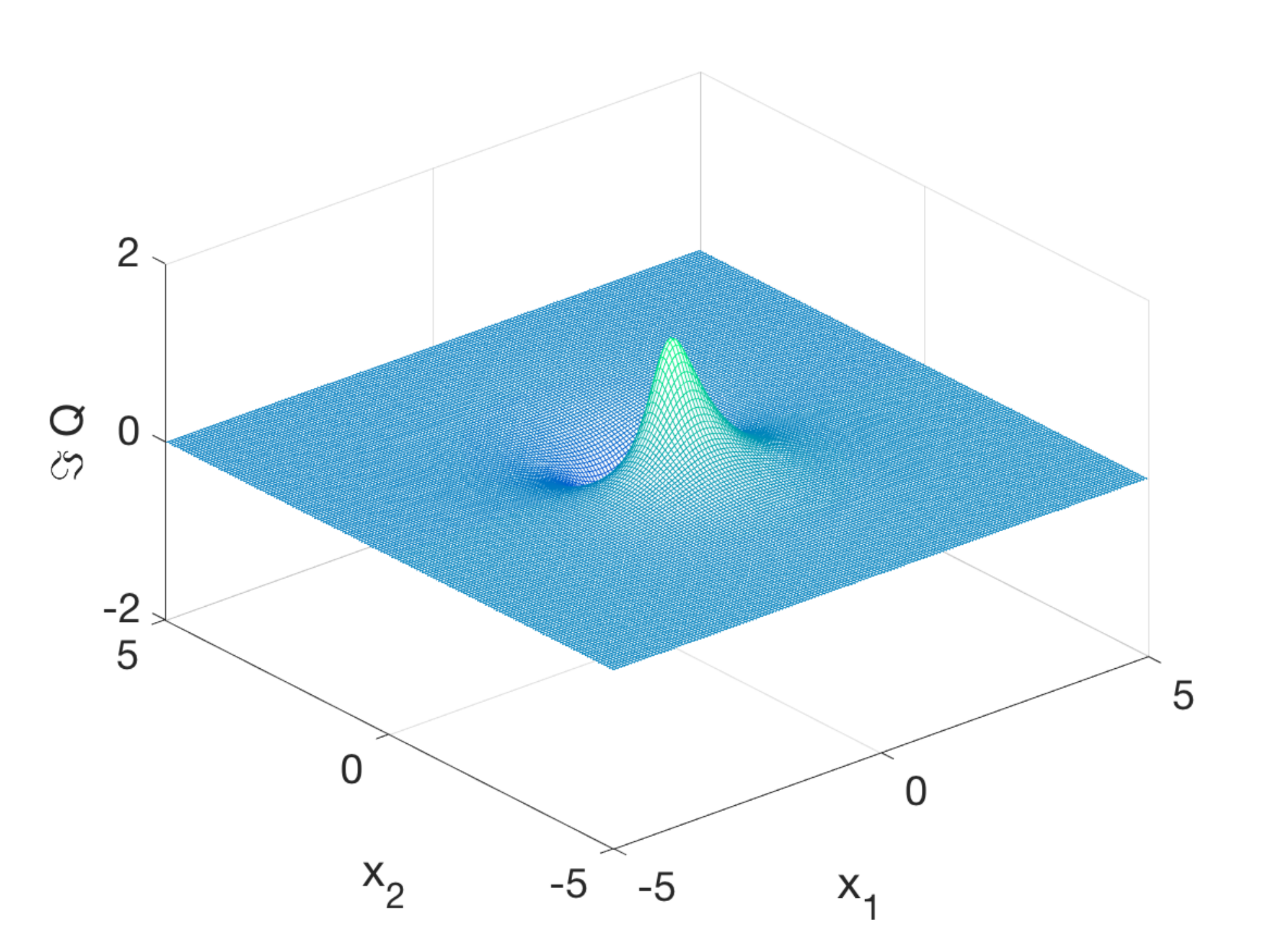}
 \caption{The stationary state solution $Q$ to equation (\ref{PDNLS}) with $\sigma =1$,  
 $\eps=\delta_{1}=0$ and $\delta_{2}=1$: On the 
 left, the real part of $Q$, on 
 the right its imaginary part. }
 \label{NLS2dsold1}
\end{figure}

{\it Step 3:} In order to combine both effects within the same model, we shall use the (zero speed) solitary  
obtained for $\eps=0$ and $\mbox{\boldmath $\delta$} \ne \mathbf{0}$ as an initial 
iterate for the case of non-vanishing $\eps$. In Fig.~\ref{NLS2dsold1e1} we show on the left the stationary state for $\eps=1$, $k=1$, $\delta_{1}=0$ and $\delta_{2}=1$, when the action of $P_\eps$ is
orthogonal to the self-steepening. When compared to the case with $\eps=0$, the solution is seen to be elongated in the $x_1$-direction. 
Next, we simulate when $P_\eps$ acts parallel to the self-steepening, that is when $\eps=1$, $k=1$, $\delta_{1}=1$ 
and $\delta_{2}=0$. The result is shown in the middle of Fig.~\ref{NLS2dsold1e1}. 
In comparison to the former case, the imaginary part of the solution is essentially rotated clockwise by 90 degrees. The elongation effect in the $x_1$-direction is still visible but less pronounced. 
\begin{figure}[htb!]
  \includegraphics[width=0.32\textwidth]{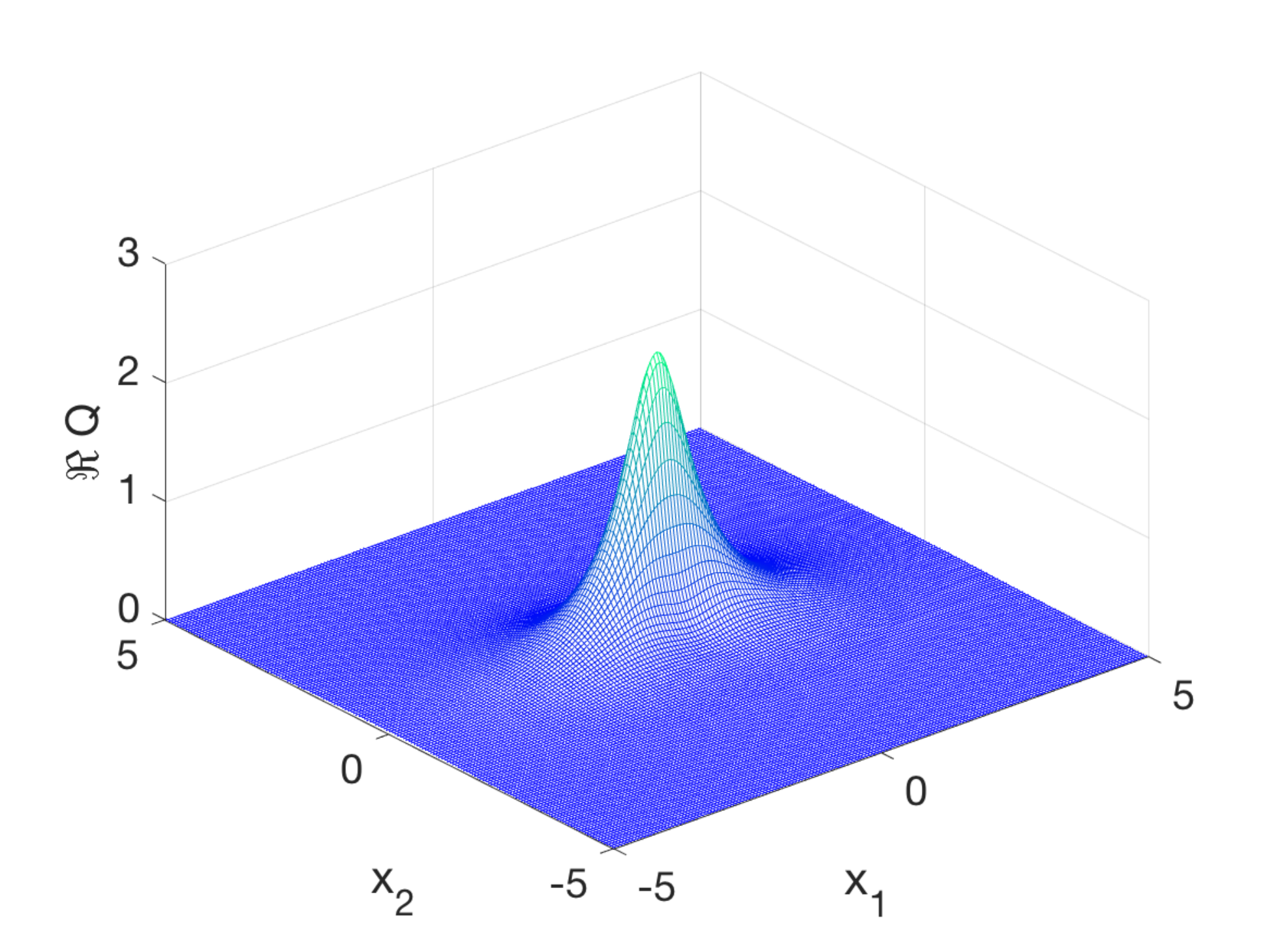}
  \includegraphics[width=0.32\textwidth]{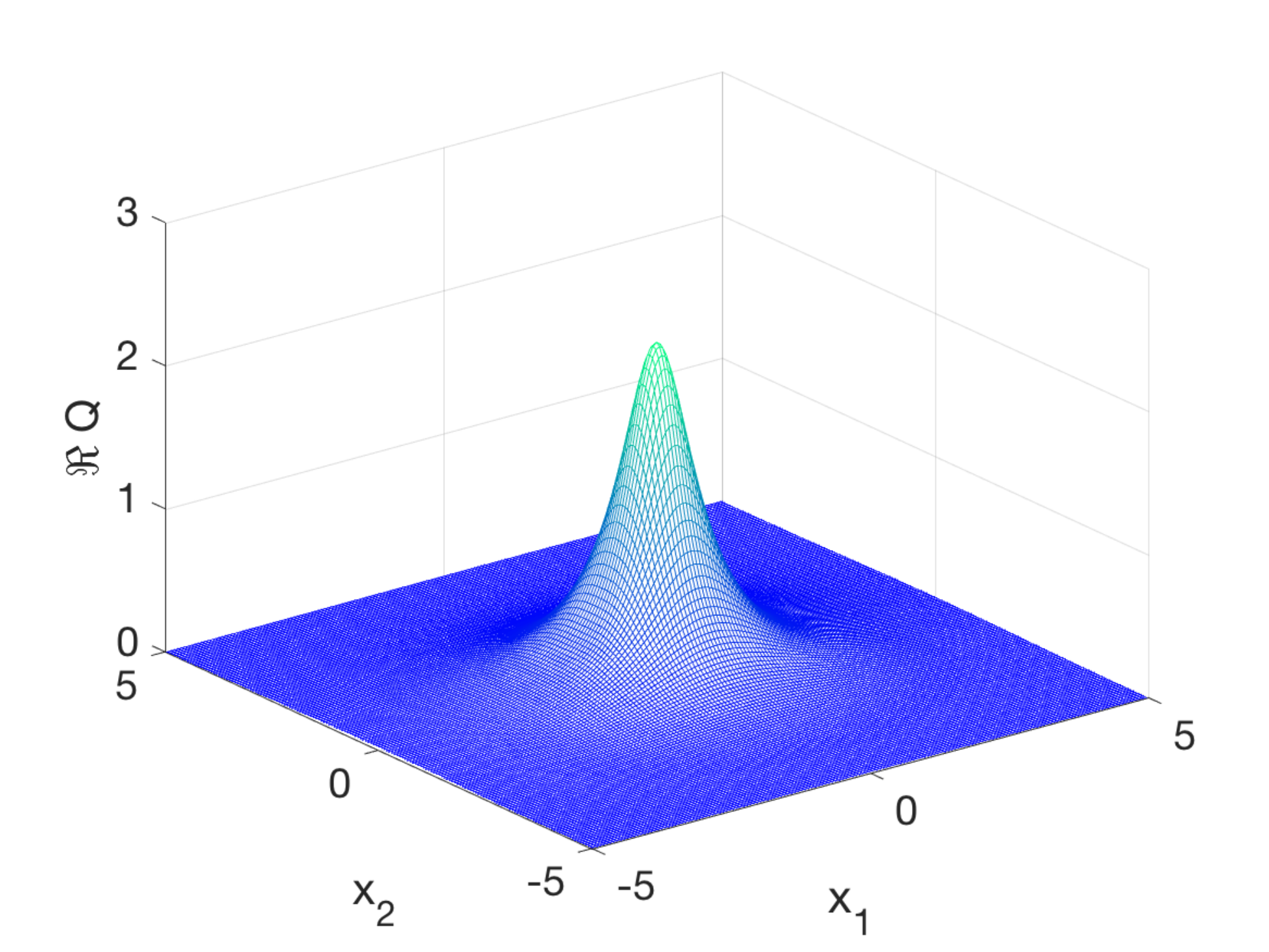}
  \includegraphics[width=0.32\textwidth]{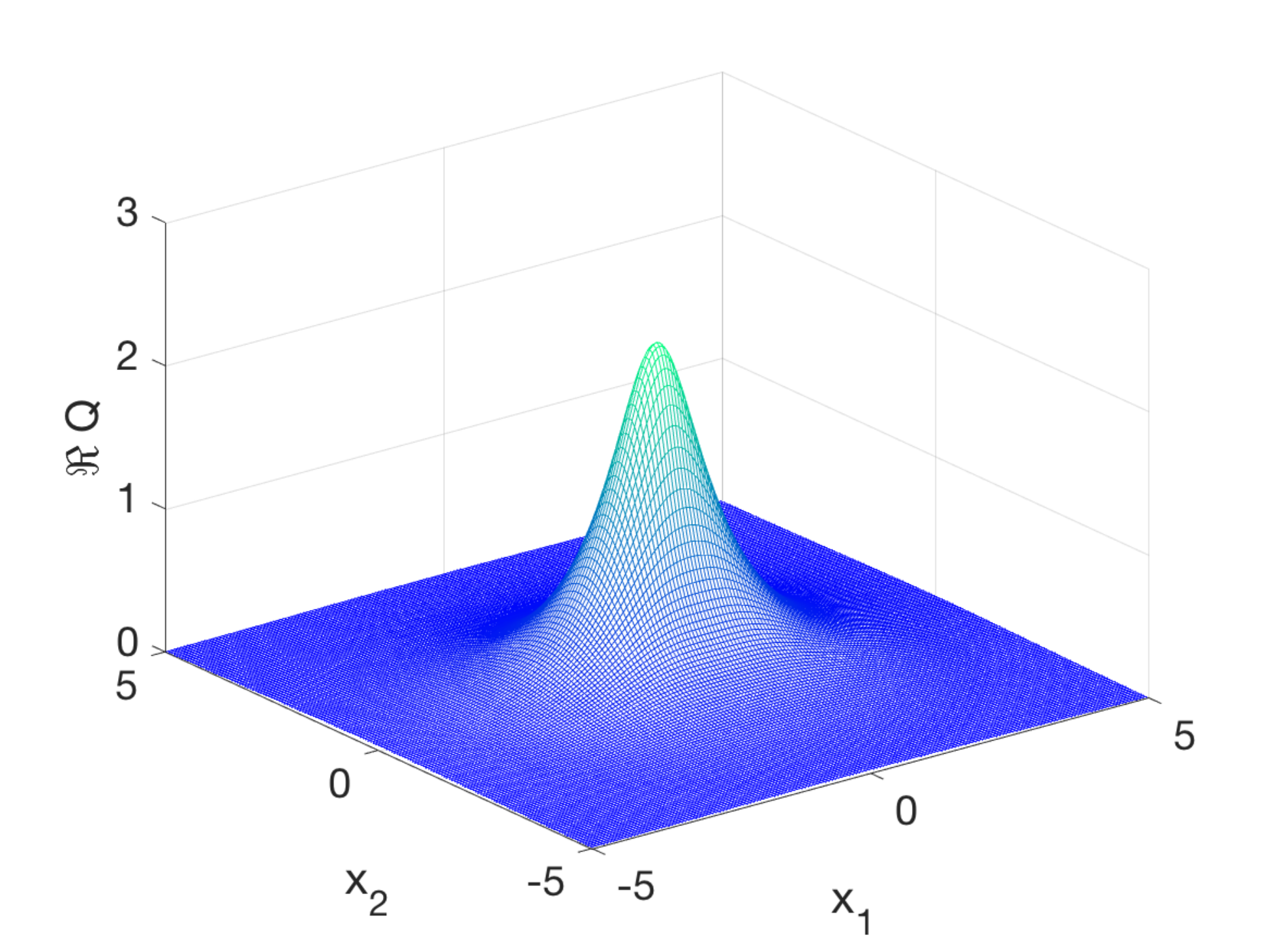}
  \includegraphics[width=0.32\textwidth]{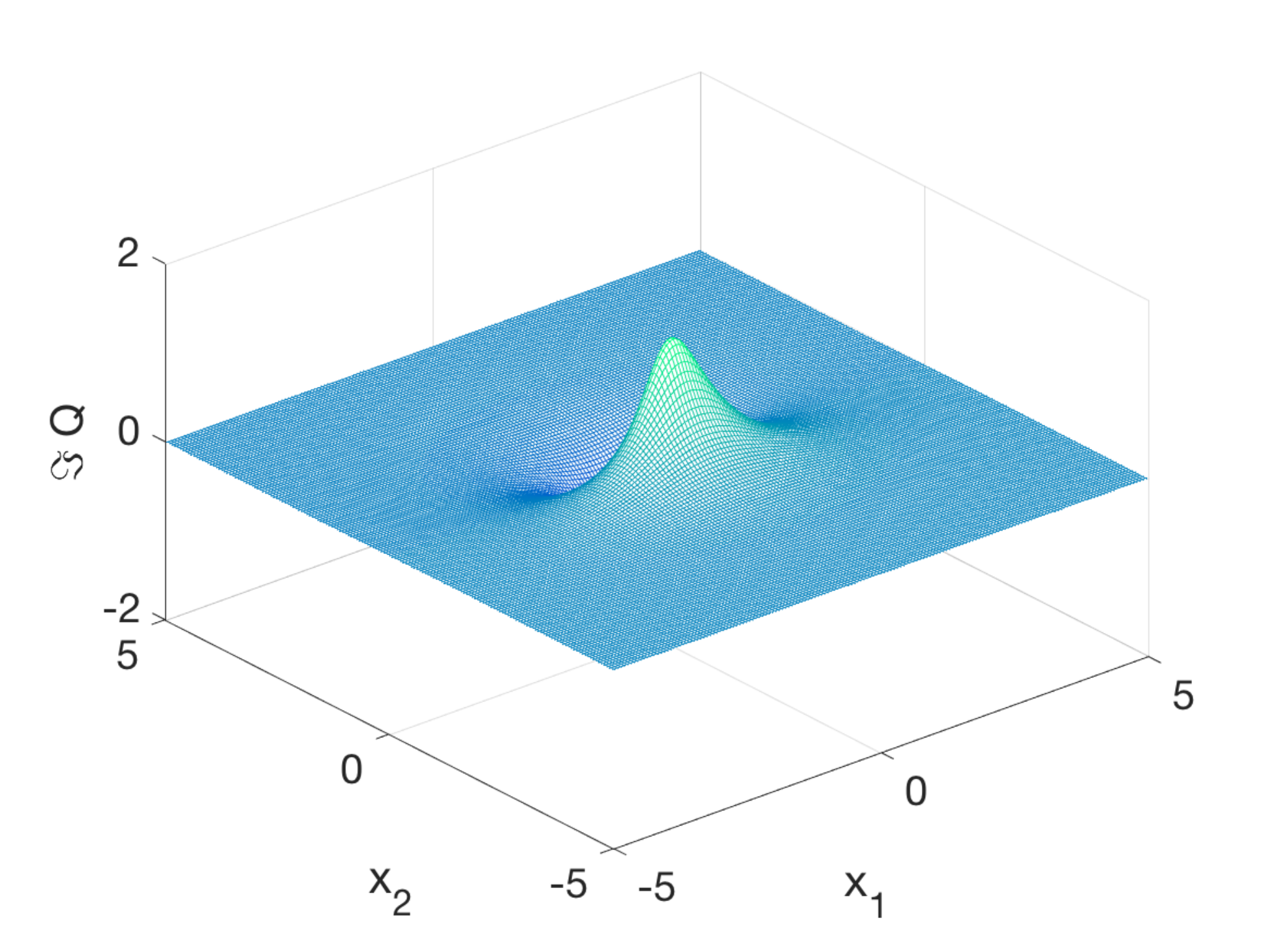}
  \includegraphics[width=0.32\textwidth]{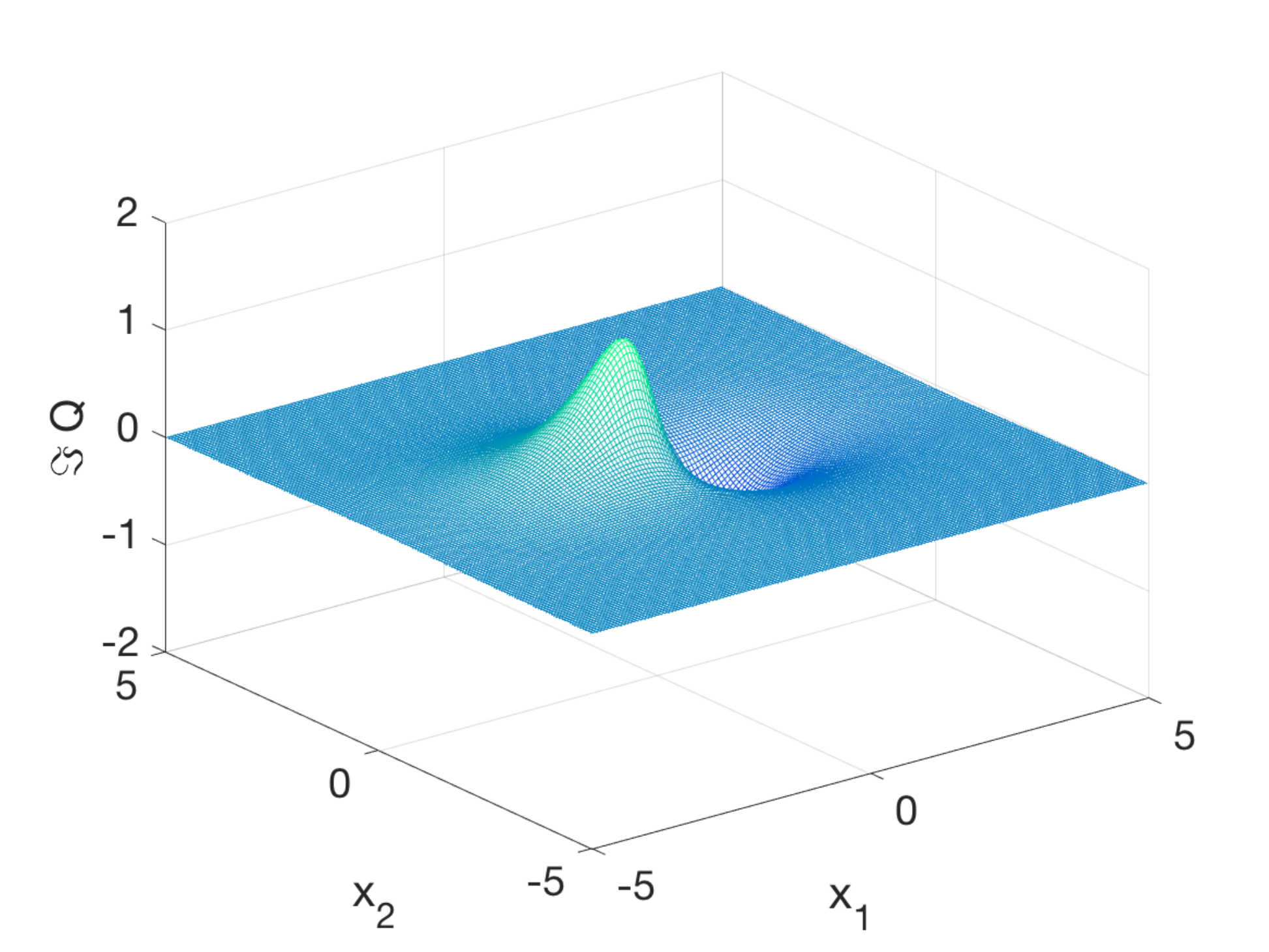}
  \includegraphics[width=0.32\textwidth]{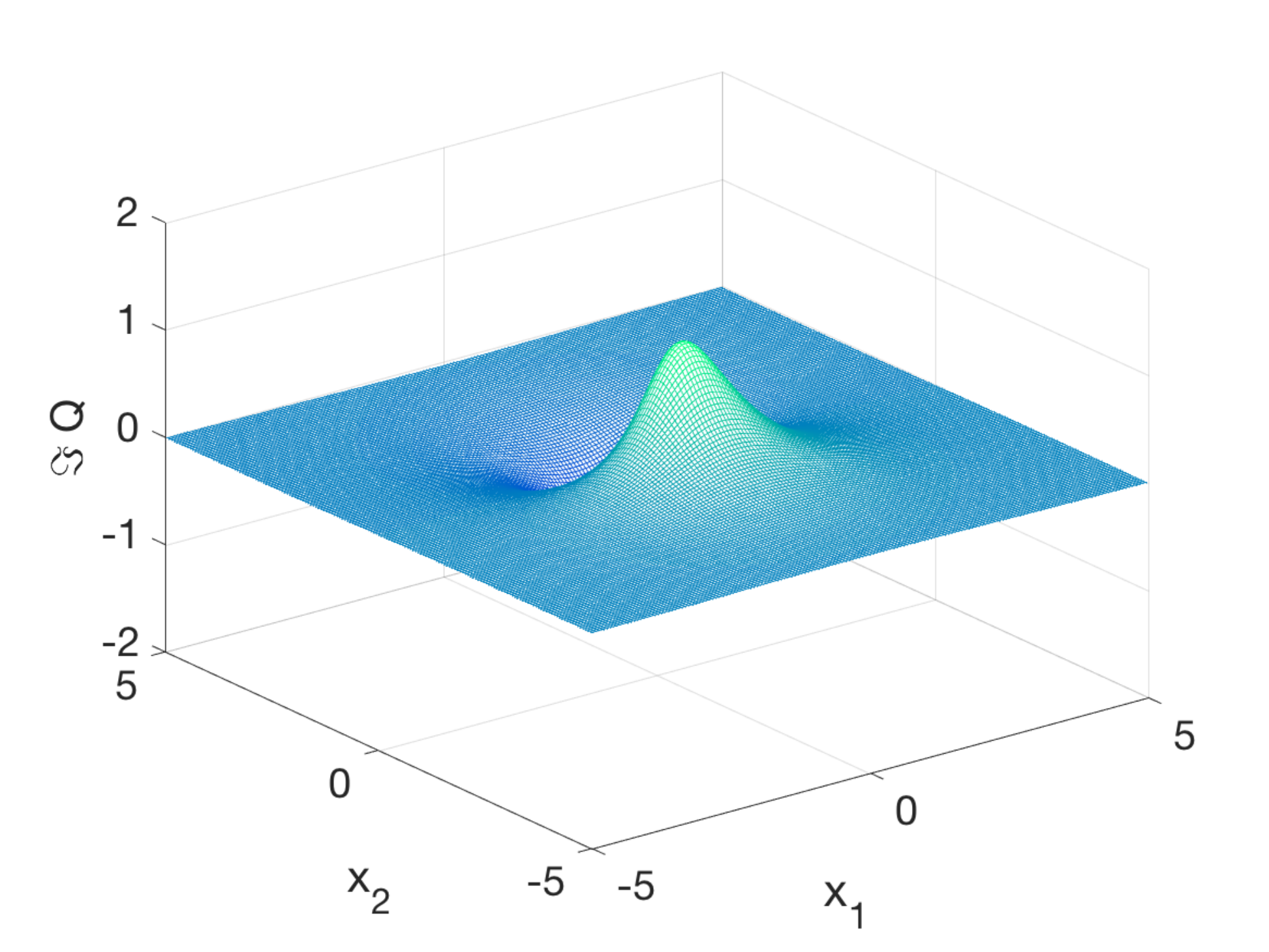}
 \caption{Real and imaginary parts of the stationary state $Q$ to equation (\ref{PDNLS}) with $\sigma =1$: On the 
 left for $\eps=1$, $k=1$, $\delta_{1}=0$ and 
$\delta_{2}=1$, in the middle for $\eps=1$, $k=1$, $\delta_{1}=1$ and 
$\delta_{2}=0$, and on the right for $\eps=1$, $k=2$, $\delta_{1}=0$ and 
$\delta_{2}=1$.}
 \label{NLS2dsold1e1}
\end{figure}

{\it Step 4:} For $\sigma >1$ stationary states become increasingly peaked, as is seen from the 1D picture in Figure \ref{fig1d}. 
Hence, to construct stationary states for higher nonlinear powers in 2D, we will consequently require more Fourier coefficients 
to effectively resolve these solutions. To this end, we work on the numerical domain \eqref{domain} with $L_{x_{1}}=L_{x_{2}}=3$ and 
$N_{x_{1}}=N_{x_{2}}=2^{10}$ 
Fourier modes. We use the ground state obtained for $\sigma=1$ as an initial iterate for the case $\sigma=2,3,$ and follow the same program as outlined above.


\section{Numerical method for the time evolution} \label{sec:numtime}

\subsection{A Fourier spectral method}

In this section, we briefly describe the numerical algorithm used to integrate our model equation in its evolutionary form \eqref{PDNLSevol}. After a Fourier transformation, 
this equation becomes
\begin{equation*}\label{PDNLSevolF}
	\partial_{t}\widehat{u}= -i\widehat{P}_{\eps}^{-1}(\xi)\Big(|\xi|^2 \widehat{u} - (1-\mbox{\boldmath ${\delta}$} \cdot \xi)\widehat{(|u|^{2\sigma}u)}\Big), \quad \xi \in \R^2.
\end{equation*}
Approximating the above by a discrete Fourier transform (via FFT) on a computational domain $\Omega$ given by \eqref{domain}, 
yields a finite dimensional system of ordinary differential equations, which formally reads
\begin{equation}\label{LN}
    \partial_{t}\widehat{u}= {L}_{\eps}\widehat{u}+{N}_{\eps}(\widehat{u}).
\end{equation}
Here ${L}_{\eps}=-iP_{\eps}^{-1}|\xi|^{2}
$ is a linear, diagonal operator in Fourier space, and 
${N}_{\eps}(\widehat{u})$ has a nonlinear and nonlocal dependence on 
$\widehat{u}$. Since $\|L_\eps\|$ can be large, equation \eqref{LN} belongs to a 
family of {\it stiff} ODEs, for which several efficient numerical schemes have been developed, cf. \cite{etna, KR11} where the particular situation of semi-classical NLS is considered. 
Driscoll's composite Runge-Kutta (RK) method \cite{Dri} has proven to be particularly efficient and thus will also be applied in the present work. 
This method uses a stiffly stable third order RK method for the high wave numbers of ${L}_{\eps}$ and combines it with a 
standard explicit fourth order RK method for the low wave numbers of ${L}_{\eps}$ and the nonlinear part ${N}_{\eps}(\widehat{u})$. 
Despite combining a third order and a fourth order method, this approach yields fourth order in-time convergence in many applications. 
Moreover, it provides an explicit method with much larger time steps 
than allowed by the usual fourth order stability conditions in 
stiff regimes. 
  
\begin{remark} The evolutionary form of our model \eqref{PDNLSevol} is in many aspects similar to the well-known 
Davey-Stewartson (DS) system, which is a non-local NLS type equation in two spatial dimensions, cf. \cite{DS, Sul}.
In \cite{KR11, KSDS, KN}, the possibility of self-similar blow-up in DS is studied, using a numerical approach similar to ours.  
\end{remark} 

As a first basic test of consistency, we apply our numerical code to the cubic NLS in 2D i.e. equation \eqref{NLS} with $\sigma =1$.  
As initial data $u_0$ we take the ground state $Q$, obtained numerically as outlined in Section \ref{sec:ground} above.
We use $N_{t}=1000$ time-steps for times $0\le t\le 1$. In this case, we know that the exact time-dependent solution $u$ is simply given by $u=Qe^{it}$. 
Comparing this to the numerical solution obtained at $t=1$ yields an 
$L^\infty$-difference of the order of $10^{-10}$.  
This verifies both the code for the time evolution and the one for the ground state $Q$ which in itself is obtained with an accuracy of order $10^{-10}$.  
Thus, the time evolution algorithm evolves the ground state with the same precision as with which it is known.

For general initial data $u_0$, we shall control the accuracy of our code in two ways: On the one hand, the resolution in space is 
controlled via the decrease of the Fourier coefficients within (the finite approximation of) $\widehat u$. The coefficients of the highest wave-numbers thereby indicate the order of magnitude of the 
numerical error made in approximating the function via a truncated Fourier series. On the other hand, the quality of the time-integration is controlled via the conserved quantity $M_\eps(t)$ defined in (\ref{CLM}). Due to unavoidable numerical errors, the latter will numerically depend on time. 
For sufficient spatial resolution, the relative conservation of $M_\eps(t)$ will overestimate the accuracy in the time-integration by $1-2$ orders 
of magnitude.

\subsection{Reproducing known results for the classical NLS} As already discussed in the introduction of this paper, the cubic NLS in two spatial dimensions is $L^2$-critical 
and its ground state solution $Q$ is strongly unstable. Indeed, any perturbation of $Q$ which 
lowers the $L^2$-norm of the initial data below that of $Q$ itself, is known to produce purely dispersive, global-in-time solutions which 
behave like the free time evolution for large $|t| \gg 1$. 
However, perturbations that increase the $L^2$-norm of the initial data above that of $Q$ are expected to generically produce a (self-similar) blow-up in finite time. This behavior can be reproduced in our simulations. 

To do so, we first take initial data of the form
\begin{equation}\label{ini1}
u_0(x_1,x_2)=Q(x_1,x_2)-0.1e^{-x_1^{2}-x_2^{2}},
\end{equation}
and work on the numerical domain $\Omega$ given by \eqref{domain} with $L_{x_{1}}=L_{x_{2}}=3$. We will use $N_{t}=5000$ time-steps 
within $0\le t\le 5$. We can see on the right of Fig.~\ref{NLS2dsol-} that the 
$L^{\infty}$-norm of the solution decreases monotonically, indicating purely dispersive behavior. 
The plotted absolute value of the solution at $t=5$ confirms this behavior. In addition, the mass
$M(t) \equiv M_0(t)$ is conserved to better than $10^{-13}$, indicating that the problem is indeed well resolved in time. 
\begin{figure}[htb!]
  \includegraphics[width=0.49\textwidth]{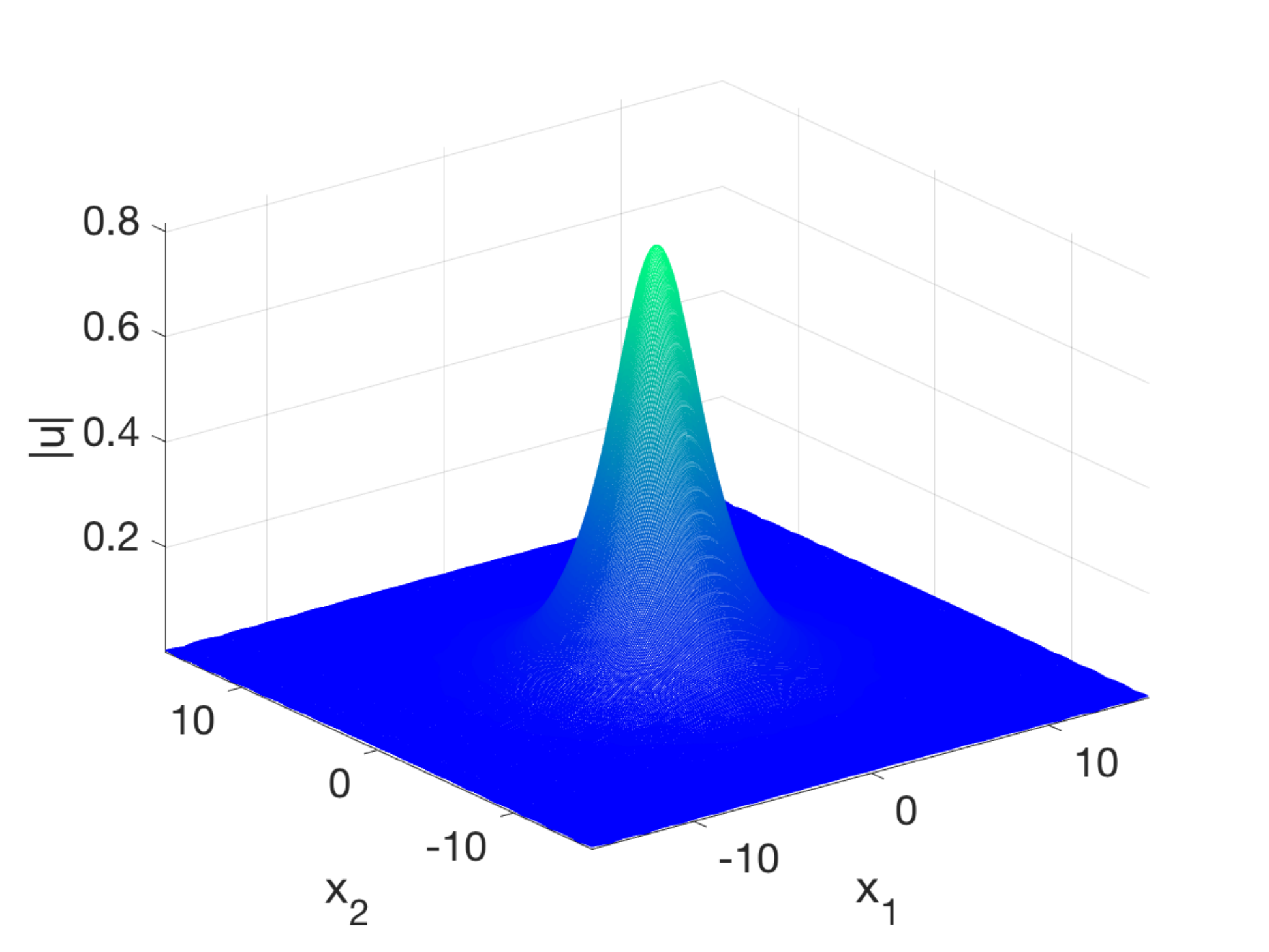}
  \includegraphics[width=0.49\textwidth]{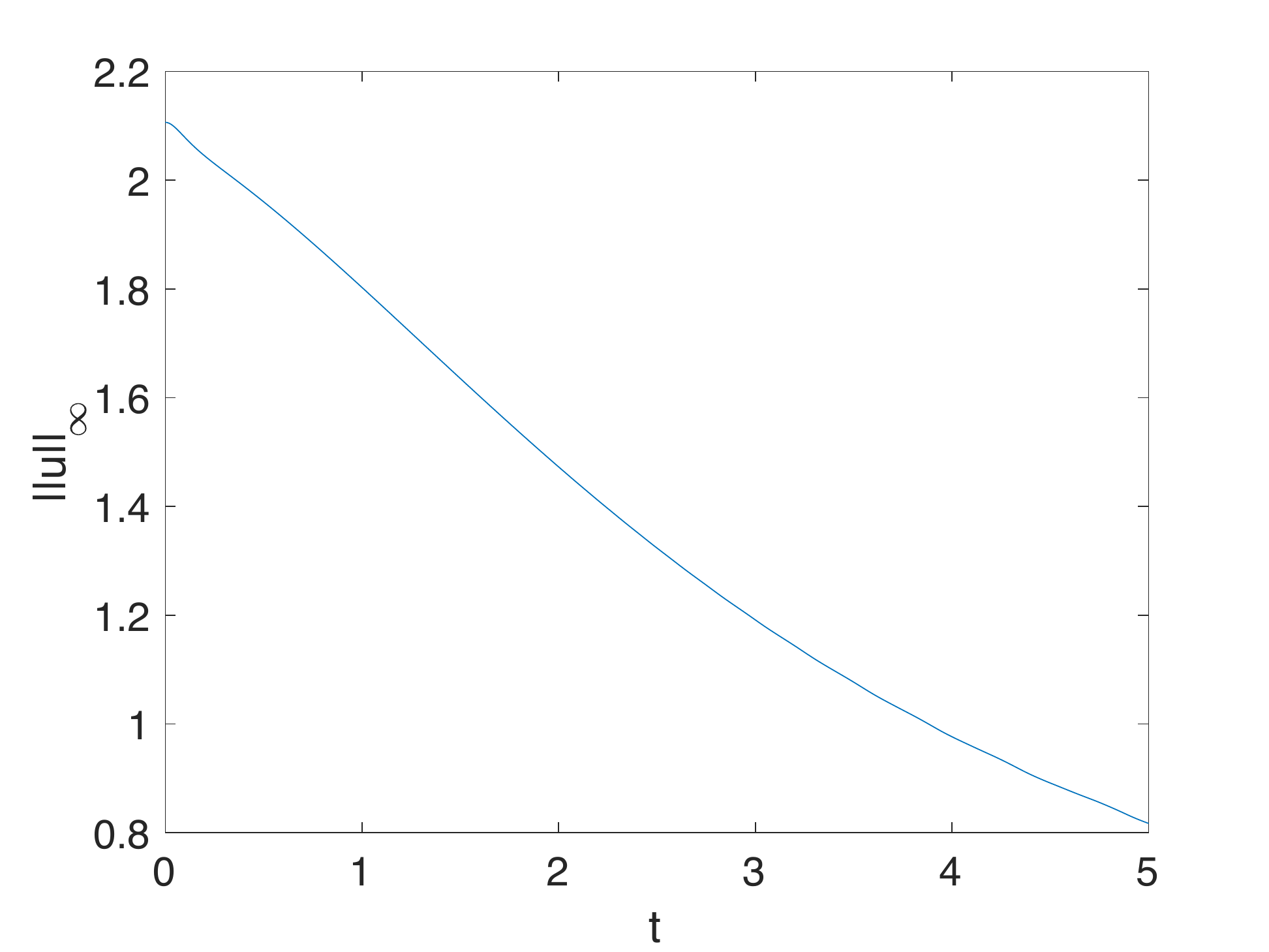}
 \caption{Solution to the classical NLS (\ref{NLS}) with 
 $\sigma=1$ and initial data \eqref{ini1}: on the left $|u|$ at $t=5$, and on the right the $L^{\infty}$-norm of the solution as a function of $t$.}
 \label{NLS2dsol-}
\end{figure}

\begin{remark}
Note that we effectively run our simulations on $\Omega\simeq\mathbb{T}^{2}$, instead of $\mathbb{R}^{2}$. As a consequence, the periodicity will 
after some time induce radiation effects appearing on the opposite side of $\Omega$. 
The treatment of (large) times $t>5$ therefore requires a larger 
computational domain to suppress these unwanted effects.
\end{remark}

Next, for initial data of the form 
\begin{equation}\label{ini2}
u_0(x_1,x_2)=Q(x_1,x_2)+0.1e^{-x_1^{2}-x_2^{2}},
\end{equation}
we again use $N_{t}=5000$ time steps for $0\le t\le 2$. As can be seen in Fig.~\ref{NLS2dsol+} on the right, there is numerical indication for  finite-time blow-up. 
The code is stopped at $t=1.89$ when the relative error in the conservation of mass $M(t)$ drops below $10^{-3}$. The solution for $t=1.88$ can be seen on the 
left of Fig.~\ref{NLS2dsol+}. This is in accordance with the self-similar blow-up established by Merle and 
Rapha\"el, cf. \cite{MR1, MR2}. In particular, we note that the 
result does not change notably if a higher resolution in both $x$ and $t$ is used. 
\begin{figure}[htb!]
  \includegraphics[width=0.49\textwidth]{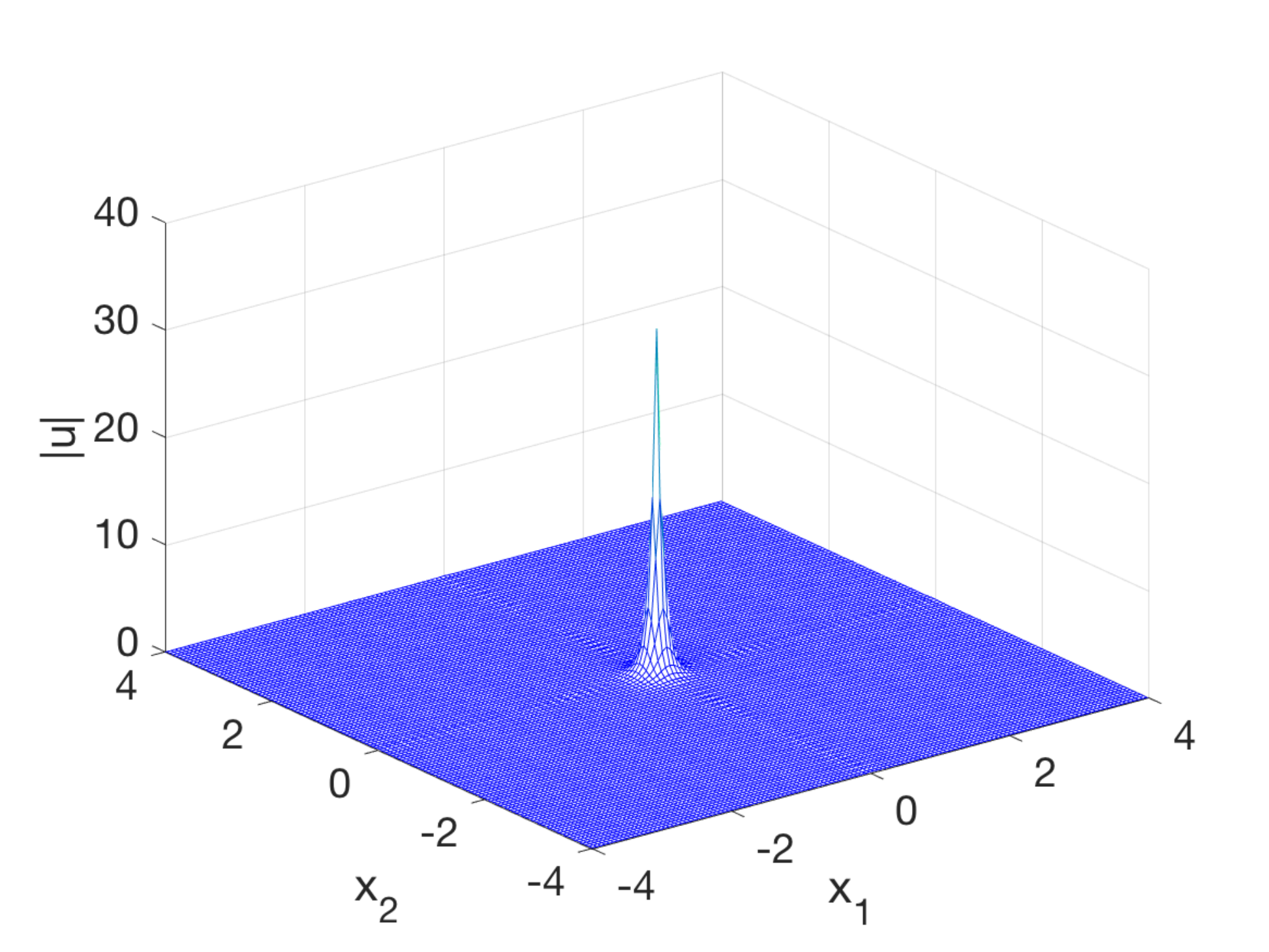}
  \includegraphics[width=0.49\textwidth]{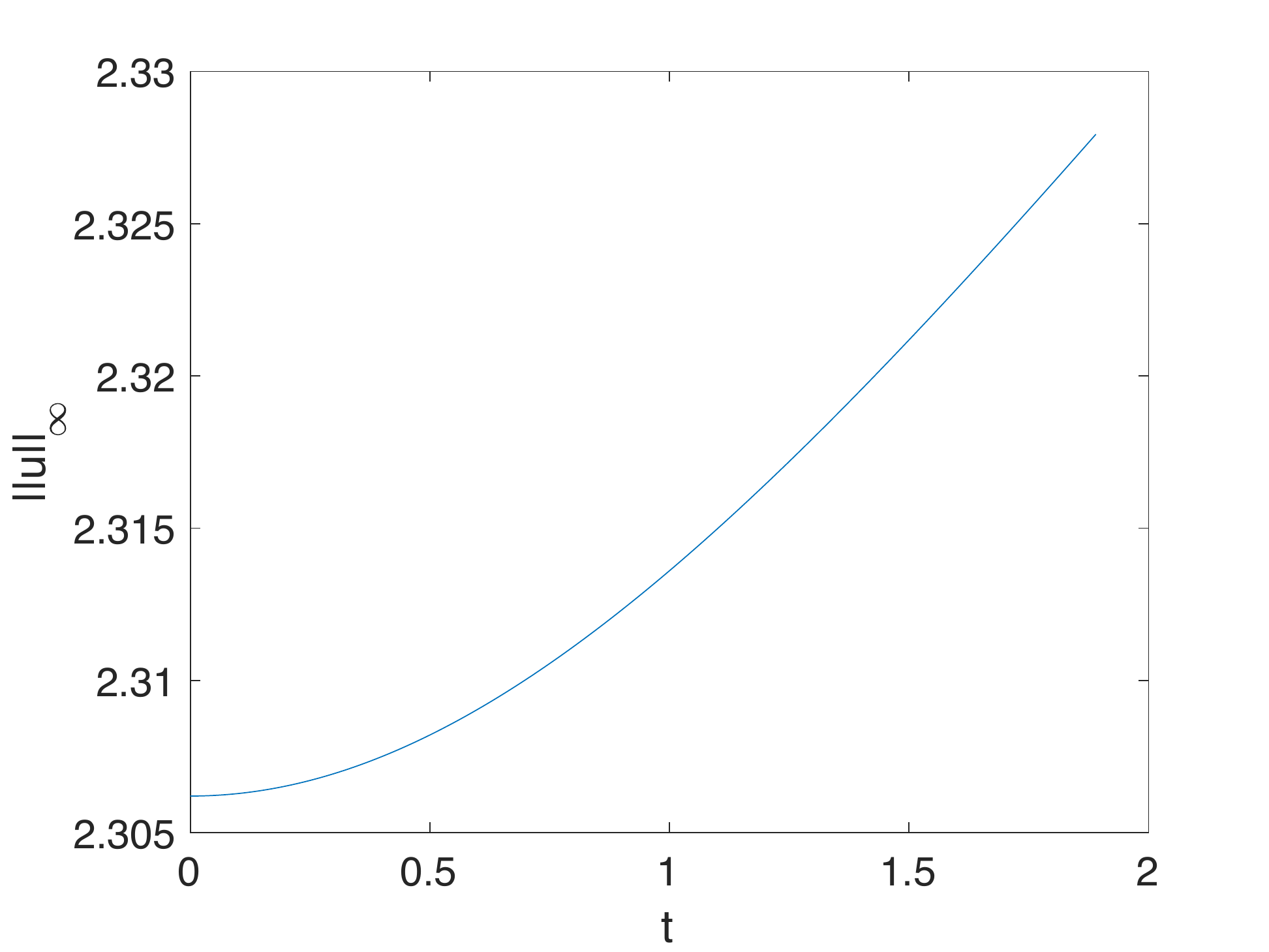}
 \caption{Solution to the classical NLS (\ref{NLS}) with  
 $\sigma=1$ and initial data \eqref{ini2}: on the left $|u|$ at $t=1.88$ and on the right the $L^{\infty}$-norm of the solution as a function of $t$.}
 \label{NLS2dsol+}
\end{figure}

\begin{remark} We want to point out that there are certainly more sophisticated methods available to numerically study self-similar blow-up, see for instance  
\cite{KMS, LSS2, Sul} for the case of NLS type models, as well as \cite{KP14, KP16} for the analogous problem in KdV type equations. 
However, these methods will not be useful for the present work, since as noted before, the model (\ref{PDNLS}) 
does not admit a simple scaling invariance, which is the underlying reason for self-similar blow-up in NLS and KdV type models. 
As a result, all our numerical findings concerning finite-time blow-up have to be taken with a grain of salt. An apparent divergence of certain norms of the solution or overflow errors 
produced by the code can indicate a blow-up, but might also just indicate that one has run out of resolution. 
The results reported in this paper therefore need to be understood as being stated with respect to the given numerical resolution. 
However, we have checked that they remain stable under changes of the resolution within the accessible limits of the computers used to run the simulations.
\end{remark}

\subsection{Time-dependent change of variables in the case with self-steepening} 
In the case of self-steepening, the ability to produce an accurate numerical time-integration in the presence of a derivative nonlinearity ($\mbox{\boldmath $\delta$}\neq 0$) becomes 
slightly more complicated. The inclusion of such a nonlinearity can lead to localized initial data moving (relatively fast) in the direction chosen by $\mbox{\boldmath $\delta$}$. 
In turn, this might cause the 
numerical solution to ``hit" the boundary of our computational domain $\Omega$. 

To avoid this issue, we shall instead perform our numerical computations in a 
moving reference frame, chosen such that the maximum of $|u(t,x)|$ remains fixed at the origin. More precisely, 
we consider the transformation 
$$ x\mapsto x- y(t),$$
and denote $\mathbf{v}(t)= \dot y(t)$.  The new unknown $u(t,x-y(t))$ solves
\begin{equation}\label{PDNLSevolcm}
i \partial_{t}u - i\mathbf{v}\cdot\nabla u+ P^{-1}_{\eps}\Delta u + P_\eps^{-1}(1+i 
\mbox{\boldmath $\delta$} \cdot \nabla )\big(|u|^{2\sigma}u\big) =0.
\end{equation}
The quantity $\mathbf{v}(t)=(v_1(t), v_2(t))$ is then determined by the condition that 
the density $\rho=|u|^{2}$ has a maximum at $(x_1,x_2)=(0,0)$ for all $t\ge 0$. We get from \eqref{PDNLSevolcm} the following equation for $\rho$:
\begin{align*}
    \partial_{t}\rho= &\, \mathbf{v}\cdot\nabla \rho + i\big(\bar{u}P^{-1}_{\varepsilon}\Delta u - uP^{-1}_{\varepsilon} \Delta \bar{u}\big)+i\big(\bar{u}P^{-1}_{\varepsilon}
    (\rho^{\sigma}u) - uP^{-1}_{\varepsilon} (\rho^{\sigma} 
    \bar{u})\big)\nonumber\\
    &\, -\bar{u}P^{-1}_{\varepsilon}\mbox{\boldmath $\delta$} \cdot \nabla (\rho^{\sigma}u) - uP^{-1}_{\varepsilon} \mbox{\boldmath $\delta$} \cdot \nabla(\rho^{\sigma}\bar{u})
    \label{At}.
\end{align*}
Differentiating this equation with respect to $x_{1}$ and $x_{2}$ respectively, and setting $x_{1}=x_{2}=0$ yields 
the desired conditions for $v_{1}$ and $v_{2}$. 

Note that the computation of the additional derivatives 
appearing in this approach is expensive, since in practice it needs to be enforced in every step of the Runge-Kutta scheme. 
Hence, we shall restrict this approach solely to cases where the numerical results appear to be strongly affected by the boundary of $\Omega$. 
In addition, we may always choose a reference frame such that one of the two components of $\mbox{\boldmath $\delta$} $ is zero, which consequently allows us to set either $v_1$, or $v_2$ equal to zero.

\subsection{Basic numerical tests for a derivative NLS in 2D} As an example, we consider the case of a cubic nonlinear, two-dimensional derivative NLS of the following form
\begin{equation}\label{dNLS1}
i \partial_{t}u + \Delta u + (1+i \delta_2\partial_{x_2})\big(|u|^{2}u\big) =0, \quad u_{\mid t=0}=u_0(x_1,x_2). \\
\end{equation}
which is obtained from our general model \eqref{PDNLSevol} for $\eps=0$ and $\delta_1=0$. 
We take initial data $u_0$ given by \eqref{ini2}. Here, $Q$ is the ground state computed earlier for this particular choice of parameters, see Fig.~\ref{NLS2dsold1}. 
We work on the computational domain \eqref{domain} with $L_{x_{1}}=L_{x_{2}}=3$, using $N_{x_{1}}=N_{x_{2}}=2^{10}$ Fourier modes and $10^{5}$ time-steps for $0\le t\le 5$. 
We also apply a Krasny filter \cite{krasny}, which sets all Fourier coefficients smaller than $10^{-10}$ equal to zero. 
For $\delta_2=1$ the real and imaginary part of the solution $u$ at the final time $t=5$ can be seen in Fig.~\ref{nls2dsold1+} below. 
Note that they are both much more localized and peaked when compared to the ground state $Q$ shown in Fig.~\ref{NLS2dsold1}, indicating a self-focusing behavior within $u$. 
Moreover, the real part of $u$ is no longer positive due to phase modulations. 
\begin{figure}[htb!]
  \includegraphics[width=0.49\textwidth]{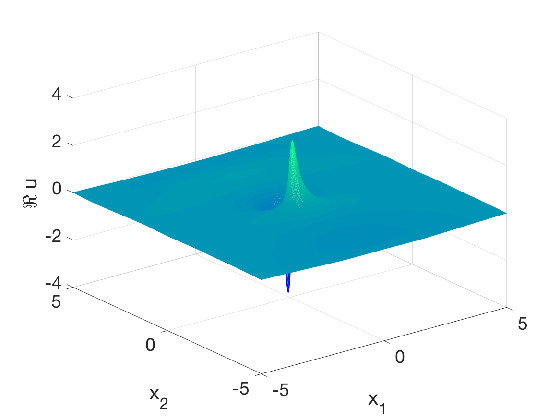}
  \includegraphics[width=0.49\textwidth]{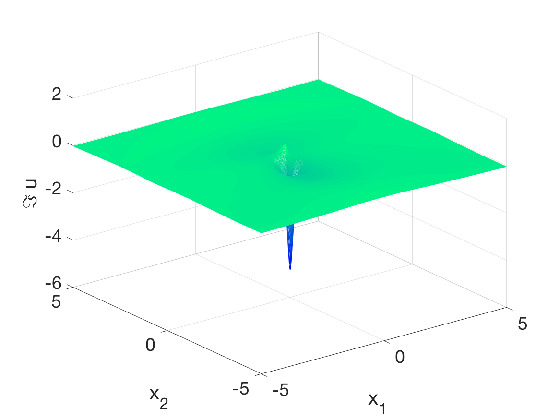}
 \caption{Real and imaginary part of the solution to (\ref{dNLS1}) with  $\delta_2=1$ at time $t=5$ corresponding to 
 $u_{0}=Q+0.1e^{-x_{1}^{2}-x_{2}^{2}}$, where $Q$ is the stationary state in Fig.~\ref{NLS2dsold1}.}
 \label{nls2dsold1+}
\end{figure}
 
Surprisingly, however, there is {\it no} indication of a finite-time blow-up, in contrast to the analogous situation without derivative nonlinearity (recall Fig.~\ref{NLS2dsol+} above). 
Indeed, the Fourier coefficients of $|u|$ at $t=5$ are seen in Fig.~\ref{nls2dsold1+fourier} to
decrease to the order of the Krasny filter. 
In addition, the $L^{\infty}$-norm of the solution, plotted in the middle of the same figure, appears to exhibit a turning point shortly before $t\approx4$. 
Finally, the velocity component $v_{2}$ plotted on the right in Fig.~\ref{nls2dsold1+fourier} seems to slowly converge to a some limiting value $v_2\approx 2$. 
The latter would indicate the appearance of a stable moving soliton, but it is difficult to decide such questions numerically. 
All of these numerical findings are obtained with $M_\eps(t)$ conserved up to errors of the order $10^{-11}$.
\begin{figure}[htb!]
  \includegraphics[width=0.32\textwidth]{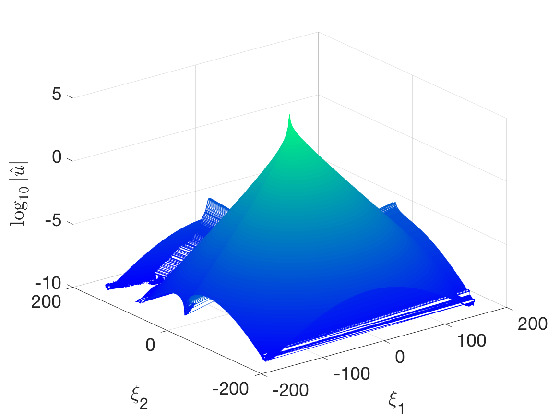}
  \includegraphics[width=0.32\textwidth]{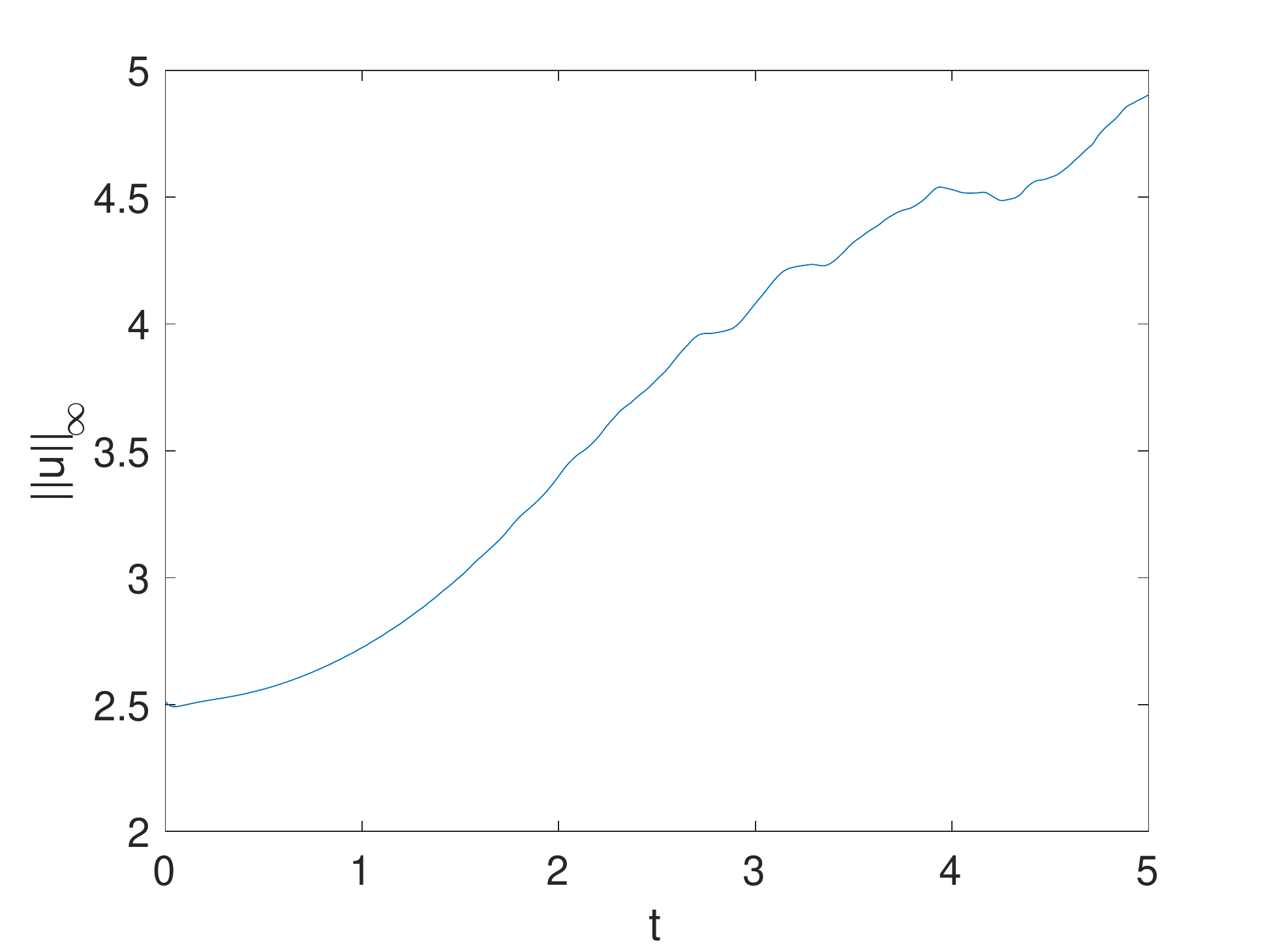}
  \includegraphics[width=0.32\textwidth]{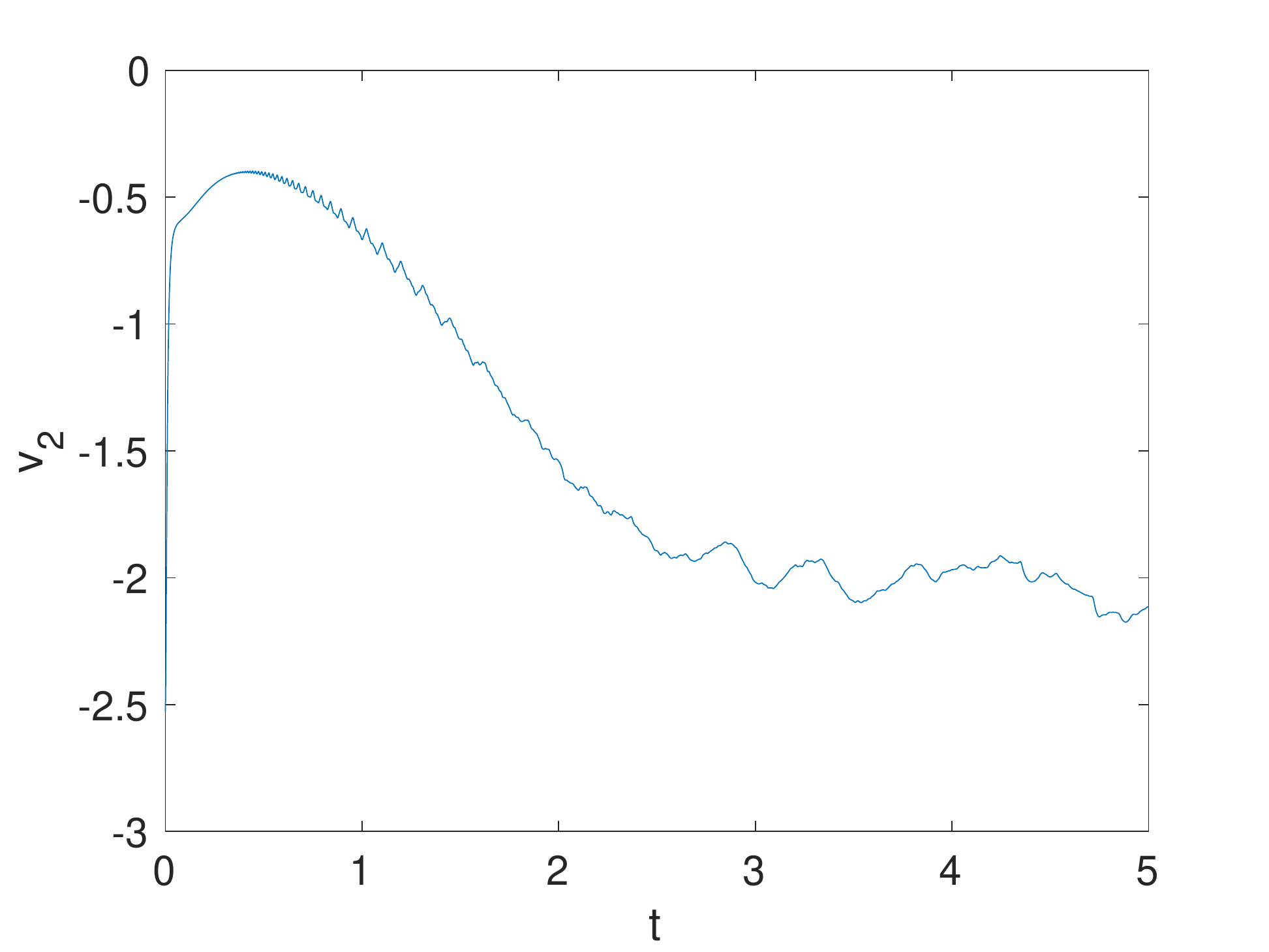}
 \caption{Solution to (\ref{dNLS1}) with $\delta_2=1$ and perturbed stationary state initial data: 
 The Fourier coefficients of $|u|$ at $t=5$ on the left; the $L^{\infty}$-norm of the solution as a function of time 
 in the middle, and the time evolution of its velocity $v_{2}$ on the right.}
 \label{nls2dsold1+fourier}
\end{figure}

It might seem extremely surprising that the addition of a derivative nonlinearity is able to suppress the appearance of finite-time blow-up. Note however, that in 
all the examples above we have used only (a special case of) perturbed ground states $Q$ as initial data. 
For more general initial data, the situation is radically different, as can be illustrated numerically in the following example: 
We solve \eqref{dNLS1} with purely Gaussian initial data of the form
\begin{equation}\label{gaussian}
u_{0}(x_1, x_2)=4e^{-(x_{1}^{2}+x_{2}^{2})}
\end{equation}
on a numerical domain $\Omega$ with $L_{x_{1}}=L_{x_{2}}=2$, using $N_{x_{1}}=N_{x_{2}}=2^{10}$ Fourier coefficients and $N_{t}=10^{5}$ time steps for 
$0\le t\le 0.25$. 
This case appears to exhibit finite-time blow-up, as is illustrated in Fig.~\ref{nls2d4gaussd1}. The conservation of the numerically computed quantity 
$M(t)$ drops below $10^{-3}$ at $t\approx T=0.1955$ which indicates that plotting accuracy is no longer guaranteed. 
Consequently we ignore data taken for later times, but note that the code stops with an overflow error for $t\approx 0.202$. 
\begin{figure}[htb!]
  \includegraphics[width=0.49\textwidth]{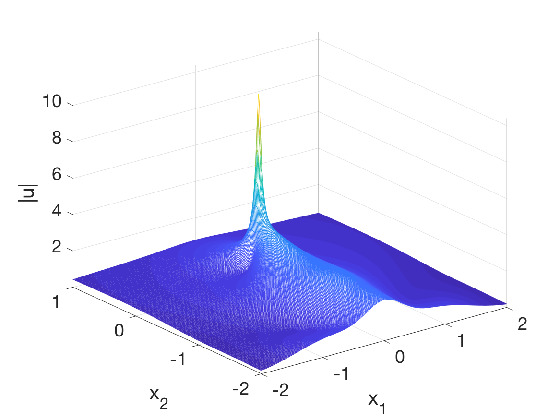}
  \includegraphics[width=0.49\textwidth]{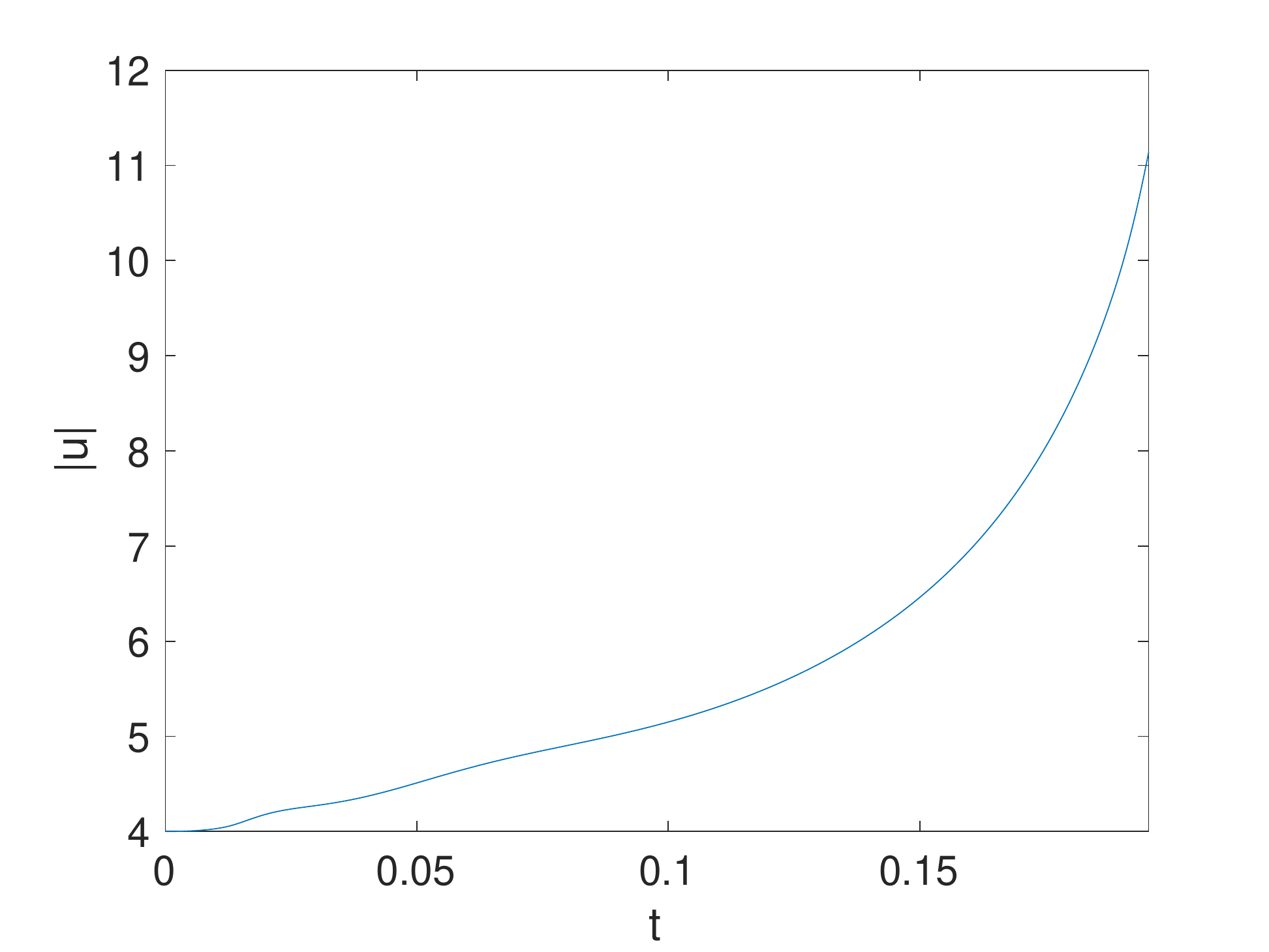}
 \caption{The modulus of the solution to (\ref{dNLS1}) with $\delta_2=1$ for Gaussian initial data $u_0=4e^{-x_{1}^{2}-x_{2}^{2}}$, at time
 $t=0.195$. On the right, the  $L^{\infty}$-norm of the solution as a function of time.}
 \label{nls2d4gaussd1}
\end{figure}

\begin{remark}
These numerical findings are consistent with analytical results for derivative NLS in one spatial dimension. For 
certain values of $\sigma\ge 1$ and certain velosit $v$, the corresponding solitary wave solutions
are found to be orbitally stable, see \cite{CO, GNW, LSS1}. However, for general initial data and $\sigma>1$ large enough, one expects finite-time blow-up, see \cite{LSS2}.
\end{remark}


\section{Global well-posedness with full off-axis variation}\label{sec:full} 

In this section we will analyze the Cauchy problem corresponding to \eqref{PDNLSevol} in the case of full off-axis dependence, i.e. $k=2$, so that
\begin{equation*}\label{Pfull}
P_\eps = 1-\eps^2 \Delta.
\end{equation*}
In this context, we expect the solution $u$ of \eqref{PDNLSevol} to be very well behaved due to the strong regularizing 
effect of the elliptic operator $P_\eps$ acting in both spatial directions. 

To prove a global-in-time existence result, we rewrite \eqref{PDNLSevol}, using Duhamel's formula,
\begin{equation}\label{intrepu}
u(t) = S_\eps(t) u_0 + i \int_{0}^{t}S_{\eps}(t-s)P_{\eps}^{-1}(1+i\mbox{\boldmath  $\delta$}\cdot\nabla)\big(|u|^{2\sigma}u\big)(s)\;ds\equiv \Phi(u)(t).
\end{equation}
Here, and in the following, we denote by $$S_{\eps}(t) = e^{itP_{\eps}^{-1}\Delta}$$ the corresponding linear propagator, which is easily seen (via Plancherel's theorem) to be 
an isometry on $H^s$ for any $s\in \R$. 
It is known that in the case with full off-axis variation, $S_\eps(t)$ does not allow for any Strichartz estimates, see \cite{Car}. 
However, the action of $P_\eps^{-1}$ allows us to ``gain" two derivatives and offset the action of the gradient term in the nonlinearity of \eqref{intrepu}.
Using a fixed point argument, we can therefore prove the following result. 

\begin{theorem}[Full off-axis variations]\label{thm:fullHs} 
Let $\eps>0$, $k=2$ and $\sigma>\frac{1}{4}$. Then for any $\mbox{\boldmath $\delta$} \in \R^{2}$ and any $u_{0} \in H^{1}(\R_{x}^2)$, there exists a 
unique global-in-time solution $u \in C(\R_{t};H^{1}(\R_{x}^2))$ to \eqref{PDNLSevol}, depending continuously on the initial data. Moreover, 
\[
\| u(t, \cdot) \|_{H_{x}^1}\le C(\eps,\|u_{0}\|_{H_{x}^{1}}), \quad \forall \, t\in \R.
\]
\end{theorem}

\begin{proof}
Let $T, M>0$. We aim to show that $u \mapsto \Phi(u)$ is a contraction on the ball
$$
X_{T,M} = \{ u \in  L^{\infty}([0,T);H^{1}(\R_{x}^{2})) : 
\|u \|_{L_{t}^{\infty}H_{x}^{1}} \le M \}.
$$
To this end, let us shortly denote 
\begin{equation}\label{intrepu1}
\Phi(u)(t) = S_{\eps}(t)u_{0} +\mathcal{N}(u)(t),
\end{equation}
where for $g(u) = |u|^{2\sigma}u$, we write
$$
\mathcal{N}(u)(t) := i\int_{0}^{t}S_{\eps}(t-s)P_{\eps}^{-1}(1+i\mbox{\boldmath  $\delta$}\cdot\nabla)g(u(s))\;ds.
$$

Now, let $u, u'\in  X_{T,M}$. Using Minkowski's inequality and recalling that $S_\eps(t)$ is an isometry on $H^1(\R^2)$ yields
\begin{align*}
\big\|\big(\mathcal{N}(u)(t) - \mathcal{N}(u')(t)\big)\big\|_{H_{x}^{1}}
\le \eps^{-2}(1+ |\bd|)\int_{0}^{t} \|g(u)-g(u')\|_{L_{x}^{2}}(s)\;ds.  
\end{align*}
To bound the integrand, we first note that
\begin{equation}\label{g1}
|g(u)-g(u')| \le C_{\sigma}(|u|^{2\sigma} + |u'|^{2\sigma})|u-u'|.
\end{equation}
If we impose $\sigma > \frac{1}{4} $, then we have by Sobolev's embedding that
$$  H^{1}(\R^{2}) \subset H^{\frac{(4\sigma-1)}{4\sigma}}(\R^{2}) \hookrightarrow L^{8\sigma}(\R^{2}) \quad \text{and}  \quad  H^{\frac{1}{2}}(\R^{2}) \hookrightarrow L^{4}(\R^{2}).$$ 
This allows us to estimate further after using \eqref{g1} and H\"older's inequality in space to give
\begin{align*}
\|g(u)-g(u')\|_{L_{x}^{2}} \le &\, \big(\| u \|^{2\sigma}_{L_{x}^{8\sigma}} + \| u' \|^{2\sigma}_{L_{x}^{8\sigma}} \big)\| u-u' \|_{L_{x}^{4}}\\
\le &\, \Big(\| u \|^{2\sigma}_{H_{x}^{1}} + \| u' \|^{2\sigma}_{H_{x}^{1}} \Big)\| u-u' \|_{H_{x}^{1}}.
\end{align*}
Together with H\"older's inequality in $t$, we can consequently bound
\begin{equation*}
\big\|\mathcal{N}(u) - \mathcal{N}(u')\big\|_{L_{t}^{\infty}H_{x}^{1}}\le 2\eps^{-2}( 1 + |\bd|)TM^{2\sigma}\| u-u' \|_{L_{t}^{\infty}H_{x}^{1}}.
\end{equation*}
By choosing $T>0$ sufficiently small, Banach's fixed point theorem directly yields a unique
local-in-time solution $u \in C([0,T],H^{1}(\R_{x}^{2}))$. Standard arguments (see, e.g., \cite{Paz}) 
then allow us to extend this solution up to a maximal time of existence $T_{\rm max}=T_{\rm max}(\| u_0\|_{H_{x}^1})>0$ and we also infer continuous dependence on the initial data.

Next, we shall prove that
\begin{equation}\label{conser}
\|P_{\eps}^{1/2}u(t)\|_{L_{x}^{2}} = \|P_{\eps}^{1/2}u_{0}\|_{L_{x}^{2}}, \quad \text{for all $t \in [0,T]$ and $T< T_{\rm max}$.}
\end{equation}
For $\eps>0$, this conservation law yields a uniform bound on the $H^{1}$-norm of $u$, since 
\[
c_\eps \|P_{\eps}^{1/2}\varphi \|_{L_{x}^{2}}\le \| \varphi \|_{H_{x}^1} \le C_\eps \|P_{\eps}^{1/2}\varphi \|_{L_{x}^{2}}, \quad C_\eps, c_\eps >0.
\]
We consequently can re-apply the fixed point argument 
as many times as we wish, thereby preserving the length of the maximal interval in each iteration, 
to yield $T_{\rm max} = +\infty$. Since the equation is time-reversible modulo complex conjugation, we obtain a global $H^{1}$-solution for all $t\in \R$, provided \eqref{conser} holds.

To prove \eqref{conser}, we adapt and (slightly) modify an elegant argument given in \cite{Oz}, which has the advantage 
that it does not require an approximation procedure via a sequence of sufficiently smooth solutions (as is classically done, see e.g. \cite{Cz}): 
Let $t \in [0,T]$ for $T < T_{\max}$. We first rewrite Duhamel's formula \eqref{intrepu}, using the continuity of the semigroup $S_\eps$ to propagate backwards in time
\begin{equation}\label{eq:N}
S_\eps(-t)u(t) =  u_{0} + S_\eps(-t)\mathcal{N}(u)(t).
\end{equation}
As  $S_{\eps}(\cdot)$ is unitary in $L^{2}$, we have $\|P^{1/2}_{\eps}u(t)\|_{L_{x}^{2}} = \|S_\eps(-t)P^{1/2}_{\eps}u(t)\|_{L_{x}^{2}}$. The latter can be expressed using the above identity:
\begin{align*}
&\|P^{1/2}_{\eps} \, u(t)\|_{L_{x}^{2}} = \\
& \ = \|P^{1/2}_{\eps}u_{0}\|_{L_{x}^{2}}
+2\re \, \big \langle S_\eps(-t)P_{\eps}^{1/2}\mathcal{N}(u)(t),P_{\eps}^{1/2}u_{0} \big\rangle_{L_{x}^{2}} + \| S_\eps(-t)P_{\eps}^{1/2}\mathcal{N}(u)(t) \|^2_{L_{x}^{2}}\\
&  \ \equiv \|P_{\eps}^{1/2}u_{0}\|_{L_{x}^{2}}  + \mathcal{I}_{1} + \mathcal{I}_{2}. 
\end{align*}
We want to show that $\mathcal{I}_{1} + \mathcal{I}_{2}=0$.
In view of \eqref{intrepu1} we can rewrite
\begin{align*}
\mathcal{I}_{1} &= -2\im \, \big\langle \int_{0}^{t}S_{\eps}(-s)P_{\eps}^{-1/2}(1+i\mbox{\boldmath  $\delta$}\cdot\nabla)g(u)(s)\;ds , P_{\eps}^{1/2}u_0  \big\rangle_{L_{x}^2} \;ds\\
&=  -2\im  \int_{0}^{t} \big\langle  (1+i\mbox{\boldmath  $\delta$}\cdot\nabla)g(u)(s) , S_{\eps}(s)u_0  \big\rangle_{L_{x}^2} \;ds.
\end{align*}
By the Cauchy-Schwarz inequality we find that this quantity is indeed finite, since 
\begin{align*}
|\mathcal{I}_{1}| \le  2T\| (1+i\mbox{\boldmath  $\delta$}\cdot\nabla)g(u)\|_{L_{t}^{\infty}L_{x}^{2}}\| S_{\eps}(\cdot)u_0 \|_{L_{t}^{\infty}L_{x}^{2}} <\infty.
\end{align*}
Denoting for simplicity $G_{\eps}(\cdot) = P_{\eps}^{-1}(1+i\mbox{\boldmath  $\delta$}\cdot\nabla)g(u)(\cdot)$, we find after a lengthy computation (see \cite{AAS} for more details) that the integral
\begin{align*}
\mathcal{I}_{2}& = 2 \re \int_{0}^{t} \big\langle  P_{\eps}G_{\eps}(s) , -i\mathcal{N}(u)(s) \big\rangle_{L_{x}^{2}}\;ds.
\end{align*}
We can express $-i\mathcal{N}(u)(s)$ using the integral formulation \eqref{eq:N} and write
\begin{equation}\label{yeta}
\mathcal{I}_{2} = 2 \re \, \Big( \int_{0}^{t} \big\langle  P_{\eps}G_{\eps}(s) , iS_{\eps}(s)u_0 \big\rangle_{L_{x}^{2}}\;ds + \int_{0}^{t} {\big\langle  P_{\eps}G_{\eps}(s) , -iu(s) \big\rangle_{L_{x}^{2}}}\;ds \Big).
\end{equation}
Next, we note that the particular form of our nonlinearity implies
\begin{align*}
 \re\,\big\langle P_{\eps}G_{\eps},  -iu \big\rangle_{L^{2}} 
 & = \im \, \big\langle (1+i\mbox{\boldmath  $\delta$}\cdot\nabla)g(u),  u \big\rangle_{L_{x}^{2}} \\
&= \im \, \|u\|^{2\sigma +2}_{L_{x}^{2\sigma +2}} - \re \, \big\langle g(u),  (\mbox{\boldmath  $\delta$}\cdot\nabla)u \big\rangle_{L_{x}^{2}} .  
\end{align*}
Here, the first expression in the last line is obviously zero, whereas for the second term we compute
\begin{align*}
\re \, \big\langle g(u),  (\mbox{\boldmath  $\delta$}\cdot\nabla)u \big\rangle_{L_{x}^{2}} = &\, \int_{\R^{2}}|u|^{2\sigma}\re\big(u(\mbox{\boldmath  $\delta$}\cdot\nabla)\overline{u}\big) \;dx \\
= &\, \frac{1}{2(\sigma+1)}\int_{\R^{2}}(\mbox{\boldmath  $\delta$}\cdot\nabla)(|u|^{2\sigma+2} )\;dx =0,
\end{align*}
for $H^1$-solutions $u$. In summary, the second term on the right-hand side of \eqref{yeta} simply vanishes and we find
\begin{align*}
\mathcal{I}_{2} = 2\im  \int_{0}^{t} \big\langle  (1+i\mbox{\boldmath  $\delta$}\cdot\nabla)g(u)(s) , S_{\eps}(s)u_0  \big\rangle_{L_{x}^2} \;ds = -\mathcal{I}_{1}.
\end{align*}
This finishes the proof of \eqref{conser}.
\end{proof}


\section{(In-)stability properties of stationary states with full off-axis variation} \label{sec:numfull}

In this section, we shall perform numerical simulations to study the orbital stability or instability properties of 
the (zero speed) solitary wave $Qe^{it}$ in the case with self-steepening $|\bd |\not =0$ and full off-axis variation $k=2$. 
In view of Theorem \ref{thm:fullHs}, we know that there cannot be 
any strong instability, i.e., instability due to finite-time blow-up. Nevertheless, we shall see that 
there is a wealth of possible scenarios, depending on the precise choice of parameters, 
$\sigma$, $\mbox{\boldmath $\delta$}$, and on the way we perturb the initial data. 

To be more precise, we shall consider initial data to equation \eqref{PDNLSevol} with $k=2$, given by
\begin{equation}
    u_0(x_{1},x_{2})=Q(x_{1},x_{2})\pm 0.1 e^{-x_{1}^{2}-x_{2}^{2}}
    \label{pert},
\end{equation}
where $Q$ is again the stationary state constructed numerically as described in Section \ref{sec:ground}. 
We will use $N_{x_{1}}=N_{x_{2}}=2^{10}$ Fourier modes, a numerical domain $\Omega$ of the form \eqref{domain} with $L_{x_{1}}=L_{x_{2}}=3$, and a time step of 
$\Delta t = 10^{-2}$. 

Recall that in a stable regime, the time-dependent solution $u$ typically oscillates around some time-periodic state plus a (small) remainder which radiates away as $t \to \pm \infty$ (see, e.g., \cite[Section 4.5.1]{Sul} for more details). In our simulations, however, we work on 
$\mathbb{T}^{2}$ instead of $\mathbb{R}^{2}$ which implies that 
radiation cannot escape to infinity. Thus, we will not be able to numerically verify the precise behavior of $u$ for large times. 
Having this in mind, 
we take it as numerical evidence for (orbital) stability, if both perturbations \eqref{pert} of $Q$ generate 
stable oscillations of $\| u(t, \cdot)\|_{L^\infty}$, see also  
\cite{KS, KMS} for similar studies.


\subsection{The case without self-steepening} Let us first address the case $\delta_{1}=\delta_{2}=0$ for nonlinear strengths $\sigma=1,2,3$.

For $\sigma=1$, we find that the 
perturbed ground state is unstable, and that the initial pulse disperses towards infinity as can be seen in Fig.~\ref{ex1ey1p1+}. 
The modulus of the solution at $t=10$ in the same figure on the right shows that the initial pulse disperses with an annular profile. 
A $``-"$ perturbation in (\ref{pert}) leads to the same qualitative behavior and a corresponding figure is omitted. 
\begin{figure}[htb!]
  \includegraphics[width=0.49\textwidth]{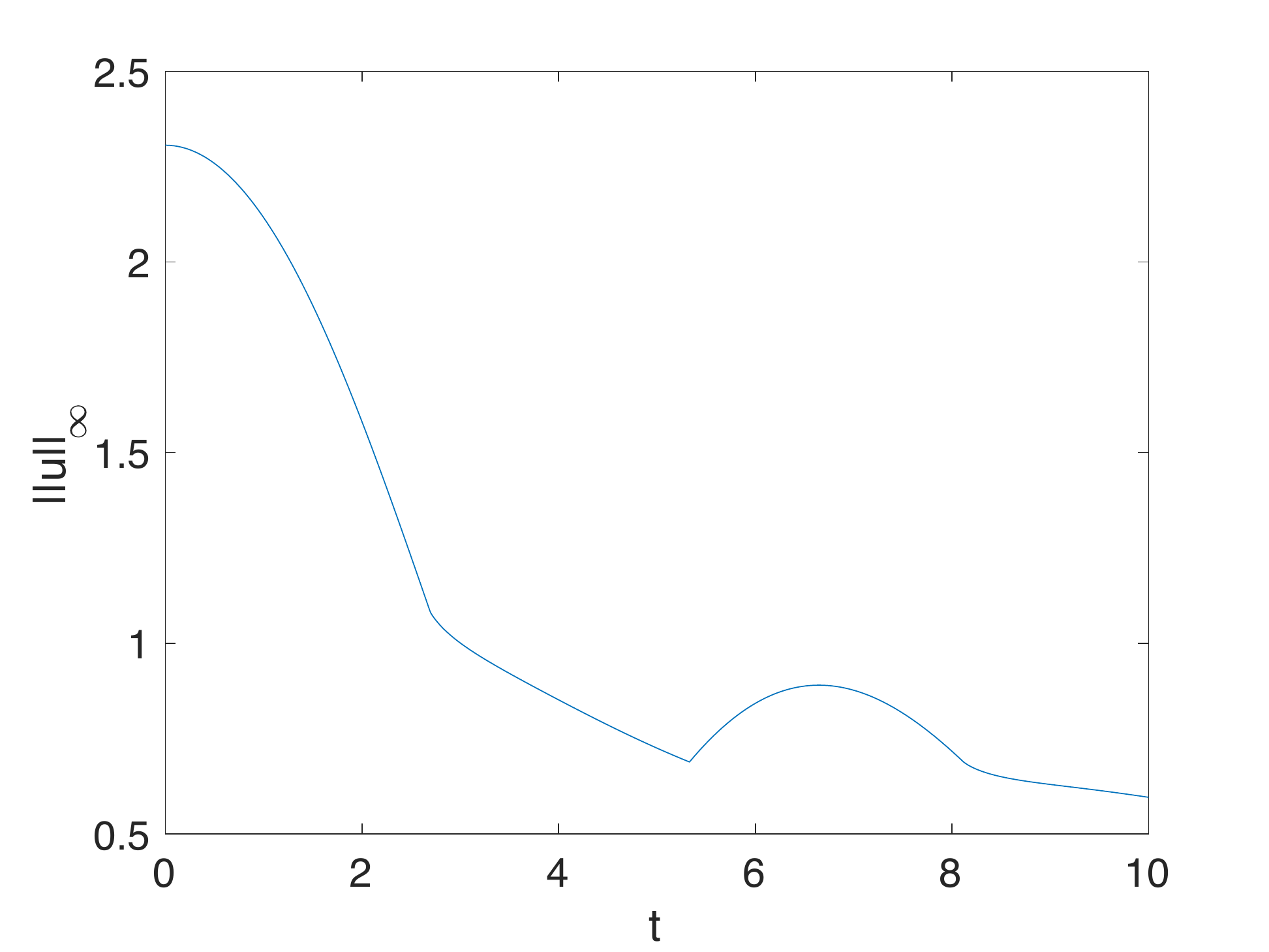}
  \includegraphics[width=0.49\textwidth]{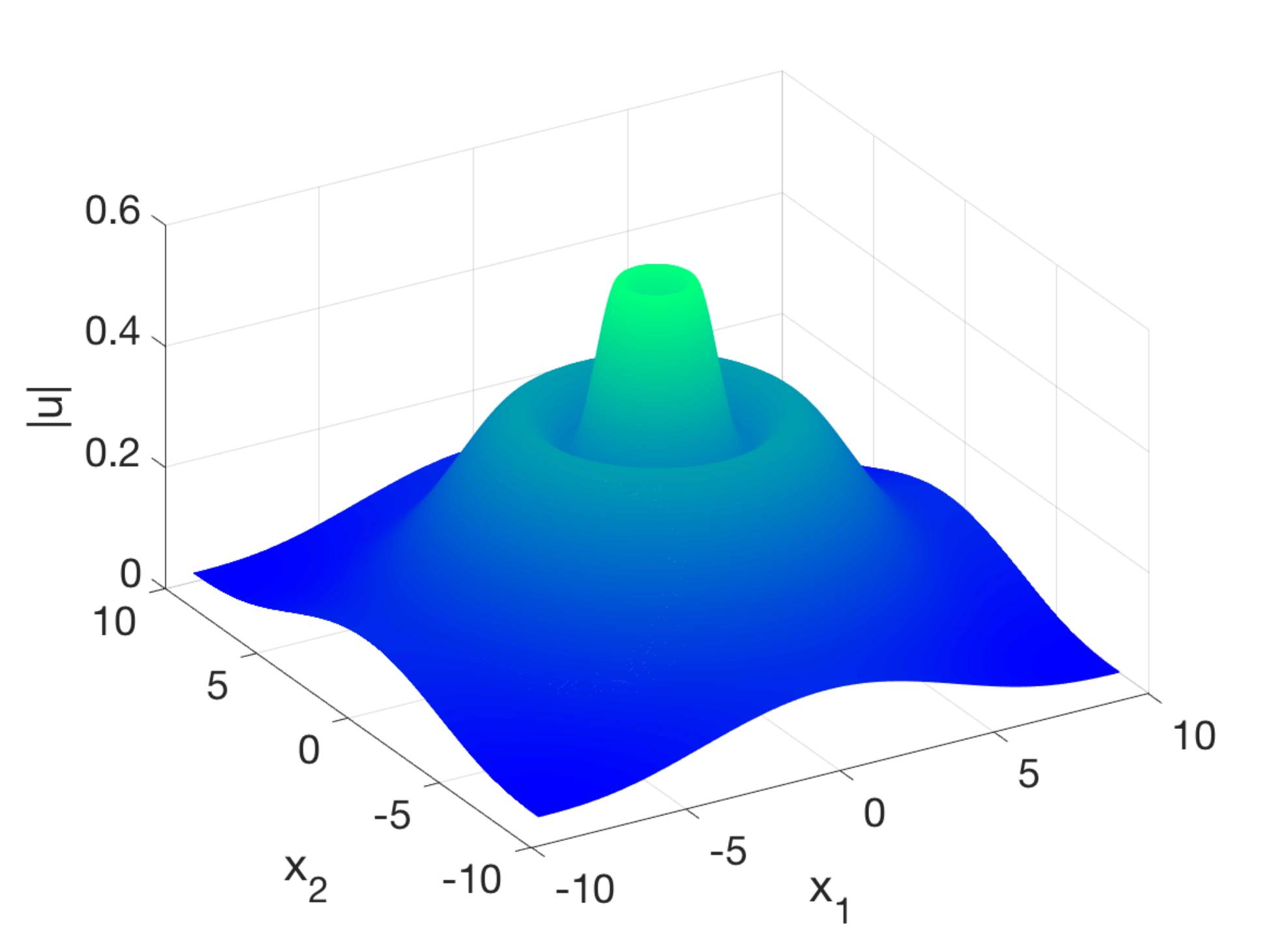}
 \caption{Solution  to equation (\ref{PDNLSevol}) with 
 $\sigma=1$, $\eps=1$, $k=2$, $\mbox{\boldmath $\delta$}=0$, and initial data \eqref{pert} with  the $``+"$ sign: On the left the $L^{\infty}$-norm of the solution as a function of $t$, and  
 on the right the modulus of the solution for $t=10$. }
 \label{ex1ey1p1+}
\end{figure}

The situation is found to be different for $\sigma =2$, where $Q$ appears to be {stable}, see Fig.~\ref{ex1ey1p2}. 
The $L^{\infty}$-norm of the solution thereby oscillates for both signs of the perturbation. 
\begin{figure}[htb!]
  \includegraphics[width=0.49\textwidth]{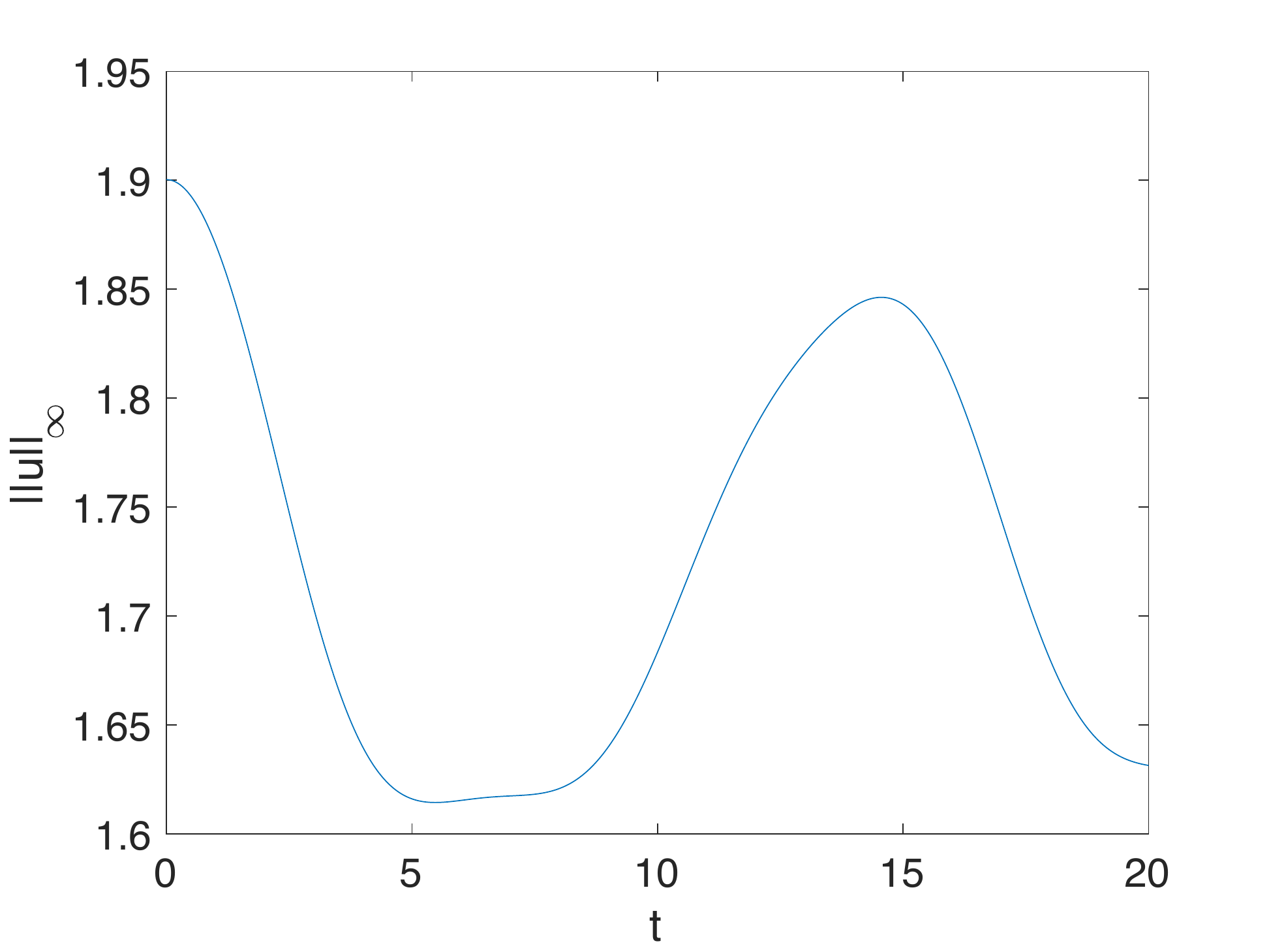}
  \includegraphics[width=0.49\textwidth]{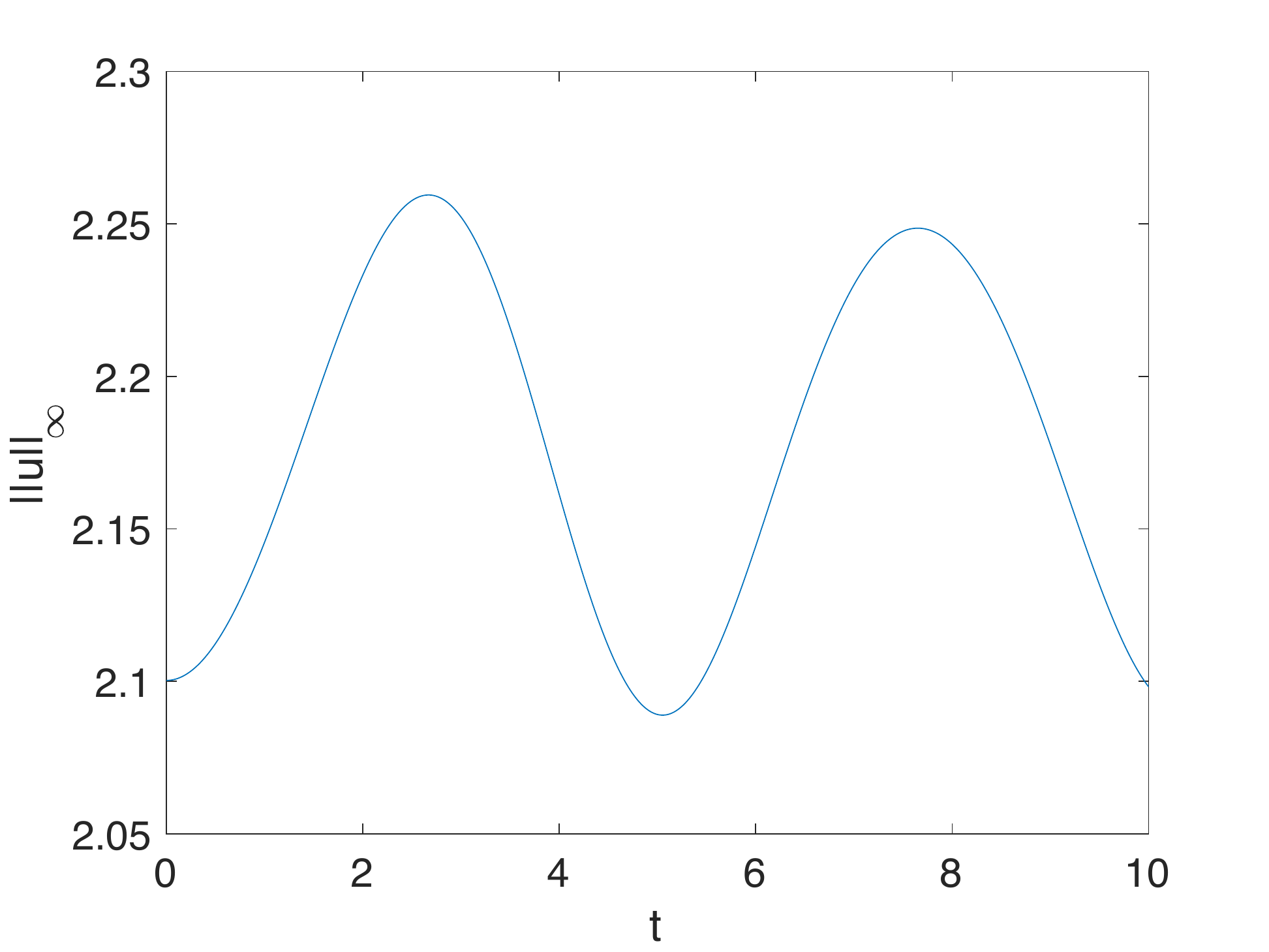}
 \caption{$L^{\infty}$-norm of the solution to equation (\ref{PDNLSevol}) with 
 $\sigma=2$, $\eps=1$, $k=2$, $\mbox{\boldmath $\delta$}=0$, and initial data \eqref{pert}: On 
 the left for 
 the $``-"$ sign, and on the right for the $``+"$ sign.}
 \label{ex1ey1p2}
\end{figure}

Finally, for $\sigma=3$ we find that the behavior depends on how we perturb the initial ground state $Q$. 
Perturbations with a $``+"$ sign in \eqref{pert} again exhibit an oscillatory behavior of the $L^{\infty}$-norm, see the right of Fig.~\ref{ex1ey1p3}. 
However, a $``-"$  perturbation yields a monotonically decreasing $L^{\infty}$-norm of the solution. The latter is again dispersed
with an annular profile. 
\begin{figure}[htb!]
  \includegraphics[width=0.49\textwidth]{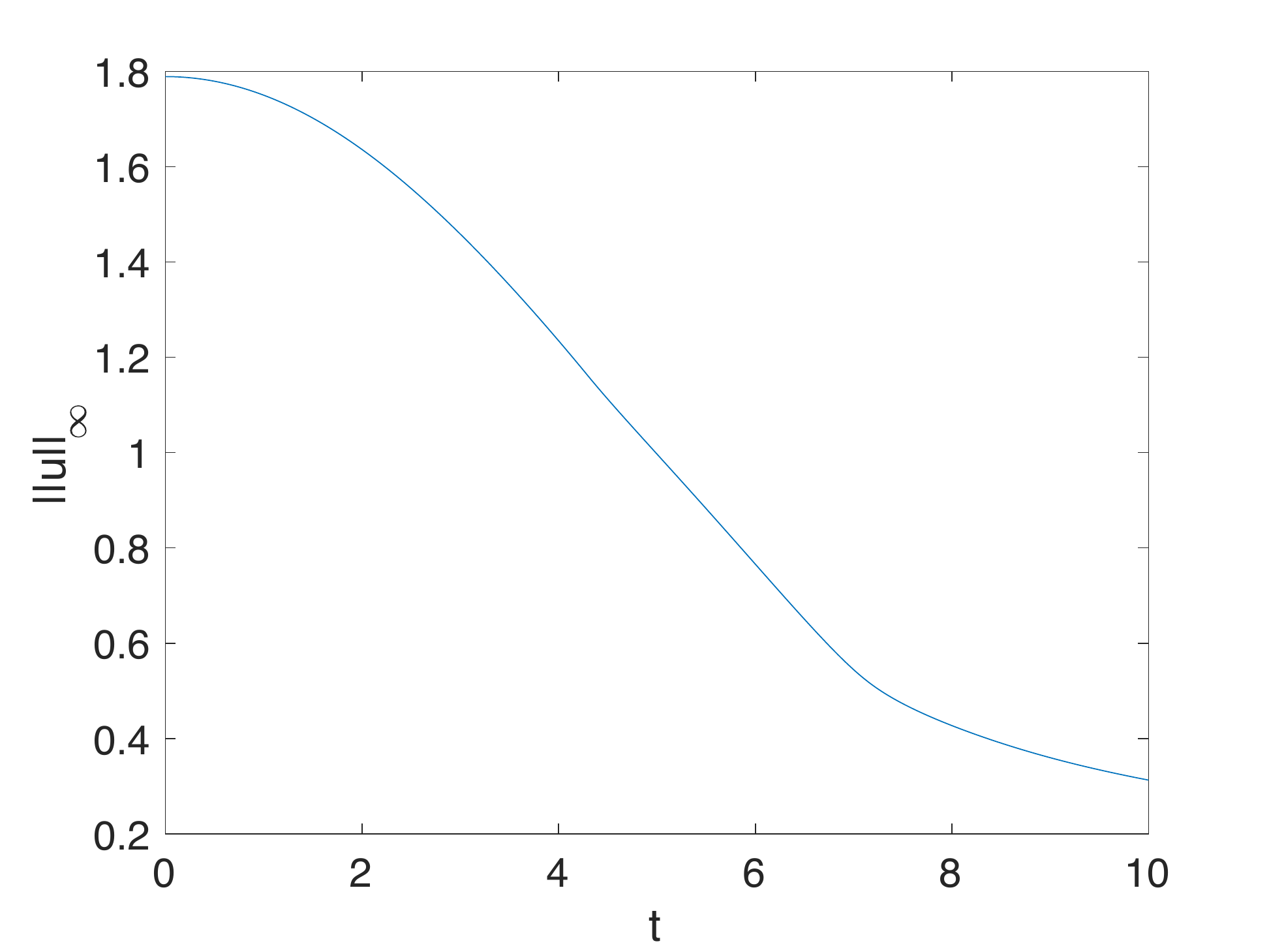}
  \includegraphics[width=0.49\textwidth]{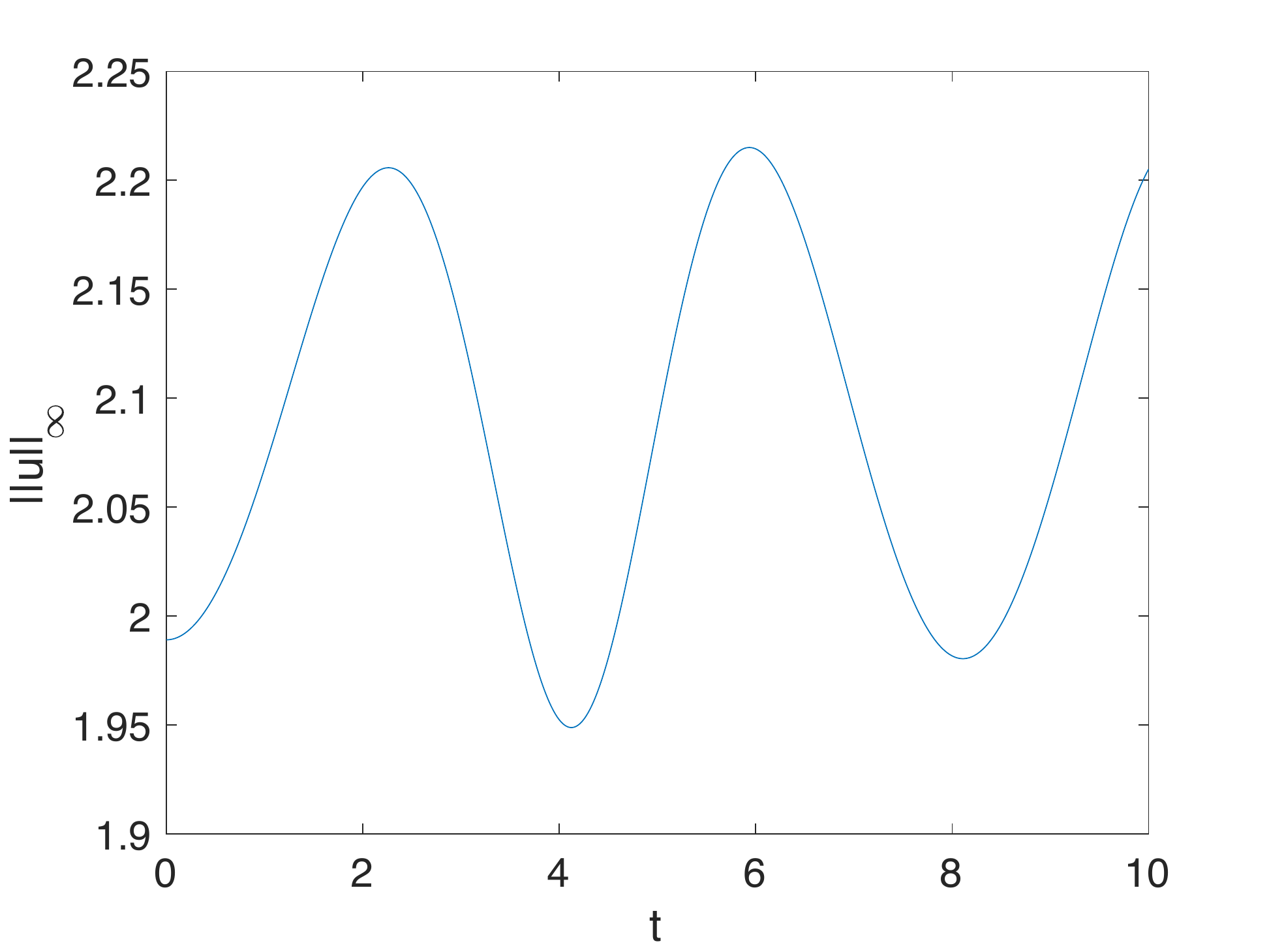}
 \caption{$L^{\infty}$-norm of the solution to equation (\ref{PDNLSevol}) with $\eps =1$, $k=2$, $\sigma=3$, $\mbox{\boldmath $\delta$}=0$ and initial data \eqref{pert}: On 
 the left for  the $``-"$ sign, on the right for the $``+"$ sign. }
 \label{ex1ey1p3}
\end{figure}


\subsection{The case with self-steepening} In this subsection, we shall perform the same numerical study in the case with self-steepening, i.e. $\delta_{1,2}\not =0$. 
For $\sigma =1$, the corresponding stationary state $Q$ seems to remain stable, since 
both types of perturbations yield an oscillatory behavior of the $L^\infty$-norm in time, see Fig.~\ref{ex1ey1d1}. This is  
in sharp contrast to the $\sigma=1$ case without self-steepening depicted in Fig.~ \ref{ex1ey1p1+} above. In addition, we see that 
the solution no longer displays an annular profile.
\begin{figure}[htb!]
  \includegraphics[width=0.32\textwidth]{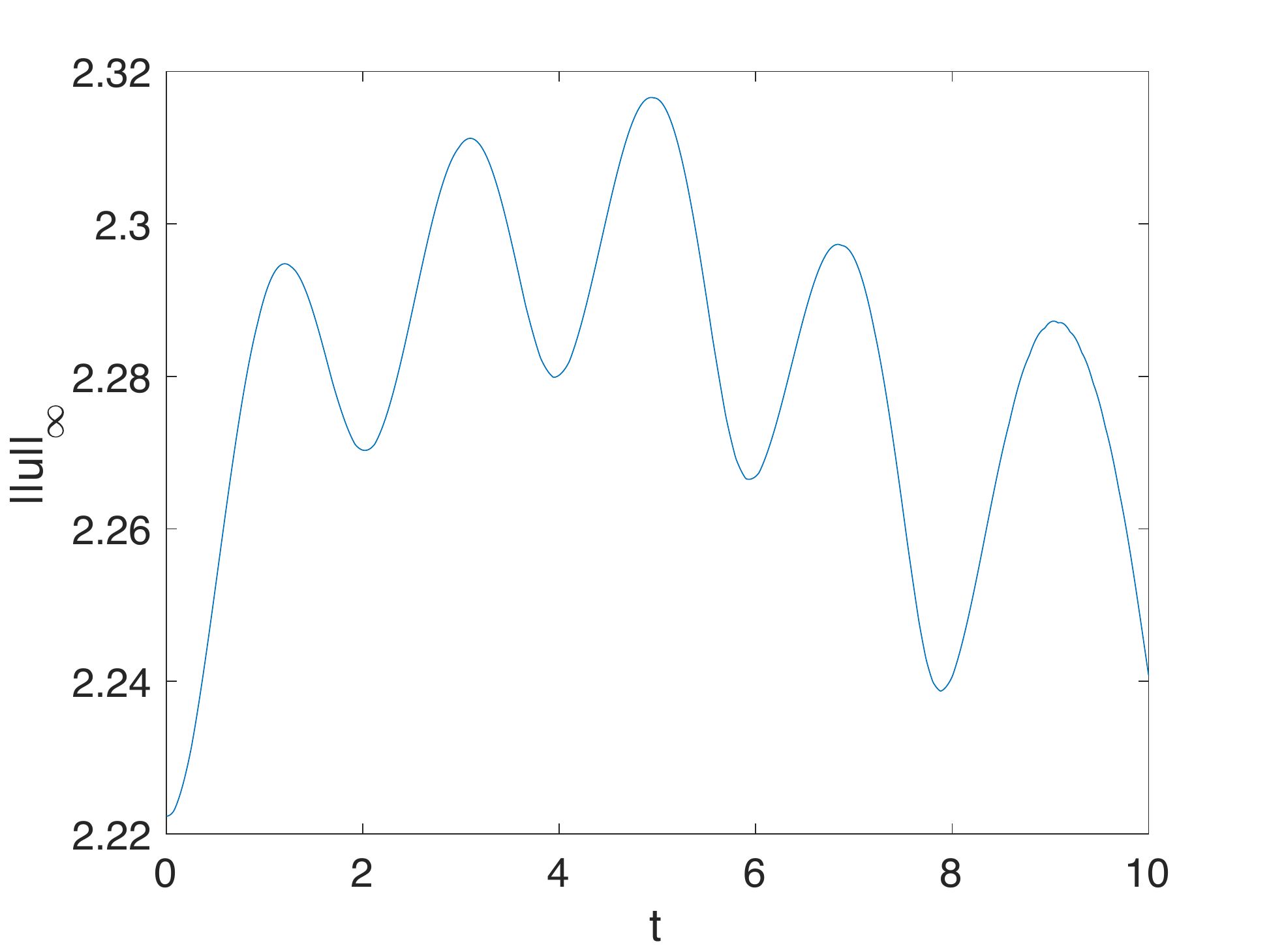}
  \includegraphics[width=0.32\textwidth]{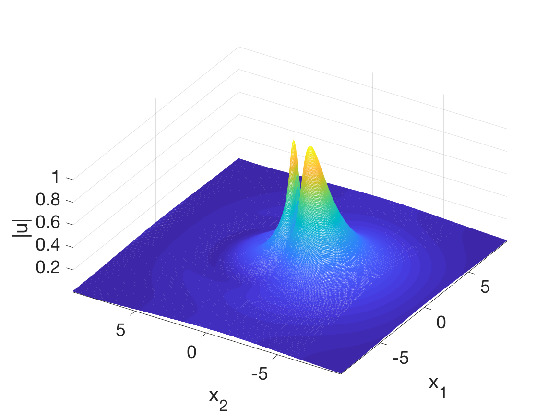}
  \includegraphics[width=0.32\textwidth]{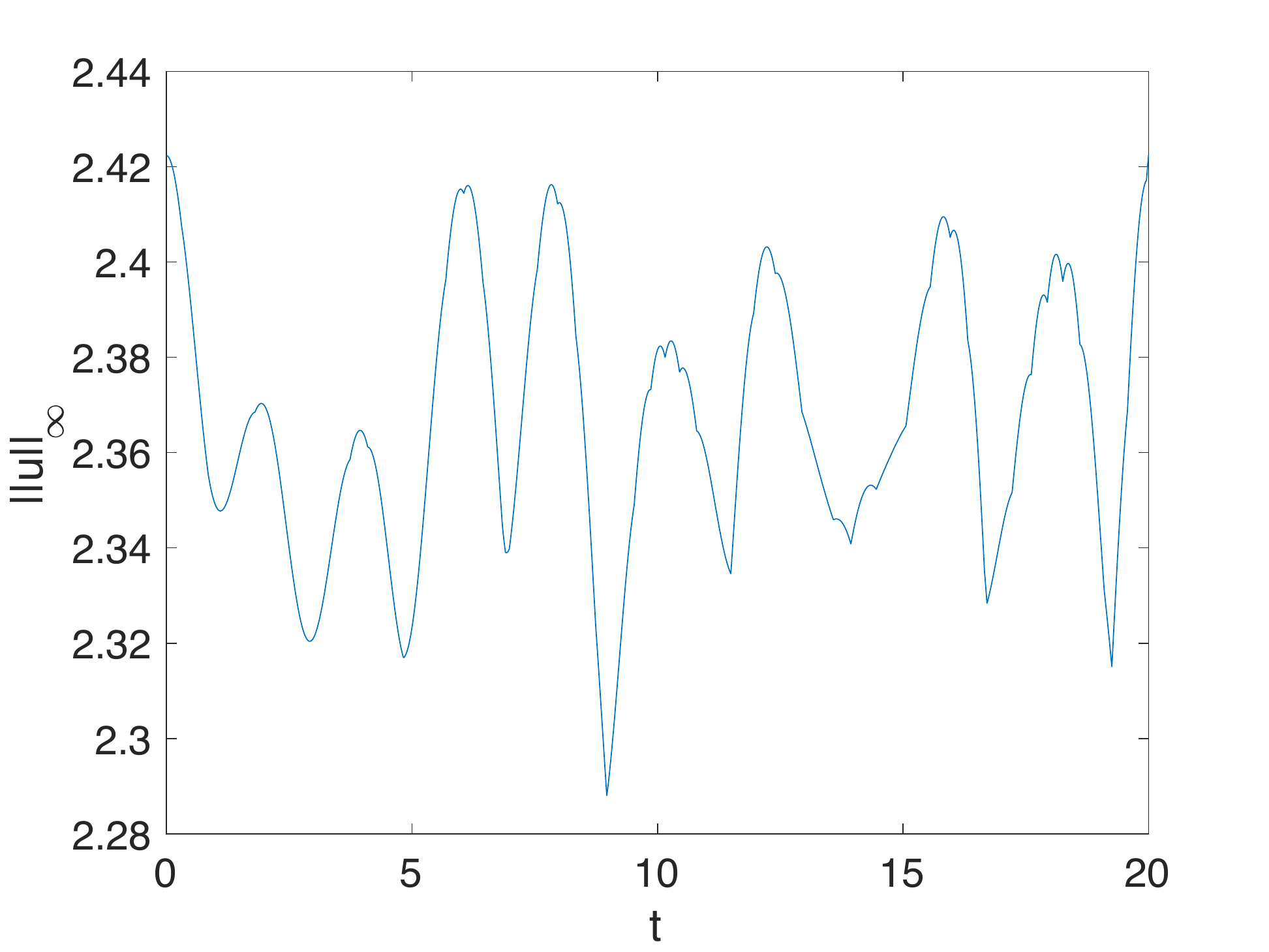}
 \caption{ Solution  to equation \eqref{PDNLSevol} with $\eps=1$, $k=2$, $\sigma=1$, $\delta_{1}=0$, $\delta_{2}=1$, and initial data \eqref{pert}: On 
 the left the $L^{\infty}$-norm for the $``-"$ sign, in the middle $|u|$ plotted at the final time, and on the right the 
 $L^{\infty}$-norm for the solution with the $``+"$ sign. }
 \label{ex1ey1d1}
\end{figure}

This stable behavior is lost in the case of higher nonlinearities. More precisely, for both $\sigma = 2$ and $3$ 
we find that the behavior of the solution $u$ depends on the sign of the considered Gaussian perturbation. On the one hand, for the $``+"$ perturbation in \eqref{pert}, 
both $\sigma =2$ and $\sigma=3$ yield an oscillatory behavior of the $L^\infty$-norm, see Fig.~\ref{ex1ey1p3d01}. 
On the other hand, the $``-"$ perturbation for both nonlinearities produce a solution with decreasing $L^{\infty}$-norm in time (although 
for $\sigma=2$ this decrease is no longer monotonically). 
\begin{figure}[htb!]
\includegraphics[width=0.32\textwidth]{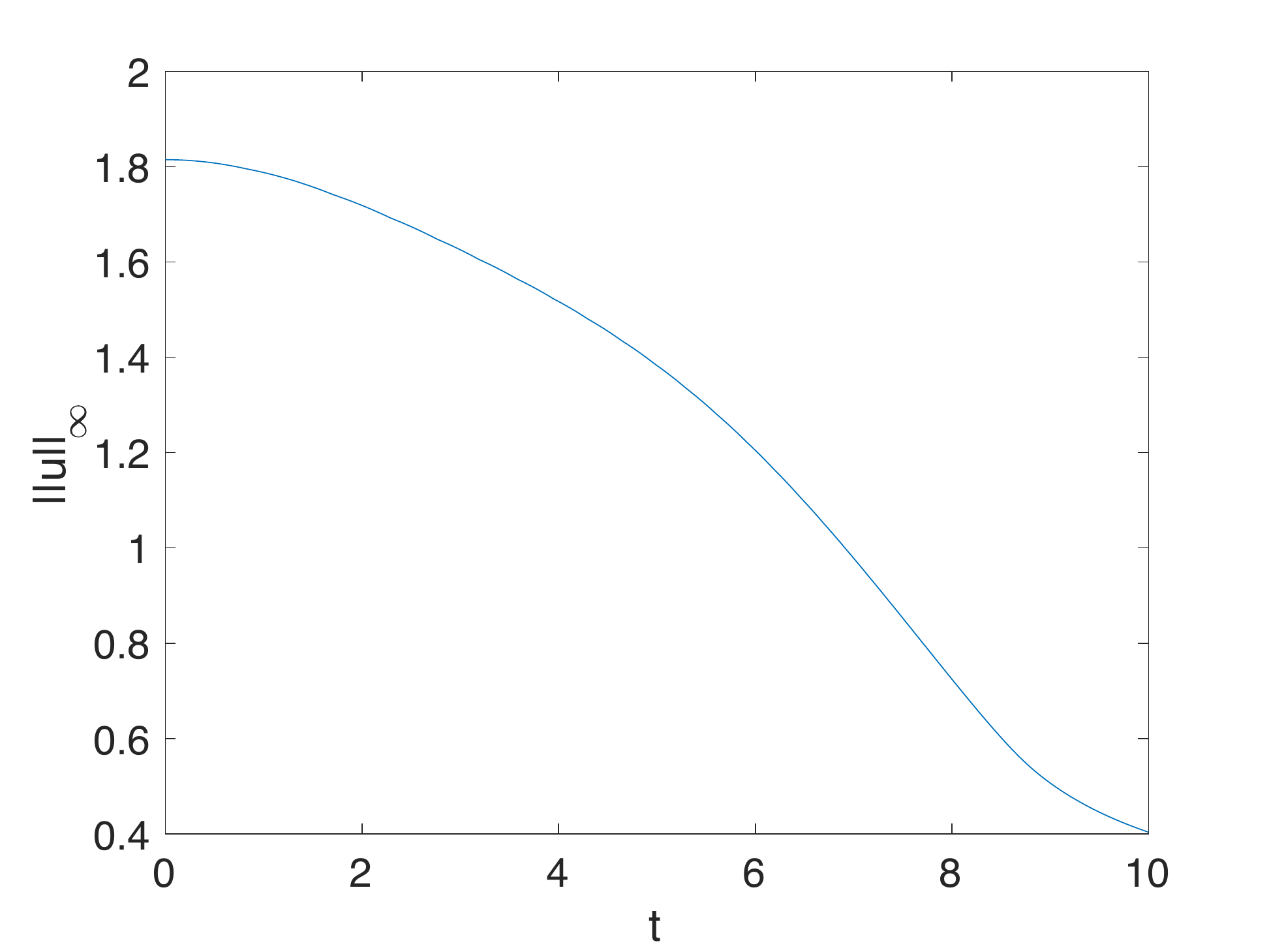}
\includegraphics[width=0.32\textwidth]{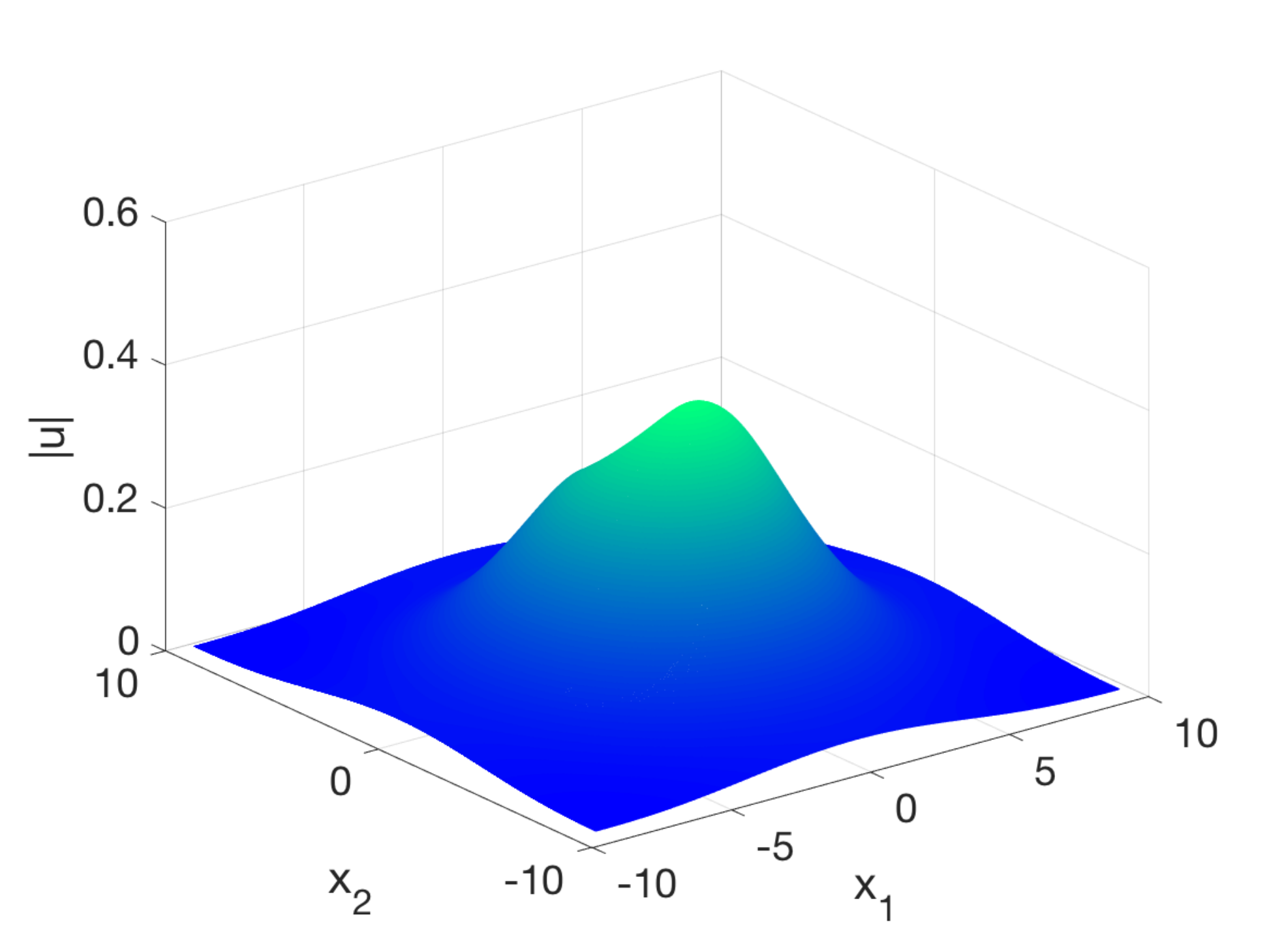}
  \includegraphics[width=0.32\textwidth]{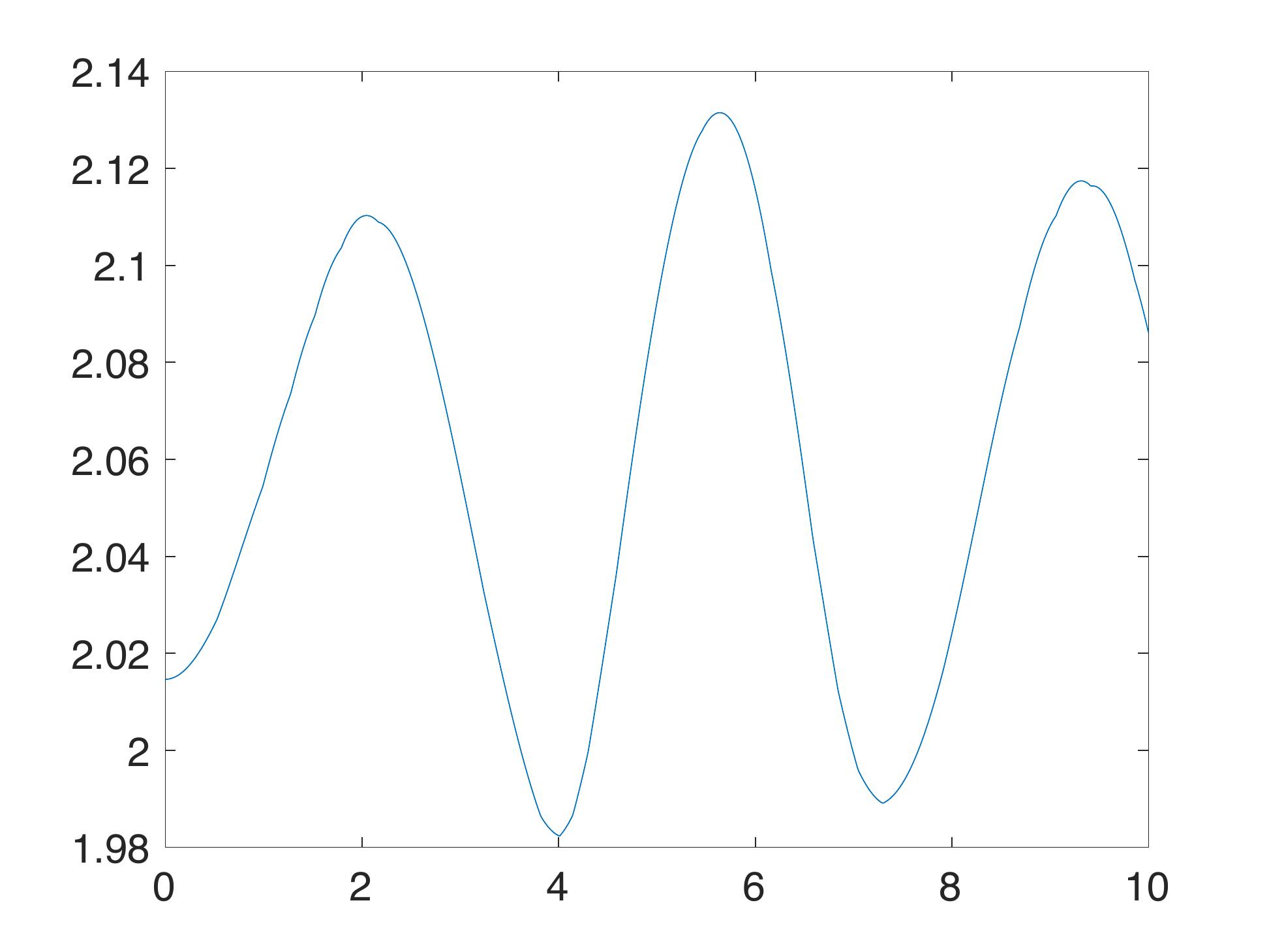}
 \caption{Solution to equation (\ref{PDNLSevol}) with $\eps=1$, $k=2$, $\sigma=3$, $\delta_{1}=0$, $\delta_{2}=0.1$, and initial data \eqref{pert}: 
 On the left, the $L^{\infty}$-norm for the $``-"$ perturbation, in the middle $|u|$ plotted at the final time, and on the right the 
 $L^{\infty}$-norm for the solution with the $``+"$ sign. }
 \label{ex1ey1p3d01}
\end{figure}

\begin{remark} Our numerical findings are reminiscent of recent results for the (generalized) BBM equation, see \cite{BMR}. In there, it is found that 
for $p\ge 5$, the regime where the underlying KdV equation is expected to exhibit blow-up, 
solitary waves can be both stable and unstable and are sensitive to the type of perturbation considered. 
The main difference to our case is of course  
that these earlier studies are done in only one spatial dimension. 
\end{remark}


\section{Well-posedness results for the case with partial off-axis variation}\label{sec:anpartial}

From a mathematical point of view, the most interesting situation arises in the case where there is only a {\it partial} off-axis variation. To study such a situation, we shall without loss of 
generality assume that $P_\eps$ acts only in the $x_{1}-$direction, i.e. 
\[
P_\eps = 1-\eps^2 \partial_{x_1}^2.
\]
In this case \eqref{PDNLS} becomes
\begin{equation}\label{ppDNLS}
i (1-\eps^2 \partial_{x_{1}}^2) \partial_{t}u + \Delta u +(1+i \mbox{\boldmath $\delta$} \cdot \nabla )(|u|^{2\sigma}u) =0, \quad u_{\mid t=0}=u_0(x_1,x_2).
\end{equation}
When $\bd =(\delta,0)^\top$ and $\sigma=1$, this is precisely the model proposed in \cite[Section 4.3]{DLS}. Motivated by this, we shall in our analysis 
only consider the case where the regularization $P_\eps$ and the derivative nonlinearity act in the same direction. 
Numerically, however, we shall also treat the orthogonal case where, instead, $\bd =(0,\delta)^\top$, see below.

\subsection{Change of unknown and Strichartz estimates}   
In \cite{AAS}, which treats the case without self-steepening, the following change of unknown is proposed in order to streamline the analysis:
\begin{equation}\label{eq:trans}
v(t,{x_{1},x_{2}}):= P_{\eps}^{1/2}u(t, {x_{1},x_{2}}).
\end{equation}
Rewriting the evolutionary form of \eqref{ppDNLS} with $\bd =(\delta,0)^\top$ in terms of $v$ yields
\begin{equation}\label{vPDNLS}
i \partial_{t}v + P_{\eps}^{-1}\Delta v +(1+i \delta\partial_{x_1} )P_{\eps}^{-1/2}(|P_{\eps}^{-1/2}v|^{2\sigma}P_{\eps}^{-1/2}v)  =0, 
\end{equation}
subject to initial data $$v_{\mid t=0}= v_{0}(x_1, x_2)\equiv P_{\eps}^{1/2}u_0(x_1, x_2).$$ Instead of \eqref{CLM}, one finds the new conservation law
\begin{equation}\label{CM}
\| v(t, \cdot) \|^{2}_{L_{x}^{2}} =\| P^{1/2}_{\eps}u(t, \cdot), \|^2_{L_{x}^2} =\| P^{1/2}_{\eps}u_{0}\|^2_{L_{x}^2} = \| v_{0} \|^{2}_{L_{x}^{2}},
\end{equation}
where we recall that $P_\eps^{1/2}$ only acts in the $x_1$-direction, via its Fourier symbol 
\[
\widehat P_1^{s/2}(\xi) =  (1+ \xi_1^2)^{s/2}, \quad \xi_1 \in \R.
\] 
This suggests to work in the mixed Sobolev-type spaces $L^{p}(\R_{x_2}; H^{s}(\R_{x_1}))$, which for any $s\in\R$ are defined through the following norm:
\begin{equation*}
\|f\|_{L^{p}_{x_2} H^{s}_{x_1}}:=\big \|P_1^{s/2}f\big \|_{L^{p}_{x_2}L^{2}_{x_1}}:=\left(\int_{\R}\left(\int_{\R}|P_1^{s/2}f(x_1, x_2)|^2\;dx_1\right)^{\frac{p}{2}}\;dx_2 \right)^{\frac1p}.
\end{equation*}
We will also make use of the mixed space-time spaces $L^q_tL^p_{x_2}H^{s}_{x_1}(I)$ for some time interval $I$ (or simply $L^q_tL^p_{x_2}H^{s}_{x_1}$ when the interval is clear from context), which we shall equip with the norm
\begin{equation*}
\|F\|_{L^q_tL^p_{x_2}H^{s}_{x_1}(I)}:=\left(\int_{I}\|F(t)\|_{L^p_{x_2}H^s_{x_1}}^q\,dt\right)^{\frac1q}.
\end{equation*}

The proof of (global) existence of solutions to \eqref{vPDNLS} will require us to use the dispersive properties of the associated linear propagator 
$S_{\eps}(t) = e^{it P_\eps^{-1} \Delta}$, which in contrast to the case $k=2$ allows for Strichartz estimates. 
However, in comparison to the usual Schr\"odinger group $e^{it\Delta}$,  these dispersive properties are considerably weaker.

In the following, we say that a pair $(q, r)$ is Strichartz {\it admissible}, 
if
\begin{equation}\label{adm}
\frac2q = \frac12-\frac1r , \quad \text{for} \ 2\le r\le\infty, 4 \le q\le\infty.
\end{equation}
Now, let $(q, r), (\gamma, \rho)$ be two arbitrary admissible pairs. It is proved in \cite[Proposition 3.4]{AAS} that 
there exist constants $C_1, C_2>0$ independent of $\eps$, such that
\begin{equation}\label{S1}
\| S_\eps(\cdot)f \|_{L^{q}_{t}L^{r}_{x_2}H^{-\frac{2}{\gamma}}_{x_1}} \le C_1\| f \|_{L_{x}^{2}}, 
\end{equation}
as well as
\begin{equation}\label{S2}
\left\|\int_{0}^{t}S_\eps(\cdot-s)F(s)\;ds \right\|_{L^{q}_{t}L^{r}_{x_2}H^{-\frac{2}{q}}_{x_1}} \le C_2\| F \|_{ L^{\gamma^{\prime}}_{t}L^{\rho^{\prime}}_{x_2}H^{\frac{2}{\gamma}}_{x_1} }. 
\end{equation}
Here, one should note the loss of derivatives in the $x_1$-direction.


\subsection{Global existence results} Using the Strichartz estimates stated above, we shall now prove some $L^2$-based global existence 
results for the solution $v$ to \eqref{vPDNLS}. In turn, this will yield global existence results (in mixed spaces) for the original equation \eqref{ppDNLS} via the transformation $v = P^{1/2}_\eps u$.  

To this end, we first recall that in the case without self-steepening $\delta =0$, the results of \cite{AAS} directly give:

\begin{proposition}[Partial off-axis variation without self-steepening]\label{thm:sub}
Let $\sigma<2$. Then for any initial data $u_0 \in L^{2}(\R_{x_2}; H^{1}(\R_{x_1}))$ there exists a unique global-in-time solution $u \in C(\R_t;L^{2}(\R_{x_2}; H^{1}(\R_{x_1})))$ to 
\begin{equation}\label{parnos}
i (1-\eps^2 \partial_{x_1}^2) \partial_{t}u + \Delta u + |u|^{2\sigma}u =0, \quad u_{\mid t=0}=u_0(x_1,x_2).
\end{equation}
\end{proposition}
Our numerical findings in the next section indicate that this result is indeed sharp, i.e., that for $\sigma\ge 2$ global existence in general 
no longer holds. 
\smallskip

Next, we shall take into account the effect of self-steepening, and rewrite \eqref{vPDNLS} using Duhamel's formula:
\begin{equation}\label{intPNLS}
\begin{split}
v(t)= &\, S_\eps(t) v_{0} + i\int_{0}^{t} S_{\eps}(t-s)P_{\eps}^{-1/2}(1+i \delta \partial_{x_1})(|P_{\eps}^{-1/2}v|^{2\sigma}P_{\eps}^{-1/2}v)(s) \;ds\\
\equiv &\, \Phi(v)(t).
\end{split}
\end{equation}
To prove that $\Phi$ is a contraction mapping, the following lemma is key.

\begin{lemma}\label{lem:parstp}
Let $g(z) = |z|^{2\sigma}z$ with $\sigma \in \N$. For $t\in [0,T]$ denote 
\begin{equation}\label{nlterm}
\mathcal{N}(v)(t) := i\int_{0}^{t} S_\eps(t-s)P_{\eps}^{-1/2}(1+i \delta\partial_{x_1} )g(P_{\eps}^{-1/2}v(s)) \;ds,
\end{equation}
and choose the admissible pair $(\gamma,\rho)= \big(\frac{4(\sigma+1)}{\sigma},2(\sigma+1)\big)$. Then for $\eps,\delta >0$, 
it holds:
\begin{align*}
& \, \big\|\mathcal{N}(v)-\mathcal{N}(v')\big\|_{L^{\gamma}_{t}L^{\rho}_{x_2}H^{-\frac{2}{\gamma}}_{x_1}} \\
& \, \lesssim \eps^{-2(\sigma+1)}(1+\delta)T^{1-\frac{\sigma}{2}}\Big(\|v\|_{L^{\gamma}_{t}L^{\rho}_{x_2}H^{-\frac{2}{\gamma}}_{x_1}}^{2\sigma}+\|v'\|_{L^{\gamma}_{t}L^{\rho}_{x_2}H^{-\frac{2}{\gamma}}_{x_1}}^{2\sigma}\Big)
\|v-v'\|_{L^{\gamma}_{t}L^{\rho}_{x_2}H^{-\frac{2}{\gamma}}_{x_1}}.
\end{align*}
\end{lemma}

\begin{proof}
First it is easy to check that $$(\gamma,\rho)= \Big(\frac{4(\sigma+1)}{\sigma},2(\sigma+1)\Big)$$
is admissible in the sense of \eqref{adm}. Moreover, 
since $\gamma >4$ we have $\frac{2}{\gamma} < \frac{1}{2}$, from 
which we infer that $H^{1-\frac{2}{\gamma}}(\R)$ is indeed a normed Banach algebra, a fact to be used below. Using the Strichartz estimate \eqref{S2} we have 
\begin{align*}
 \big\|\mathcal{N}(v)&-\mathcal{N}(v')\big\|_{L^{\gamma}_{t}L^{\rho}_{x_2}H^{-\frac{2}{\gamma}}_{x_1}}\\
 &\le C_{2}\big\|P_{\eps}^{-1/2}\big(1+i\delta\partial_{x_1}\big)\big(g(P_{\eps}^{-1/2}v)-g(P_{\eps}^{-1/2}v')\big)\big\|_{L^{\gamma'}_{t}L^{\rho'}_{x_2}H^{\frac{2}{\gamma}}_{x_1}}.
\end{align*}
For simplicity we shall in the following denote $u=P_{\eps}^{-1/2}v,u'=P_{\eps}^{-1/2}v'$ in view of \eqref{eq:trans}. Keeping $t$ and $x_2$ fixed we can estimate
\begin{align*}
\big\|P_{\eps}^{-1/2}\big(1+i\delta\partial_{x_1}\big)&\big(g(P_{\eps}^{-1/2}v)-g(P_{\eps}^{-1/2}v')\big)\big\|_{H^{\frac{2}{\gamma}}_{x_1}}\\
&\le \eps^{-1}\|\big(1+i\delta\partial_{x_1}\big)\big(g(u)-g(u')\big)\big\|_{H^{\frac{2}{\gamma}-1}_{x_1}}\\
&\le  \eps^{-1}(1+\delta)\|g(u)-g(u')\|_{H^{1-\frac{2}{\gamma}}_{x_1}},
\end{align*}
where in the last inequality we have used the fact that $ H^{1-\frac{2}{\gamma}}(\R) \subset H^{\frac{2}{\gamma}}(\R)$. Next, use again \eqref{g1} which together with the algebra property of $H^{1-\frac{2}{\gamma}}(\R)$ for $\sigma \in \N$ implies
\begin{align*}
\|g(u)-g(u')\|_{H^{1-\frac{2}{\gamma}}_{x_1}} &\lesssim \Big( \|u\|^{2\sigma}_{H^{1-\frac{2}{\gamma}}_{x_1}}  + \|u'\|^{2\sigma}_{H^{1-\frac{2}{\gamma}}_{x_1}}  \Big)\|u-u'\|_{H^{1-\frac{2}{\gamma}}_{x_1}}\\
&\lesssim \eps^{-2(\sigma+1)}\Big( \|v\|^{2\sigma}_{H^{-\frac{2}{\gamma}}_{x_1}}  + \|v'\|^{2\sigma}_{H^{-\frac{2}{\gamma}}_{x_1}}  \Big)\|v-v'\|_{H^{-\frac{2}{\gamma}}_{x_1}}.
\end{align*}
It consequently follows after H\"older's inequality in $x_2$, that we obtain 
\begin{align*}
\big\|\mathcal{N}(v)&-\mathcal{N}(v')\big\|_{L^{\rho}_{x_2}H^{-\frac{2}{\gamma}}_{x_1}}\\
&\lesssim \eps^{-(2\sigma+1)}(1+\delta)\Big( \|v\|^{2\sigma}_{L_{x_2}^{\rho}H^{-\frac{2}{\gamma}}_{x_1}}  + \|v'\|^{2\sigma}_{L_{x_2}^{\rho}H^{-\frac{2}{\gamma}}_{x_1}}  \Big)\|v-v'\|_{L_{x_2}^{\rho}H^{-\frac{2}{\gamma}}_{x_1}}.
\end{align*}
The result then follows after applying yet another H\"older's inequality in $t$.
\end{proof}

This lemma allows us to prove the following global existence result for \eqref{ppDNLS}. 

\begin{theorem}[Partial off-axis variation with parallel self-steepening]\label{thm:subdel}
Let $\sigma = 1$ and $\bd =(\delta,0)^{\top}$ for $\delta \in \R$.
Then for any $u_0\in L^{2}(\R_{x_2}; H^{1}(\R_{x_1}))$ there exists a unique global solution 
$u\in C(\R_t;L^{2}(\R_{x_2}; H^{1}(\R_{x_1})))$ to \eqref{ppDNLS}.
\end{theorem}

Here, the restriction $\sigma =1$ is due to the fact that this is the only $\sigma \in \N$ (required for the normed algebra property above) for which the problem is subcritical.  Indeed, in view of the estimate in Lemma \ref{lem:parstp}, the exponent $1-\frac{\sigma}{2} >0$ yields a contraction for small times. 

\begin{proof} 
We seek to show that $v\mapsto \Phi (v)$ is a contraction mapping in a suitable space. To this end, we denote, as before,
\begin{equation*}
\Phi(v)(t) = S_\eps(t)v_{0} + \mathcal{N}(v)(t),
\end{equation*}
where $\mathcal{N}(v)$ is given by \eqref{nlterm}. Let $T,M>0$ and denote
\begin{align*}
Y_{T,M} = &\{ v \in  L^{\infty}([0,T);L^{2}(\R_{x}^{2})) \cap L^{8}([0,T);L^{4}(\R_{x_2} ; H^{-\frac{1}{4}}(\R_{x_1}))): \\ 
& \ \|v \|_{L_{t}^{\infty}L_{x}^{2}} + \|v \|_{L^{8}_{t}L^{4}_{x_2}H^{-\frac{1}{4}}_{x_1}} \le M \}.
\end{align*}

The Strichartz estimates \eqref{S1} and \eqref{S2} together with Lemma \ref{lem:parstp} imply that for any admissible pair $(q,r)$ and solutions $v,v' \in Y_{T,M}$ that 
\begin{align*}
& \| \Phi(v) - \Phi(v') \|_{L^{q}_{t}L^{r}_{x_2}H^{-\frac{2}{q}}_{x_1} } \\
&\,  \le\| S_{\eps}(t)(v_{0}-v'_{0}) \|_{L^{q}_{t}L^{r}_{x_2}H^{-\frac{2}{q}}_{x_1}} + \|\mathcal{N}(v) - \mathcal{N}(v')\|_{L^{q}_{t}L^{r}_{x_2}H^{-\frac{2}{q}}_{x_1}} \\
&\, \le C_{1}\|v_{0} -v'_{0}\|_{L_{x}^{2}} + C_{2}\big\|P_{\eps}^{-1/2}\big(1+i\delta\partial_{x_1}\big)\big(g(P_{\eps}^{-1/2}v) - g(P_{\eps}^{-1/2}v')\big)\big\|_{L^{\frac{8}{7}}_{t}L^{\frac{4}{3}}_{x_2}H^{\frac{1}{4}}_{x_1}}\\
&\, \le C_{\sigma,\eps}\Big(\| v_{0} -v'_{0} \|_{L_{x}^{2}} +  T^{1/2}M^{2}\|v-v'\|_{L^{8}_{t}L^{4}_{x_2}H^{-\frac{1}{4}}_{x_1}} \Big).
\end{align*}
Choosing $M=M(\|v_{0}\|_{L_{x}^{2}})$ and $T$ sufficiently small, it is clear that $\Phi$ is a contraction on $Y_{T,M}$. Banach's fixed point theorem 
and a standard continuity argument thus yield the existence of a unique maximal solution $v \in C([0,T_{\rm max}),L^{2}(\R_{x}^{2}))$ where $T_{\rm max}=T_{\rm max}(\|v_0\|_{L_{x}^2})$. 
Continuous dependence on the initial data follows by classical arguments. 

The conservation property \eqref{CM} for $v$ follows similarly as in the proof of Proposition $4.2$ in \cite{AAS} and we shall therefore 
only sketch its main steps below.
By the unitary of $S_{\eps}(\cdot)$ in $L^{2}$ we obtain 
\begin{align*}
\|v(t)\|_{L_{x}^{2}} &= \|v_{0}\|_{L_{x}^{2}}
+2\re \, \big \langle S_\eps(-t)\mathcal{N}(v)(t),v_{0} \big\rangle_{L_{x}^{2}} + \| S_\eps(-t)\mathcal{N}(v)(t) \|^2_{L_{x}^{2}}\\
&=: \|v_{0}\|_{L_{x}^{2}}  + \mathcal{I}_{1} + \mathcal{I}_{2}. 
\end{align*}
To show that $\mathcal{I}_{1} + \mathcal{I}_{2}=0$, we use \eqref{intrepu1} and rewrite
\begin{align*}
\mathcal{I}_{1} = -2\im  \int_{0}^{t} \big\langle  P_{\eps}^{-1/2}(1+i\delta \partial_{x_{1}})g(P_{\eps}^{-1/2}v)(s) , S_{\eps}(s)v_0  \big\rangle_{L_{x}^2} \;ds.
\end{align*}
By duality in $x_1$ and H\"older's inequality in $t$ and $x_2$ we find that this quantity is indeed finite, since
\begin{align*}
|\mathcal{I}_{1}| \le 2 \| P_{\eps}^{-1/2}(1+i\delta \partial_{x_{1}})g(P_{\eps}^{-1/2}v)\|_{L^{\gamma'}_{t}L^{\rho'}_{x_2}H^{\frac{2}{\gamma}}_{x_1}}\| S_{\eps}(\cdot)v_0 \|_{L^{\gamma}_{t}L^{\rho}_{x_2}H^{-\frac{2}{\gamma}}_{x_1}} <\infty.
\end{align*}
Once again we find, after a lengthy computation (see \cite{AAS} for more details), that
\begin{align*}
\mathcal{I}_{2}& = 2 \re \int_{0}^{t} \big\langle  P_{\eps}^{-1/2}(1+i\delta \partial_{x_{1}})g(P_{\eps}^{-1/2}v)(s) , -i\mathcal{N}(u)(s) \big\rangle_{L_{x}^{2}}\;ds.
\end{align*}
We express $-i\mathcal{N}(u)(s)$ using the integral formulation \eqref{intPNLS} and write
\begin{align*}
\mathcal{I}_{2} &= 2 \re \, \int_{0}^{t} \big\langle  P_{\eps}^{-1/2}(1+i\delta \partial_{x_{1}})g(P_{\eps}^{-1/2}v)(s) , iS_{\eps}(s)u_0 \big\rangle_{L_{x}^{2}}\;ds \\
&+ \int_{0}^{t}\im \, \|P_{\eps}^{-1/2}v(s)\|^{2\sigma 
+2}_{L_{x}^{2\sigma +2}} - \delta \, \re \, \big\langle g(P_{\eps}^{-1/2}v),  \partial_{x_1}P_{\eps}^{-1/2}v \big\rangle_{L_{x}^{2}}(s) \;ds.
\end{align*}
Here the second time integral vanishes entirely, and, as in the full off-axis case, the latter term in the integrand vanishes due to
\begin{equation*}
\re \, \big\langle g(P_{\eps}^{-1/2}v), \partial_{x_{1}}P_{\eps}^{-1/2}v \big\rangle_{L_{x}^{2}} = \frac{2}{(\sigma+1)}\int_{\R}\int_{\R}\partial_{x_1}(|P_{\eps}^{-1/2}v|^{2\sigma+2} )\;dx_{1} dx_{2} =0.
\end{equation*}
In summary, we find that 
\begin{align*}
\mathcal{I}_{2} = 2\im  \int_{0}^{t} \big\langle  P_{\eps}^{-1/2}g(P_{\eps}^{-1/2}v)(s) , S_{\eps}(s)u_0  \big\rangle_{L_{x}^2} \;ds = -\mathcal{I}_{1},
\end{align*}
which finishes the proof of \eqref{CM}. We can thus extend $v$ to become a global solution by repeated iterations to conclude $T_{\rm max} = +\infty$. 

Finally, 
we use the fact that $v=P_{\eps}^{1/2}u$ to obtain a unique global-in-time solution $u\in C(\R_t;L^{2}(\R_{x_2}; H^{1}(\R_{x_1})))$ which finishes the proof. 
\end{proof}

\begin{remark} It is possible to treat the critical case $\sigma=2$ using the same type of arguments as in \cite{CaWe} (see also \cite{AAS}). 
Unfortunately, this will only yield local-in-time solutions up to some time $T=T(u_0)>0$, which depends on the 
initial profile $u_0$ (and not only its norm). Only for sufficiently small initial data $\| u_0\|_{L^2_{x_2}H^1_{x_1}}< 1$, does one obtain a global-in-time solution. But 
since it is hard to detect small nonlinear effects numerically, we won't be 
concerned with this case in the following. We also mention the possibility of obtaining (not necessarily unique) global weak solutions 
for derivative NLS, which has been done in \cite{AmSi} in one spatial dimension.
\end{remark}

Theorem \ref{thm:subdel} covers the situation in which a partial off-axis regularization acts parallel to the self-steepening. At present, no analytical result 
for the case where the two effects act orthogonal to each other is available. Numerically, however, it is possible to study such a szenario:
To this end, we we recall that from the physics point of view, both $\varepsilon$ and $| \bd |$ have to be considered as (very) small parameters. 
With this in mind, we study the time-evolution of \eqref{ppDNLS} with $\sigma =1$, Gaussian initial data of the form \eqref{gaussian}, and a relatively small 
self-steepening, furnished by $\delta_1=0$ and $\delta_{2}=0.1$.
In the case where $\eps=0$, it can be seen on the left of
Fig.~\ref{gaussdy01inf} that the $L^{\infty}$-norm of the 
solution indicates a finite-time blow-up at $t\approx T = 0.1445$. In the same situation 
with a small, but nonzero $\varepsilon=0.1$, one can see that, instead, 
oscillations appear within the $L^{\infty}$-norm of the solution for $t\ge T$. 
\begin{figure}[htb!]
  \includegraphics[width=0.49\textwidth]{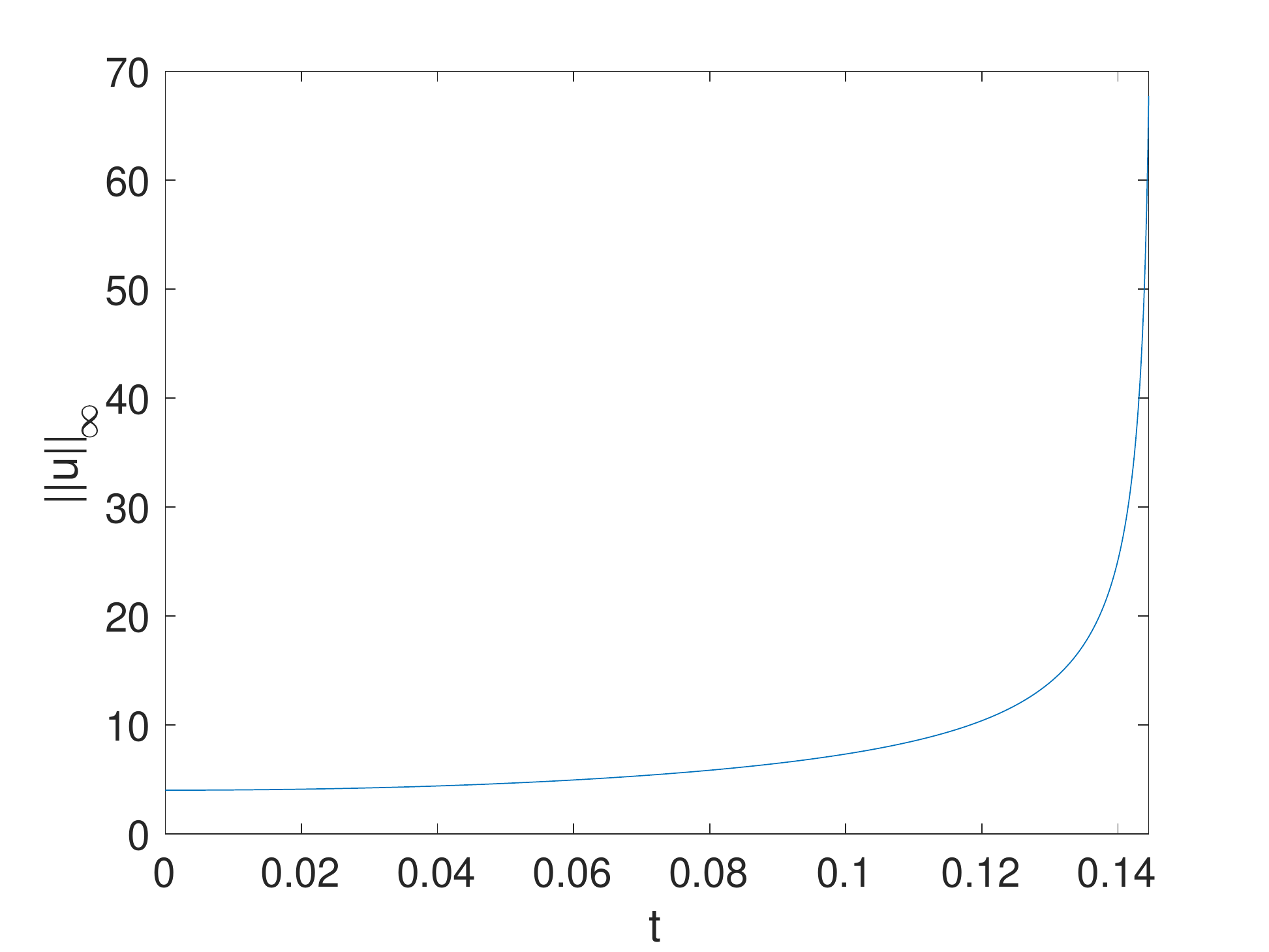}
  \includegraphics[width=0.49\textwidth]{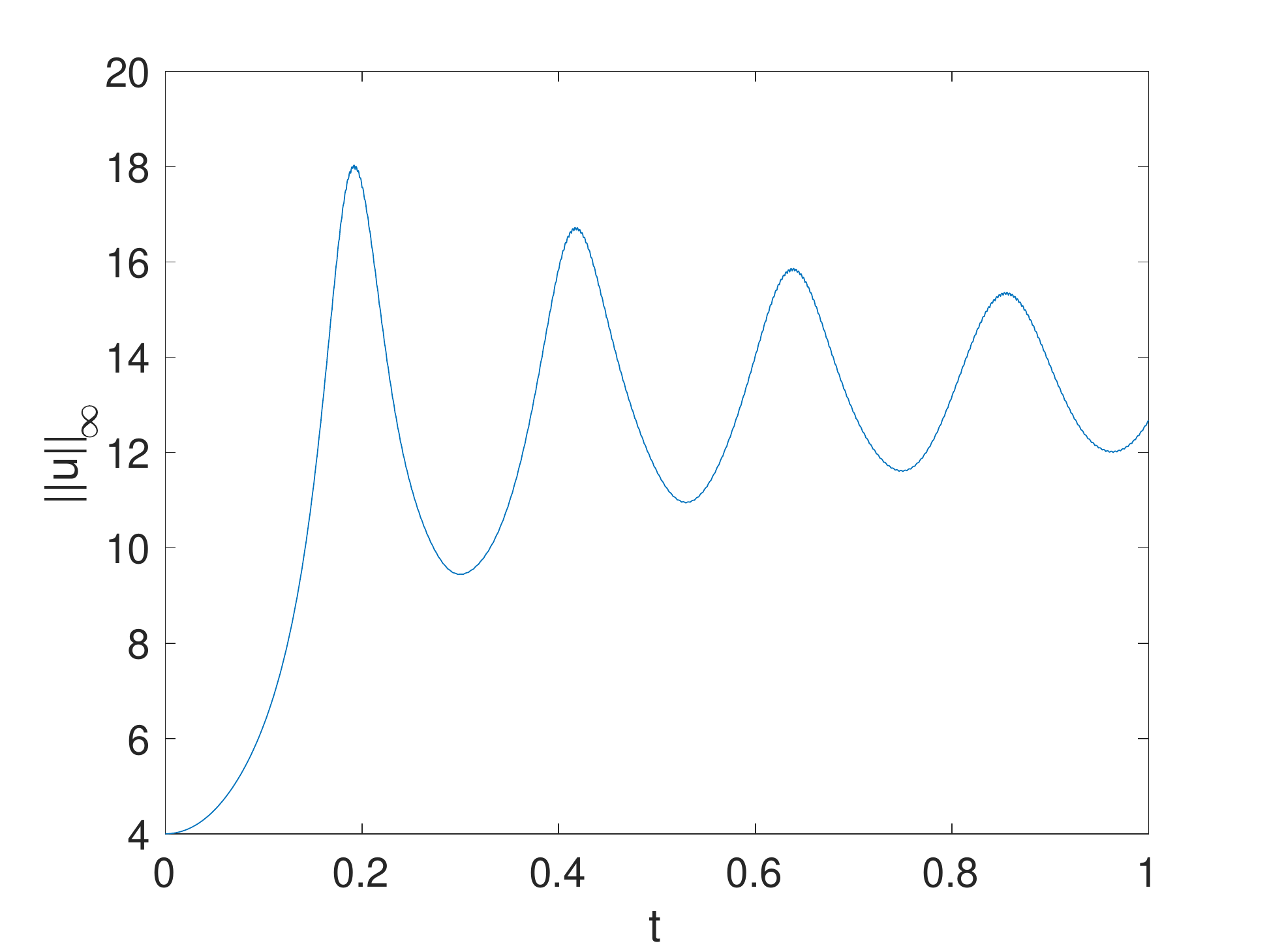}
 \caption{$L^{\infty}$-norm of the solution to \eqref{ppDNLS} with $\sigma =1$, $\bd=(0, 0.1)$, and 
 initial data $u_0=4\exp^{-x_{1}^{2}-x_{2}^{2}}$: On the left for 
 $\eps=0$, on the right for $\eps=0.1$.}
 \label{gaussdy01inf}
\end{figure}

Note that 
these oscillations appear to decrease in amplitude, which indicates the possibility of an asymptotically stable 
final state as $t\to +\infty$. 
A similar behavior can be seen for different choices of parameters and also for a full, two-dimensional off-axis variation (not shown here).


\section{Numerical studies for the case with partial off-axis variation}\label{sec:numpartial}

In this section we present numerical studies for the model \eqref{ppDNLS} with $\eps =1$ and different values of the self-steepening parameter ${\bd}$, 
as well as $\sigma>0$. We will always use $N_{x_{1}}=N_{x_{2}}=2^{10}$, Fourier coefficients on the numerical domain $\Omega$ given by \eqref{domain} with 
$L_{x_{1}}=L_{x_{2}}=3$. The time step is $\Delta t=10^{-2}$ unless otherwise noted. The initial data is the same as in \eqref{pert}, i.e. a numerically constructed stationary 
state $Q$ perturbed by adding and subtracting small Gaussians, respectively.

\subsection{The case without self-steepening}
We shall first study the particular situation furnished by equation \eqref{parnos} with $\eps =1$. It is 
obtained from the general model \eqref{PDNLS} in the case without self-steepening $\delta_{1}=\delta_{2}=0$:

In the case $\sigma=1$, the ground state perturbation in \eqref{pert} with a $``+"$ sign is unstable and results in a purely dispersive solution 
with monotonically decreasing $L^{\infty}$-norm, see Fig.~\ref{Nlssolex1+}. 
The modulus of the solution at time $t=2.5$ is shown on the right of the same figure. Interestingly, the initial hump appears to separate into four smaller humps and 
we thus lose radial symmetry of the solution. 
The situation is qualitatively similar for perturbations corresponding to the $``-"$ sign in \eqref{pert} and we thus omit a corresponding figure. 
\begin{figure}[htb!]
  \includegraphics[width=0.49\textwidth]{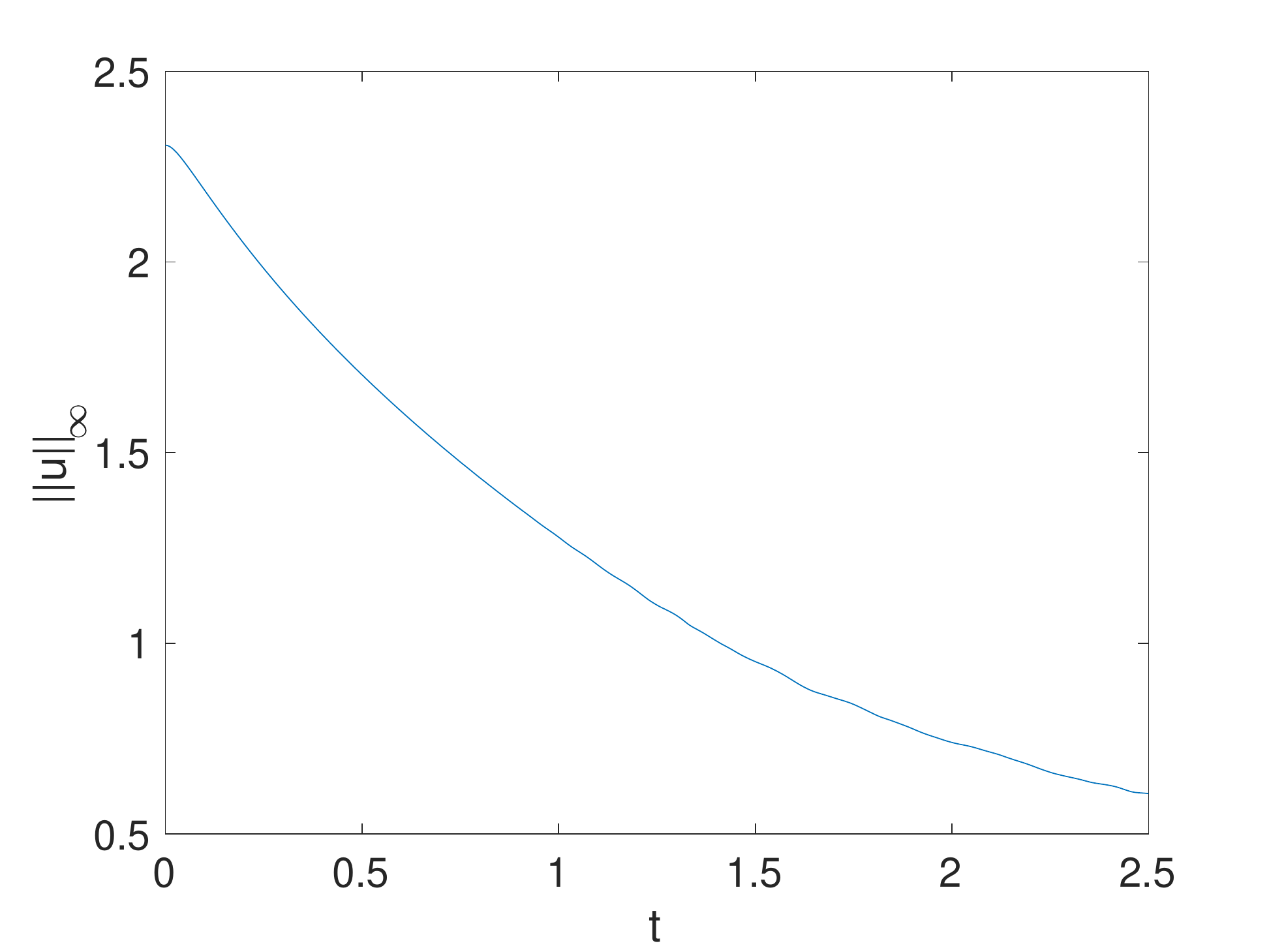}
  \includegraphics[width=0.49\textwidth]{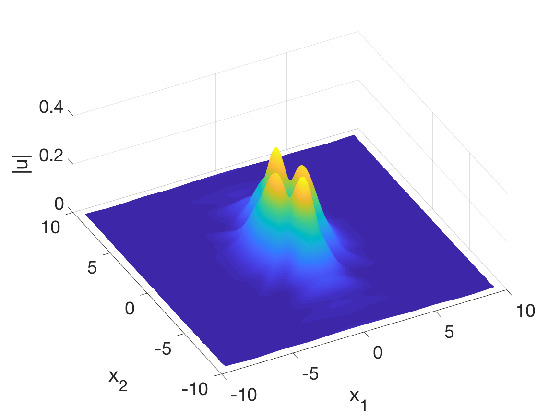}
 \caption{Solution to \eqref{parnos} with  $\eps =1$, 
 $\sigma=1$, and initial data (\ref{pert}) with a $``+"$ sign: On the 
 left the $L^{\infty}$-norm in dependence of time, on the right the modulus of $u$ at $t=2.5$.}
 \label{Nlssolex1+}
\end{figure}

The situation changes significantly for $\sigma=2$, as can be seen in 
Fig.~\ref{Nlssolp2ex1+max}. While the $L^{\infty}$-norm of the solution obtained from initial data \eqref{pert} with the $``-"$ sign is again decreasing, the $``+"$ sign yields 
a monotonically increasing $L^\infty$-norm indicating a blow-up at $t\approx 0.64$. 
\begin{figure}[htb!]
  \includegraphics[width=0.49\textwidth]{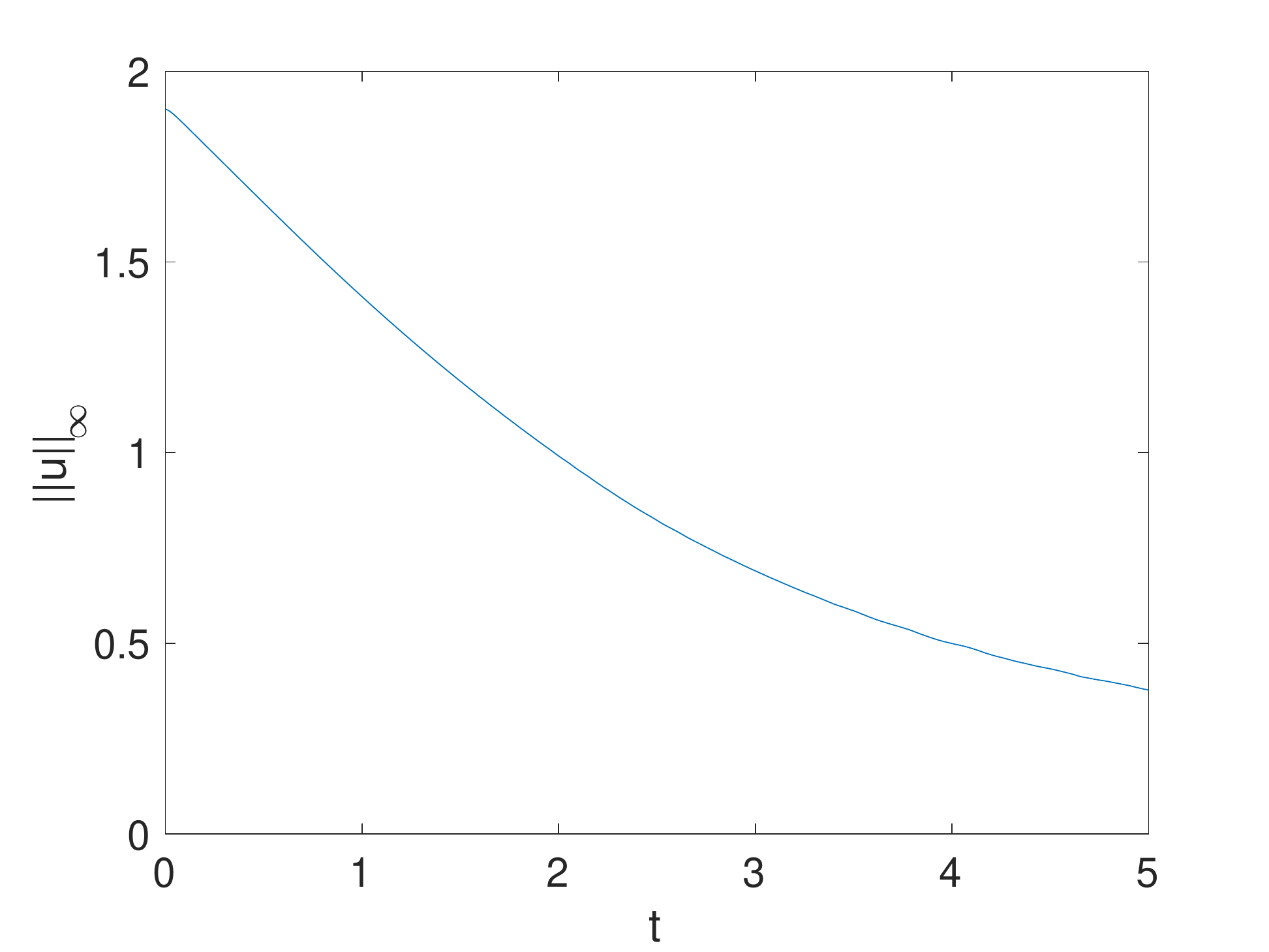}
  \includegraphics[width=0.49\textwidth]{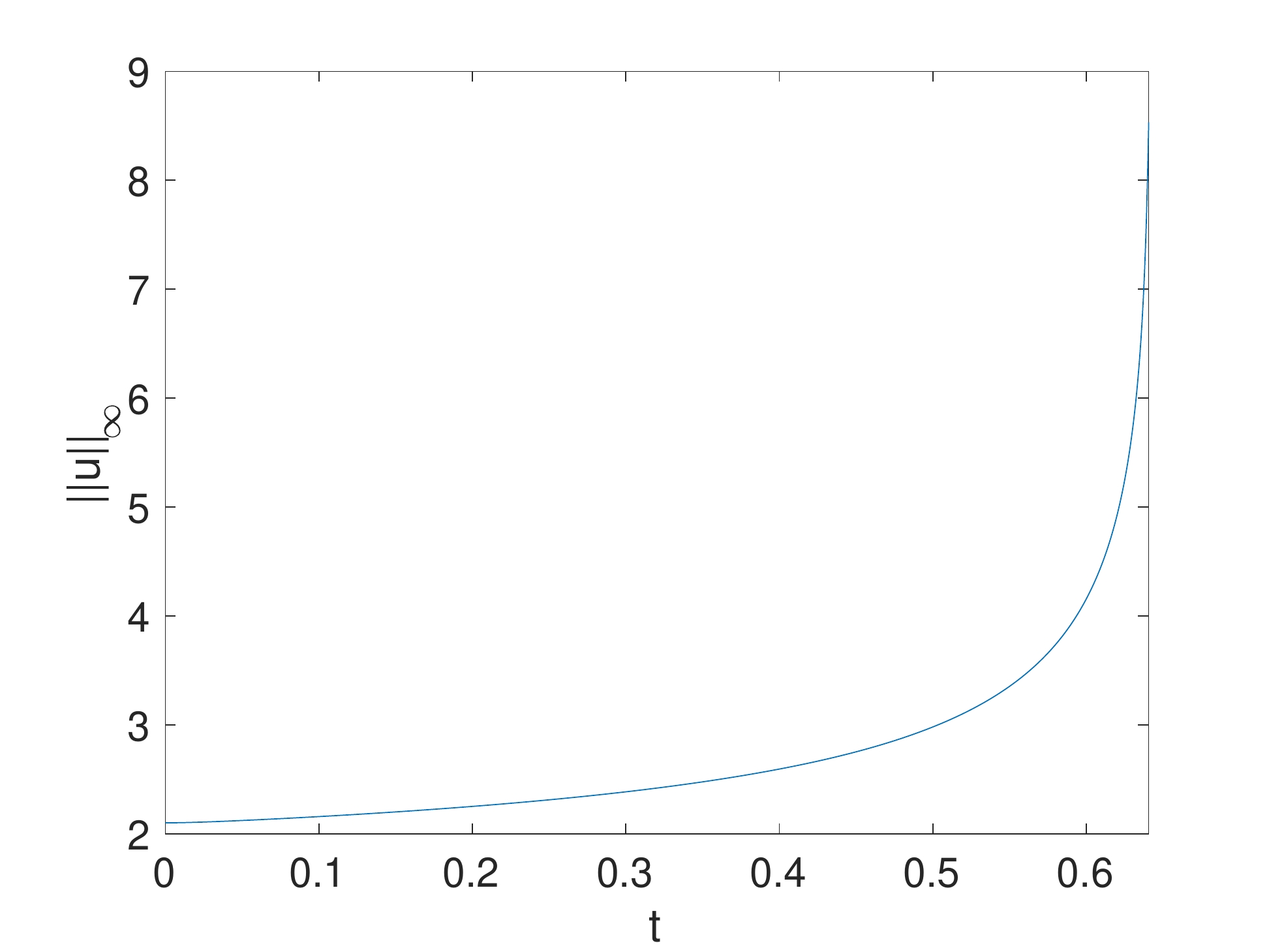} 
 \caption{Time-dependence of the $L^{\infty}$-norm of the solution to \eqref{parnos} with $\eps =1$, $\sigma=2$, and initial data \eqref{pert}: 
 On the left, the case with a $``-"$ perturbation; on the right the case with $``+"$ sign.}
 \label{Nlssolp2ex1+max}
\end{figure}
The modulus of the solution at the last recorded time $t=0.6405$ is shown in 
Fig.~\ref{Nlssolp2ex1+t06405} on the left. It can be seen that it is strongly compressed in the $x_2$-direction. 
The corresponding Fourier coefficients are shown on the right of the same figure. They also indicate the appearance of a singularity in the $x_{2}$-direction.  
\begin{figure}[htb!]
  \includegraphics[width=0.49\textwidth]{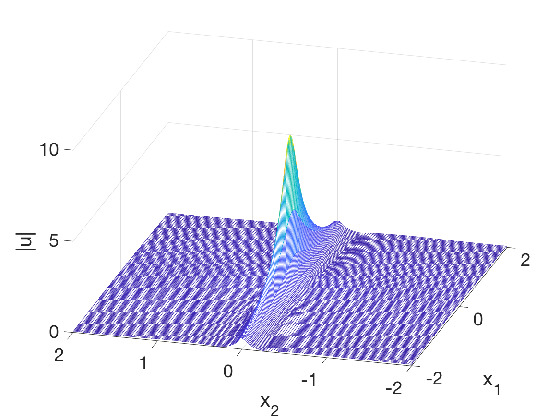}
  \includegraphics[width=0.49\textwidth]{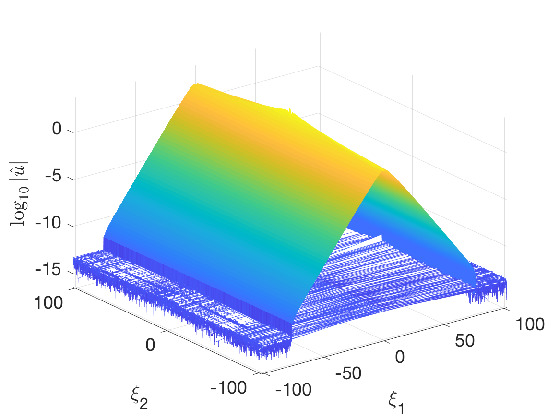} 
 \caption{Solution to \eqref{parnos} for 
 $\eps =1$, $\sigma=2$ and initial data (\ref{pert}) with the $``+"$ sign: On the left the modulus of the 
 solution at the last recorded time $t=0.6045$; on the right the corresponding Fourier coefficients of $u$ given by $\hat{u}$.}
 \label{Nlssolp2ex1+t06405}
\end{figure}

These numerical findings indicate that the global existence result stated in Theorem \ref{thm:subdel} is indeed sharp. 
It also shows that the two-dimensional model with partial off-axis variation essentially behaves like the classical one-dimensional focusing NLS in the 
unmodified $x_2$-direction (i.e., the direction in which $P_\eps$ does not act). Recall that for the classical 
one-dimensional (focusing) NLS, finite-time blow-up is known to appear as soon as $\sigma \ge 2$.

\subsection{The case with self-steepening parallel to the off-axis variation}
In this subsection, we include the effect of self-steepening and consider equation \eqref{ppDNLS} with $\eps =1$, $\delta_{2}=0$, and $\delta_{1}>0$. 

For $\sigma=1$, the stationary state $Qe^{it}$ appears to be stable against all studied perturbations. Indeed, the situation is found to be qualitatively similar to the case 
with full off-axis perturbations (except for a loss of radial symmetry) and we therefore omit a corresponding figure.

When $\sigma=2$, the stationary state no longer appears to be stable. However, we also do not have any indication of finite-time blow-up in this case. 
Indeed, given a $``-"$ perturbation in the initial data \eqref{pert}, it can be seen on the left of Fig.~\ref{ex1p2dx03} that the $L^{\infty}$-norm of the solution simply 
decreases monotonically in time. 
\begin{figure}[htb!]
  \includegraphics[width=0.32\textwidth]{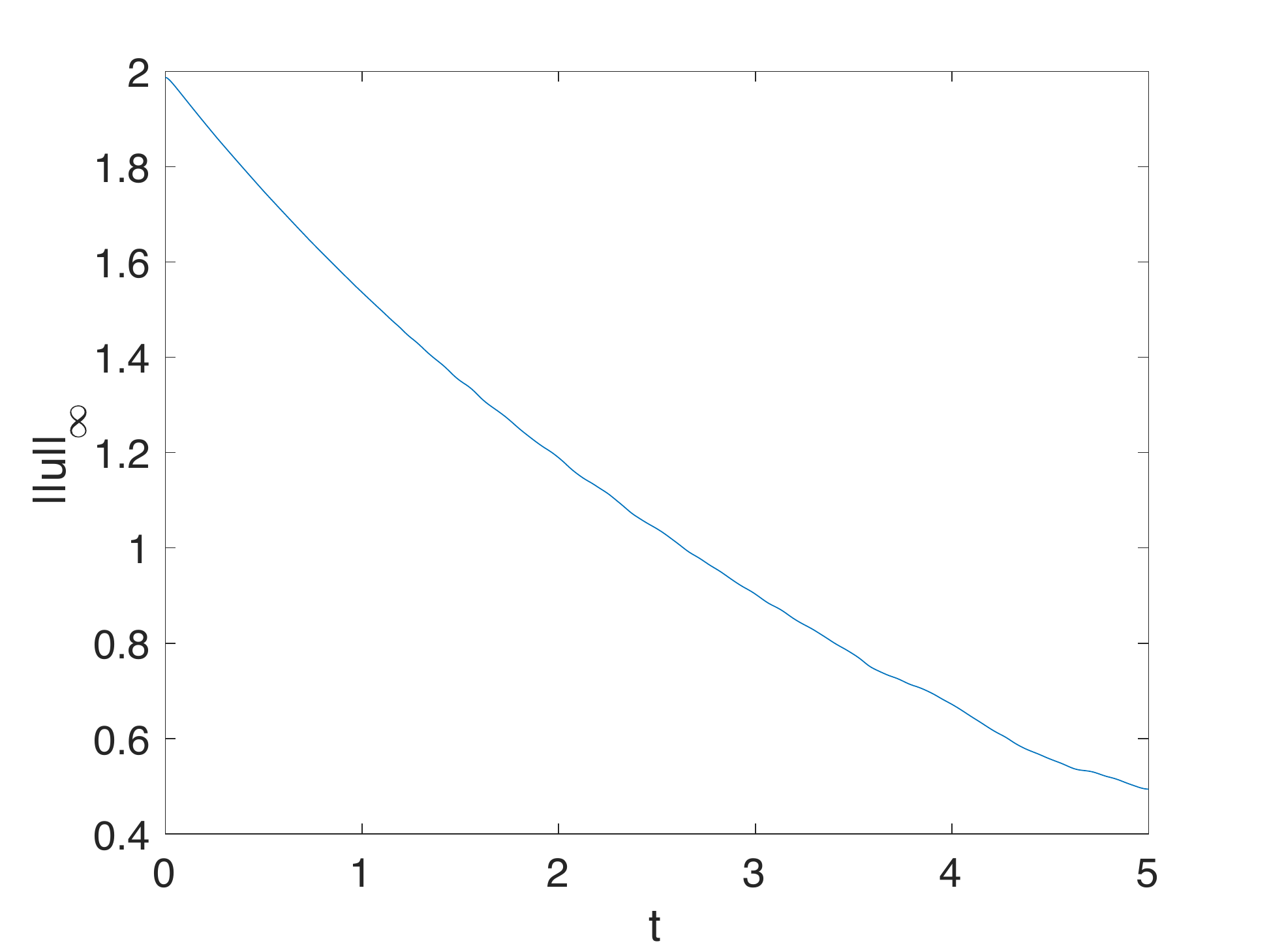}
  \includegraphics[width=0.32\textwidth]{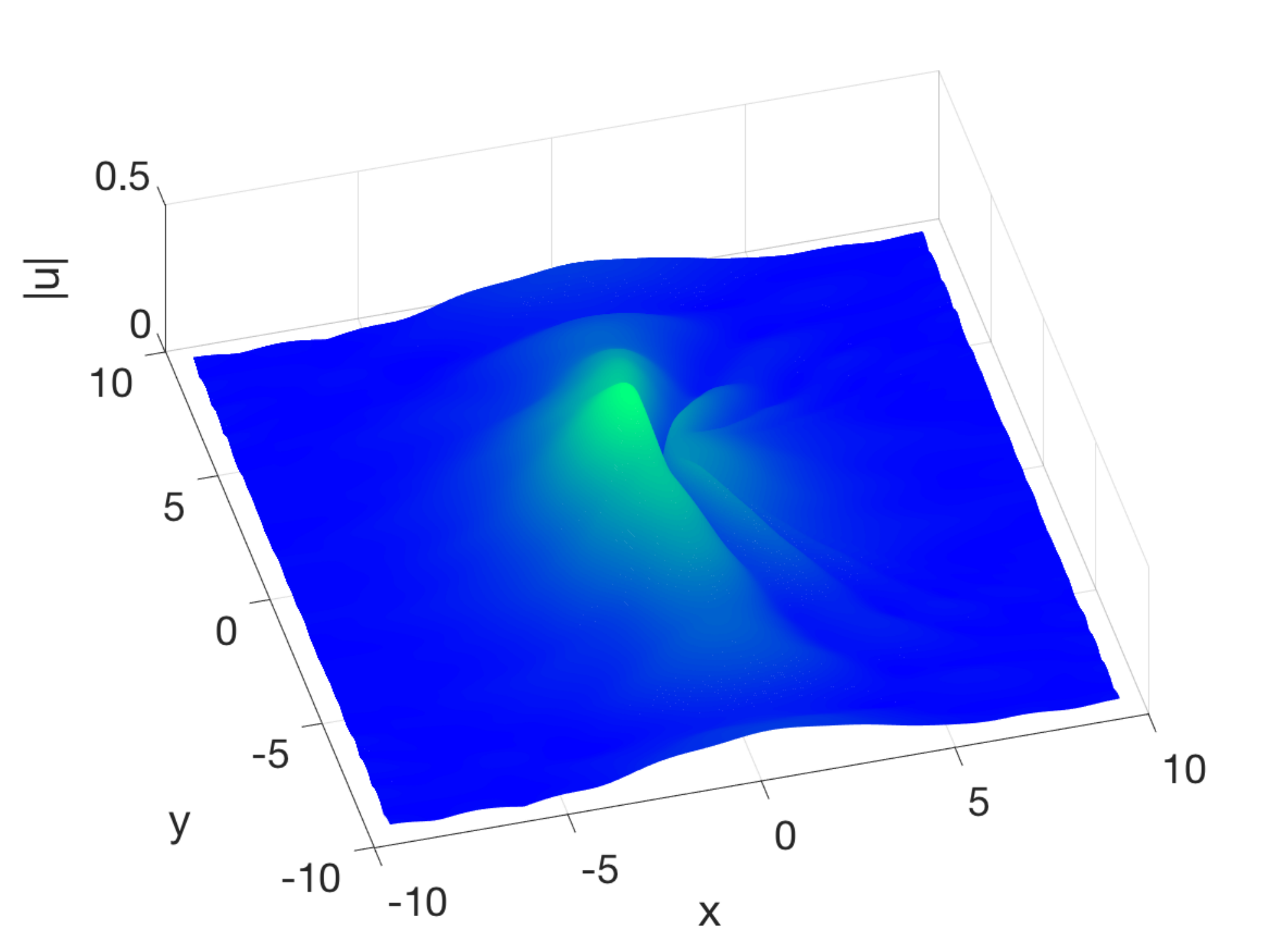}
  \includegraphics[width=0.32\textwidth]{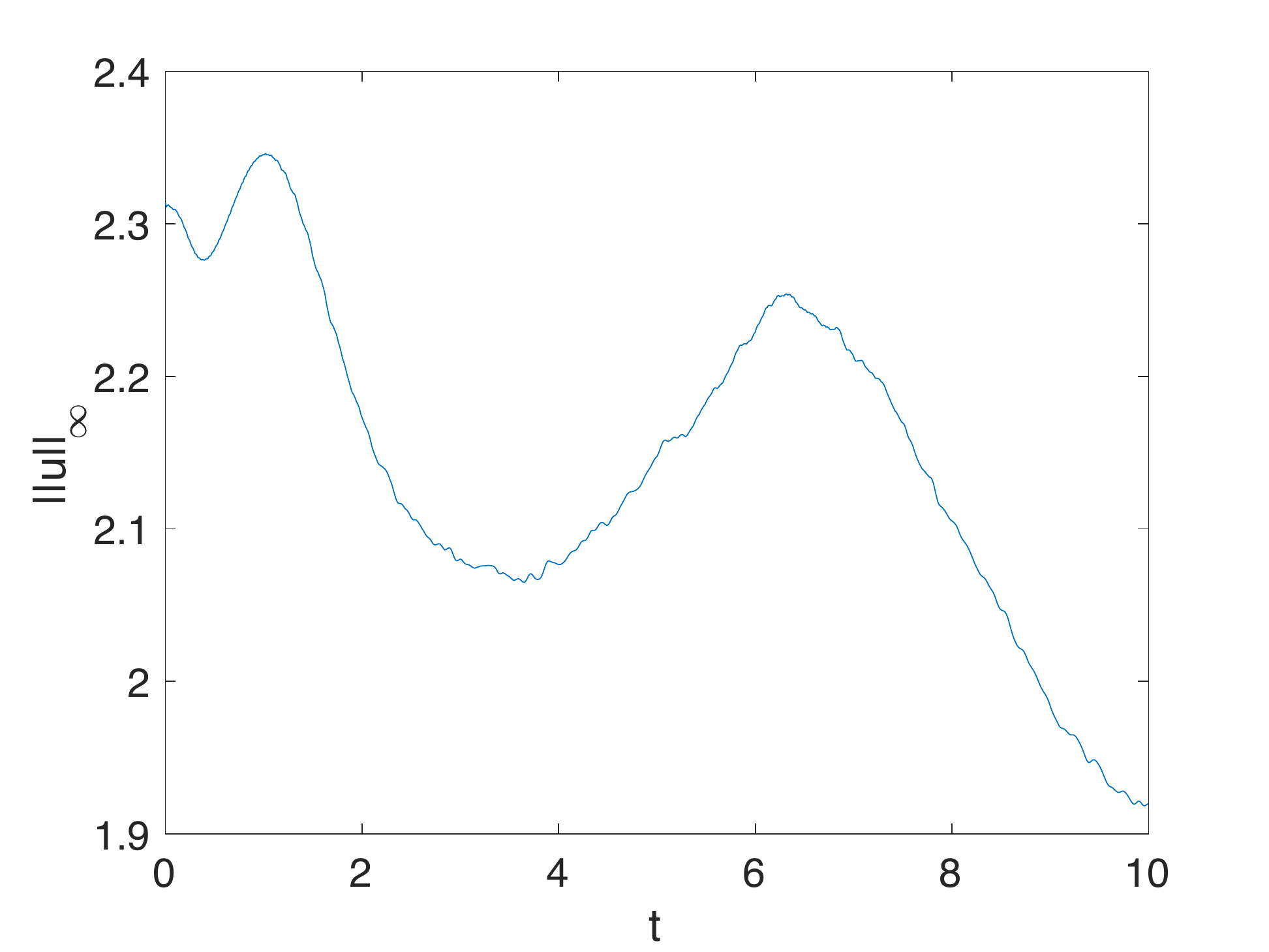}
 \caption{Solution to \eqref{ppDNLS} with 
 $\eps=1$, $\sigma=2$, and $\bd=(0.3, 0)^\top$: On the left, the $L^\infty$-norm of the solution obtained for initial data \eqref{pert} with the $``-"$ sign, 
 on the right for the $``+"$, and in the middle $|u|$ at $t=5$ for the $``-"$ sign perturbation.  }
 \label{ex1p2dx03}
\end{figure}
Notice, that there is still an effect of self-steepening visible in the modulus of the solution $|u|$, depicted in the middle of the same figure.  
The behavior of the $L^\infty$-norm in the case of a $``+"$ perturbation is shown on the right of Fig.~\ref{ex1p2dx03}. It is no longer monotonically decreasing but 
still converges to zero. 

For $\sigma=3$, a $``-"$ perturbation of
\eqref{pert} is found to be qualitatively similar to the case $\sigma =2$ and we therefore omit a figure illustrating this behavior. 
However, the situation radically changes if we consider a perturbation with the $``+"$ sign, see Fig.~\ref{ex1p3dx01}. The 
$L^{\infty}$-norm of the solution indicates a blow-up for $t\approx 0.1555$, where the code stops with an overflow error. 
\begin{figure}[htb!]
  \includegraphics[width=0.49\textwidth]{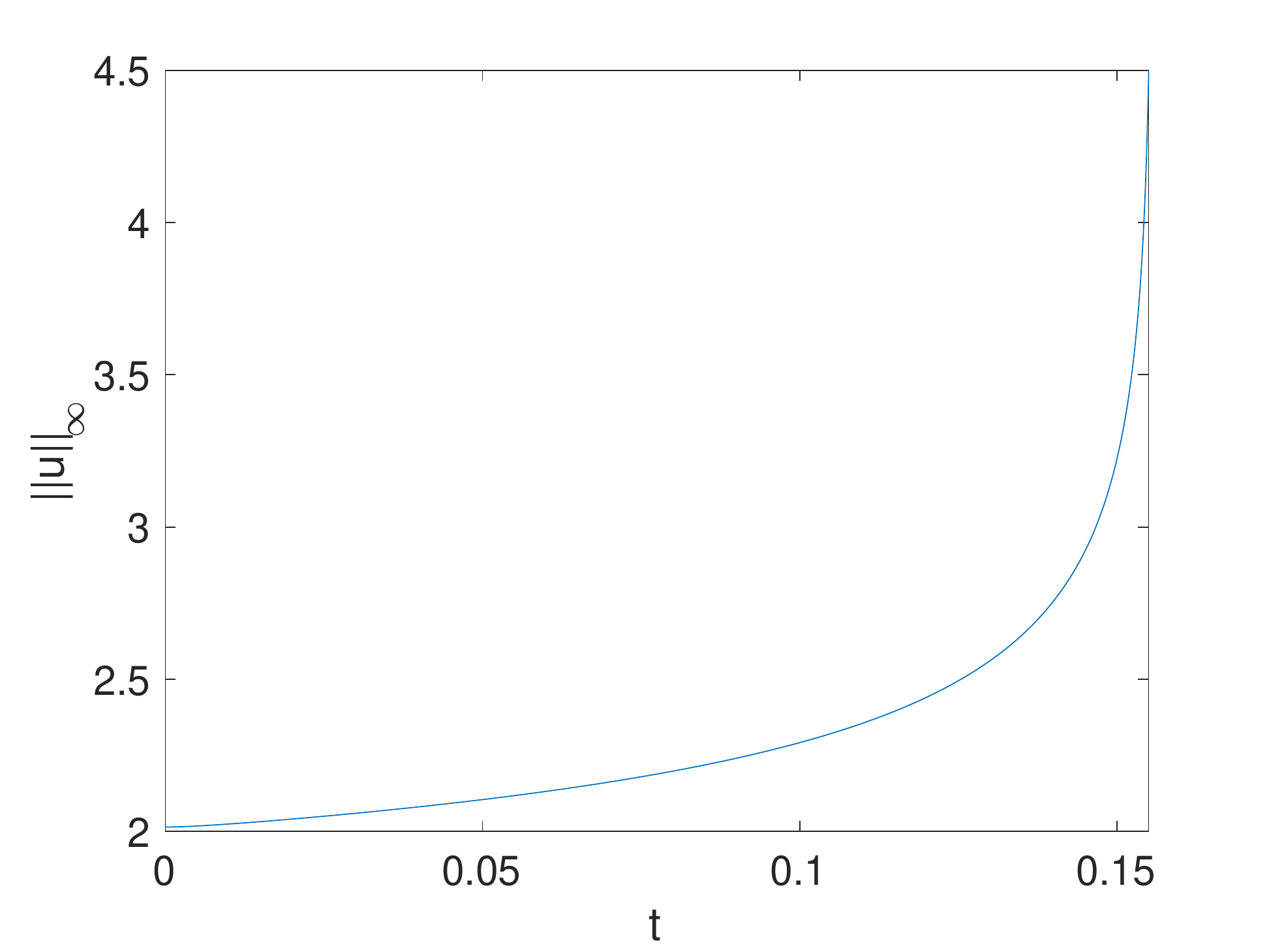}
  \includegraphics[width=0.49\textwidth]{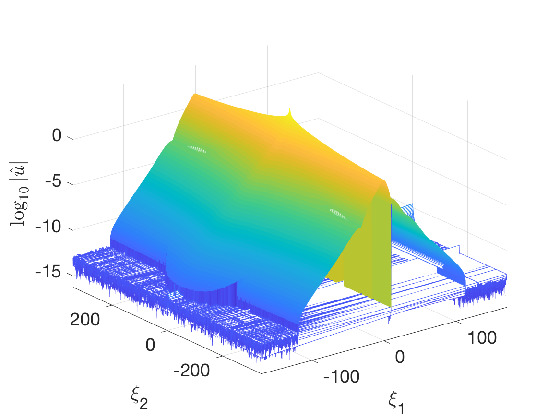}
 \caption{$L^{\infty}$-norm of the solution to \eqref{ppDNLS} with 
$\eps=1$, $\sigma=3$, $\bd_{1}=(0.1, 0)^\top$, and  
 initial data \eqref{pert} with the $``+"$ sign. On the right the 
 modulus of the Fourier coefficients of the solution at time $t=0.155$.  }
 \label{ex1p3dx01}
\end{figure}
In this particular simulation we have used $10^{4}$ time steps 
for $t\in[0,0.17]$ and $N_{x_{1}}=2^{10}$, $N_{x_{2}}=2^{11}$ Fourier modes (since the maximum of the solution hardly moved, it was not necessary to use a co-moving frame). 
The solution is still well resolved in time at $t=0.155$ since $M_\eps(t)$ remains numerically conserved up 
to the order of $10^{-11}$. But despite the higher resolution in $x_{2}$ used for this simulation, 
the Fourier coefficients indicate a loss of resolution in the $x_{2}$-direction. 
The modulus of the solution at the last recorded time is plotted in Fig.~\ref{ex1p3dx01t}. Note that $|u|$ is still regular in the $x_{1}$-direction in which $P_\eps^{-1}$ acts, 
but it has become strongly compressed in the $x_{2}$-direction. 
\begin{figure}[htb!]
  \includegraphics[width=0.49\textwidth]{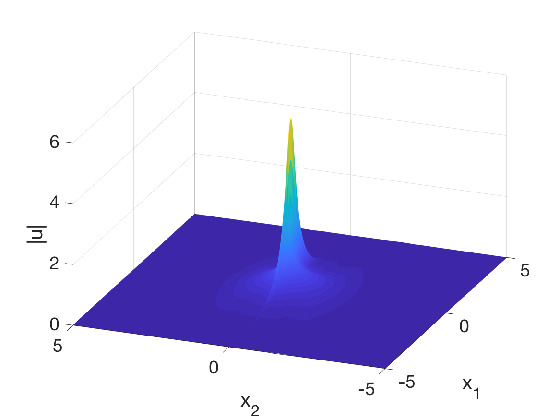}
 \caption{The modulus of the solution to equation \eqref{ppDNLS} with 
 $\eps=1$, $\sigma=3$, $\bd_{1}=(0.1, 0)$, and initial data \eqref{pert} with the $``+"$ sign, plotted at time $t=0.155$.}
 \label{ex1p3dx01t}
\end{figure}

\subsection{The case with self-steepening orthogonal to the off-axis variation}
Finally, we shall consider the same model equation \eqref{ppDNLS} with $\eps =1$, but this time we let $\delta_{1}=0$ for 
non-vanishing $\delta_{2}>0$. This is the only case, for which we do {\it not} have any analytical existence results at present.

For $\sigma=1$, it can be seen that a $``-"$ sign in the initial data \eqref{pert} yields a purely dispersive solution with monotonically decreasing $L^{\infty}$-norm, see Fig.~\ref{ex1d1} which also shows a picture of $|u|$ at $t=20$. The $``+"$ sign again leads to oscillations of the $L^{\infty}$-norm in time, indicating stability of the ground state. The situation for $\sigma =2$ is qualitatively very similar and hence we omit the corresponding figure.
\begin{figure}[htb!]
  \includegraphics[width=0.32\textwidth]{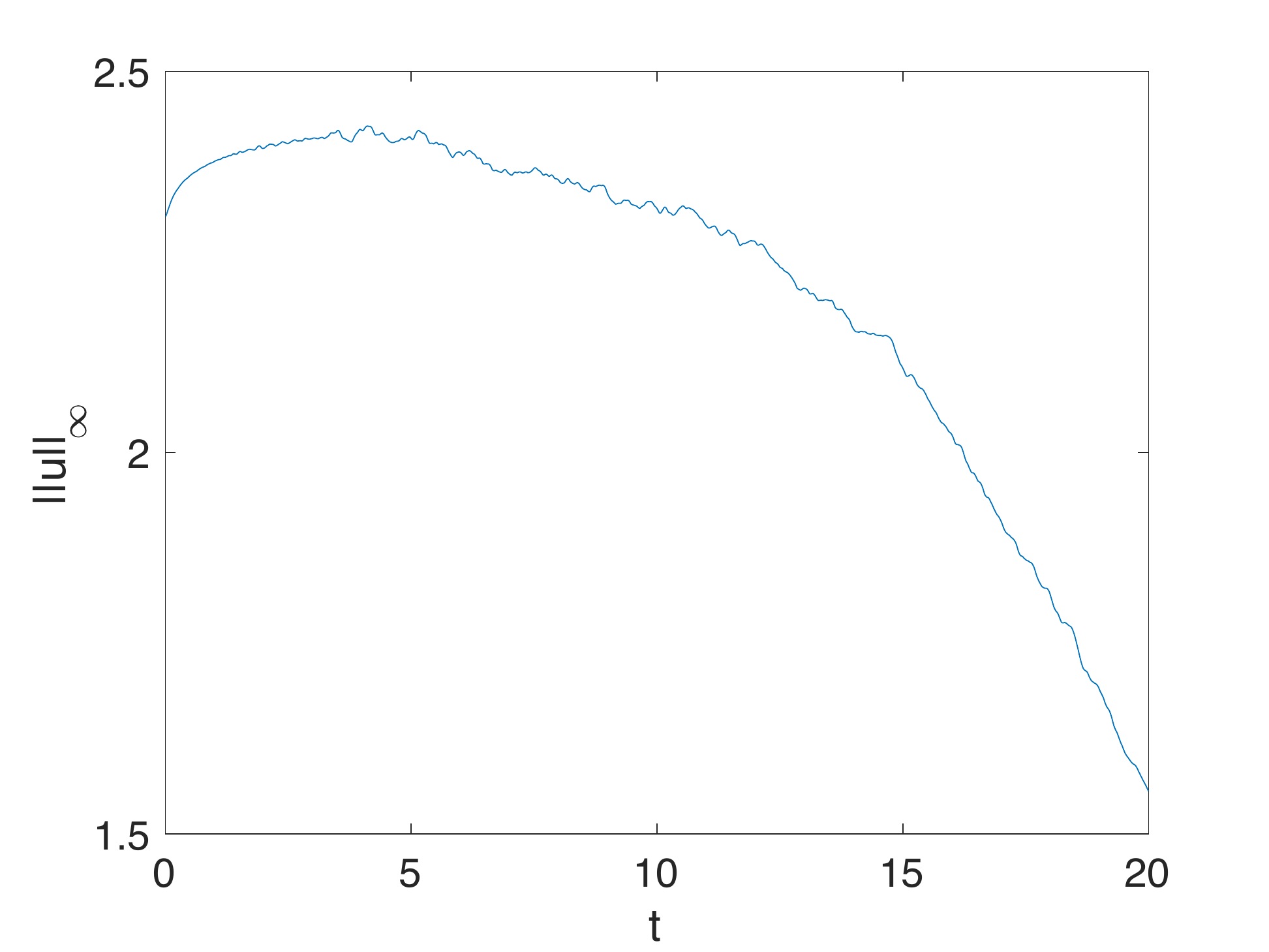}
  \includegraphics[width=0.32\textwidth]{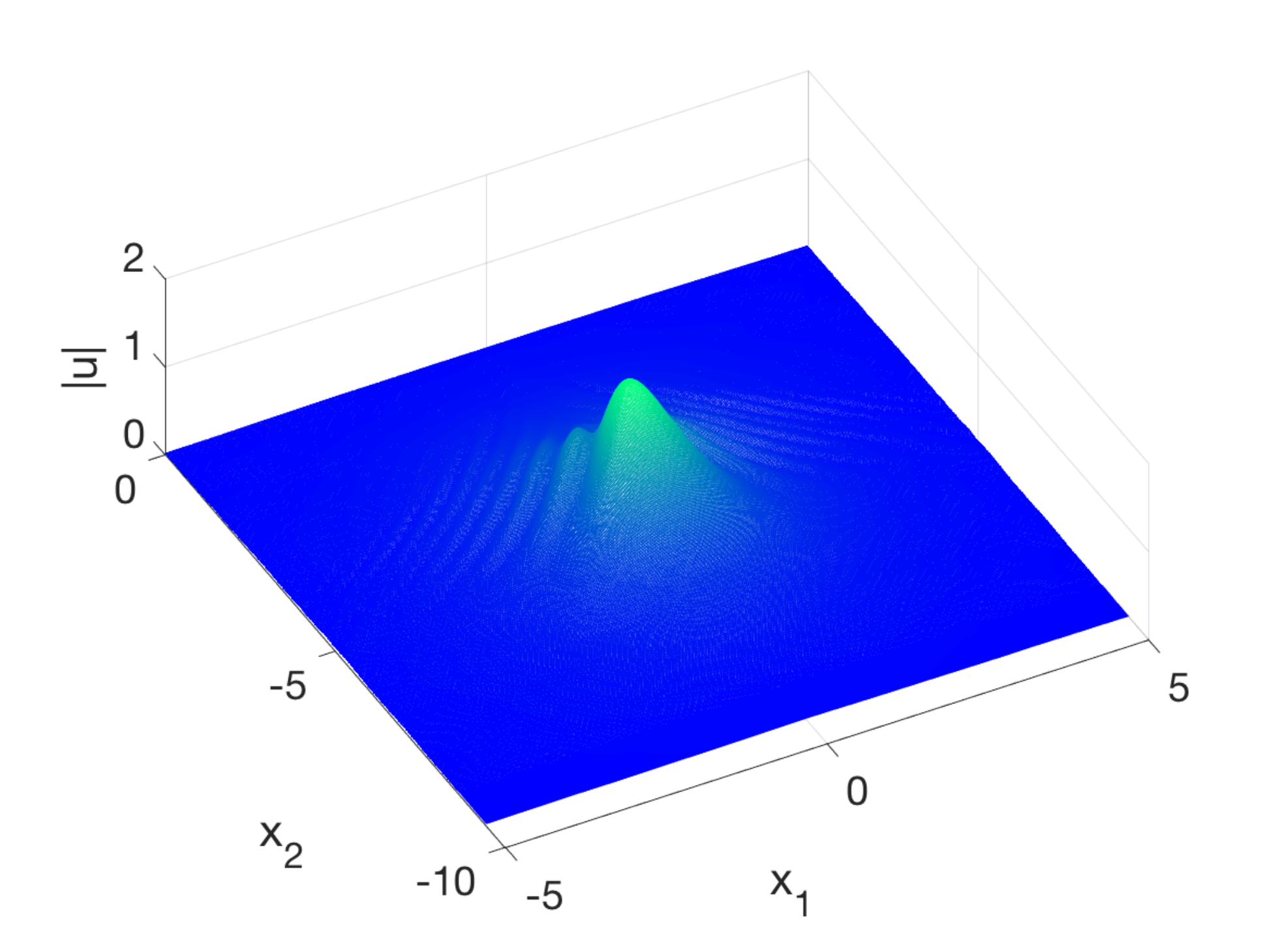}
  \includegraphics[width=0.32\textwidth]{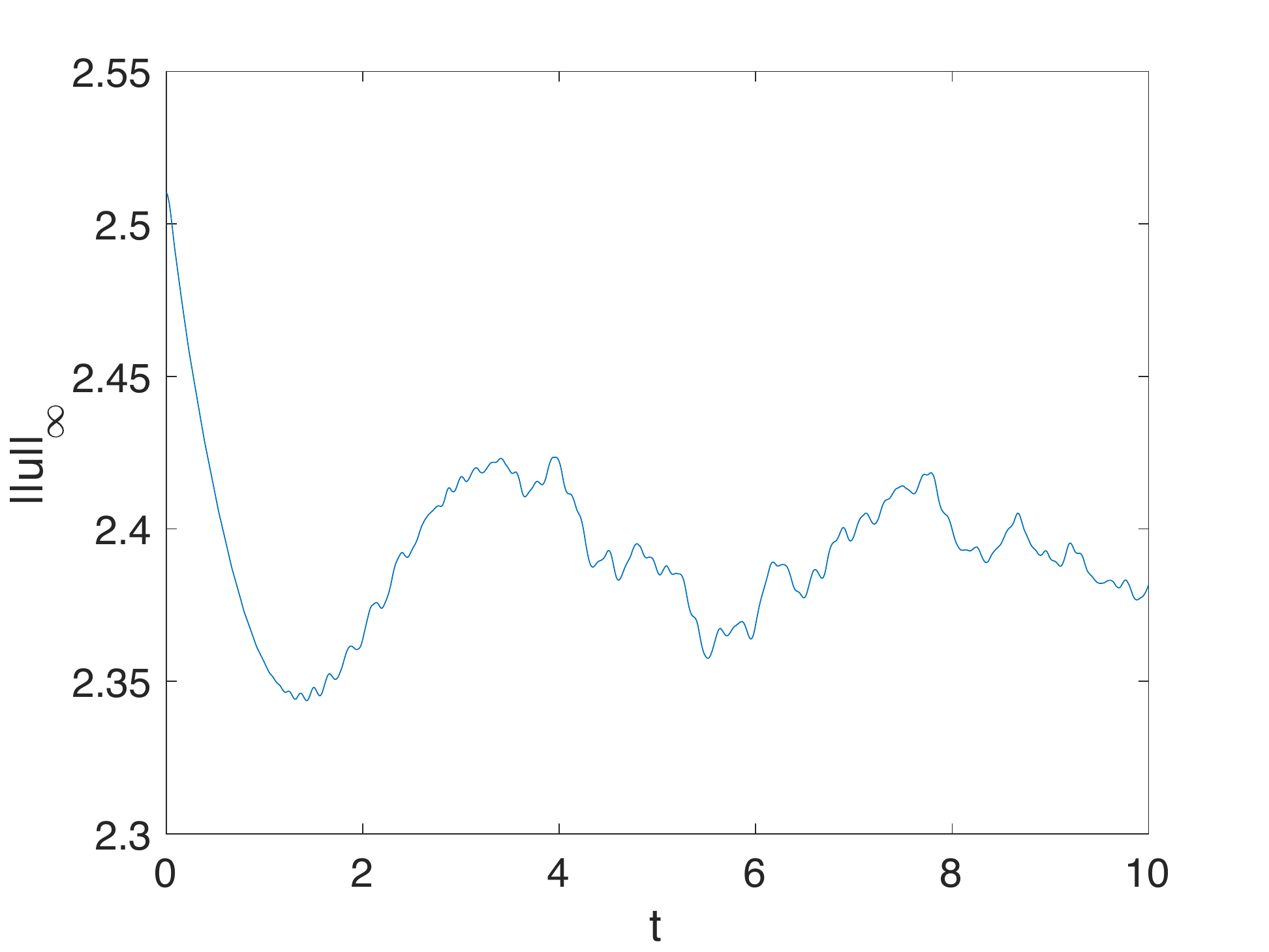}
 \caption{Solution to \eqref{ppDNLS} with $\eps=1$, $\sigma=1$, and $\bd_{1}=(0,1)^\top$. On the left the $L^\infty$-norm of the solution for initial data \eqref{pert} with 
 the $``-"$ sign, on the right for the one with $``+"$ sign, and in the middle $|u|$ at time $t=20$ for the $``-"$ sign.}
 \label{ex1d1}
\end{figure}

For $\sigma=3$ and a $``-"$ sign in the initial data \eqref{pert}, we again find a purely dispersive solution. 
However, the behavior of the solution obtained from a perturbation of $Q$ with the $``+"$ sign is less clear. 
As one can see in Fig.~\ref{ex1p3d01+}, the solution is initially focused up to a certain point after which 
its $L^{\infty}$-norm decreases again.
\begin{figure}[htb!]
  \includegraphics[width=0.49\textwidth]{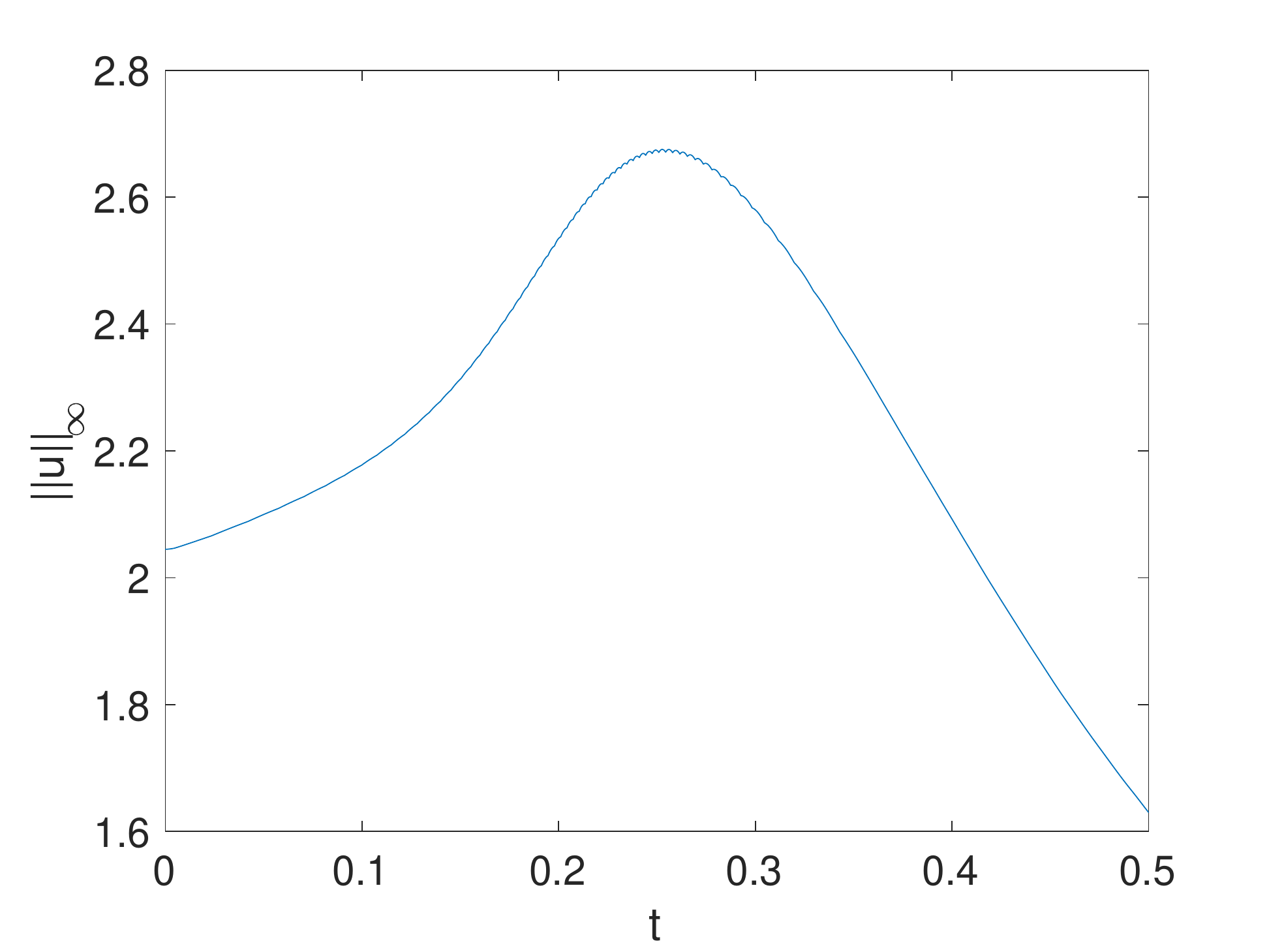}
  \includegraphics[width=0.49\textwidth]{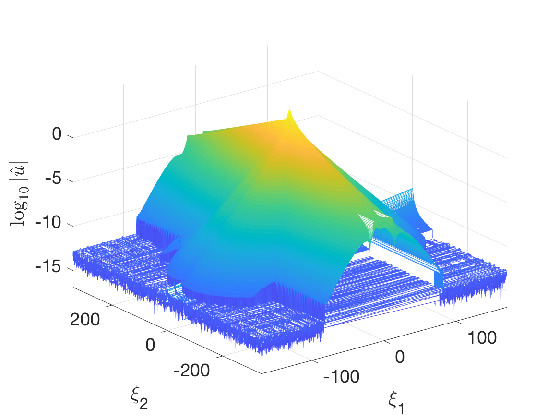}
 \caption{Solution to \eqref{ppDNLS} with $\eps=1$, $\sigma=3$, $\bd_{1}=(0,0.1)^\top$, and 
 initial data \eqref{pert} with the $``+"$ sign: On the left the $L^{\infty}$-norm of $u$ as a function of time, on the right the Fourier coefficients $\widehat u$ at $t=0.25$.}
 \label{ex1p3d01+}
\end{figure}

This simulation is done with 
$N_{x_{1}}=2^{10}$, $N_{x_{2}}=2^{11}$ Fourier modes and 
$N_{t}=10^{4}$ time steps for $t\in[0,0.5]$. The relative 
conservation of the numerically computed quantity $M_\eps(t)$ is 
better than $10^{-10}$ during the whole computation indicating an 
excellent resolution in time. The spatial resolution is indicated by the Fourier coefficients of the solution near the maximum of the $L^{\infty}$-norm 
as shown on the right of Fig.~\ref{ex1p3d01+}. Obviously, a much higher resolution is needed in the $x_{2}$-direction, but even near the maximum of the $L^{\infty}$-norm 
the modulus of the Fourier coefficients decreases to the order of $10^{-5}$. The modulus of the solution at time $t=0.5$ can be seen in 
Fig.~\ref{ex1p3d01+u}. It  shows a strong compression in the $x_{2}$-direction but nevertheless remains regular for all times. 
This is in stark contrast to the analogous situation with parallel self-steepening and off-axis variations, cf. Figures~\ref{ex1p3dx01} and \ref{ex1p3dx01t} above.
\begin{figure}[htb!]
  \includegraphics[width=0.49\textwidth]{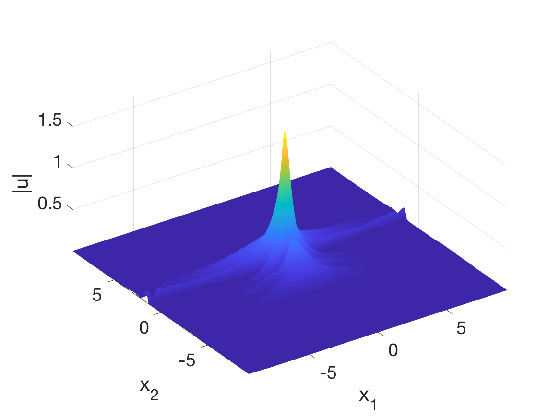}
 \caption{The modulus of the solution to \eqref{ppDNLS} with $\eps=1$, $\sigma=3$, $\bd_{1}=(0,0.1)^\top$, and initial data \eqref{pert} with the $``+"$ sign, plotted 
 at $t=0.5$.}
 \label{ex1p3d01+u}
\end{figure}


\bibliographystyle{amsplain}

\begin{thebibliography}{99}

\bibitem{AmSi} D. M. Ambrose and G. Simpson. {\it Local existence theory for derivative nonlinear Schr\"odinger equations with non-integer power nonlinearities.}
 SIAM J. Math. Anal. {\bf 47} (2015), no. 3, 2241--2264.

\bibitem{AAS} P. Antonelli, J. Arbunich, and C. Sparber, \emph{Regularizing nonlinear Schr\"odinger equations through partial off-axis variations}. SIAM J. Math. Anal. (2018), to appear.

\bibitem{BBM} T.~B. Benjamin, J.~L. Bona, and J.~J. Mahony, {\it Model equations for long waves in nonlinear dispersive systems}. Phil. Trans. Royal Soc. London A {\bf 227} (1972), 47--78.


\bibitem{BMR} J.~L. Bona, W.~R. McKinney, and J.~M. Restrepo, {\it Stable and unstable solitary-wave solutions of the generalized regularized long-wave equation}. 
J. Nonl. Sci. {\bf 10} (2000), no. 6, 603--638.

\bibitem{Car} R. Carles, {\it On Schr\"odinger equations with modified dispersion.} Dyn. Partial Differ. Equ. {\bf 8} (2011), no. 3, 173--184.

\bibitem{Cz} T. Cazenave, \emph{Semilinear Schr\"odinger equations}. Courant Lecture Notes in Mathematics vol. 10, American Mathematical Society, 2003. 

\bibitem{CaWe} T. Cazenave, F. Weissler, {\it Some remarks on the nonlinear Schr\"odinger equation in the critical case}. In: Lecture Notes Math. vol. 1394, Springer, New York, 1989, 19--29.

\bibitem{CO} M. Colin and M. Ohta, {\it Stability of solitary waves for derivative nonlinear Schr\"odinger equation}. Ann. Inst. H. Poincar\'e Anal. Non Lineaire {\bf 23} (2006), no. 5, 753--764.

\bibitem{DS} A. Davey and K. Stewartson, { \it On three-dimensional packets of water waves}. Proc. R. Soc. Lond. Ser. A {\bf 338} (1974), no. 1613, 101--110.

\bibitem{Dri} T. Driscoll, {\it A composite Runge-Kutta Method for the spectral solution of semilinear PDEs.}  J. Comput. Phys. {\bf 182} (2002), 357--367.

\bibitem{DLS} E. Dumas, D. Lannes, and J. Szeftel, \emph{Variants of the focusing NLS equation. Derivation, justification and open problems related to filamentation}. 
In: {\it CRM Series in Mathematical Physics}, pp. 19--75. Springer, 2016.

\bibitem{Fi} G. Fibich, {The nonlinear Schr\"odinger equation; Singular solutions and optical collapse}. Springer Series on Appl. Math. Sciences vol. 192, Springer Verlag, 2015.

\bibitem{GNW} Z. Guo, C. Ning, and Y. Wu, {\it Instability of the solitary wave solutions for the genenalized derivative Nonlinear Schr\"odinger equation in the critical frequency case}. 
Preprint {\tt  arXiv:1803.07700}.

\bibitem{HaOz} N. Hayashi and T. Ozawa. {\it On the derivative nonlinear Schr\"odinger equation}. {Phys. D} {\bf 55} (1992), no. 1-2, 14--36.

\bibitem{JLPS1} R. Jenkins, J. Liu, P. Perry, and C. Sulem, {\it Soliton resolution for the derivative nonlinear Schr\"odinger Equation}. Preprint {\tt  arXiv:1710.03819}.

\bibitem{JLPS2} R. Jenkins, J. Liu, P. Perry, and C. Sulem, {\it Global well-posesedness for the derivative nonlinear Schr\"odinger Equation}. Preprint {\tt arXiv:1710.03810.}

\bibitem{etna} C.~Klein, {\it Fourth-order time-stepping for low dispersion Korteweg-de Vries and nonlinear Schr\"odinger Equation}. Electronic Trans. Num. Anal. {\bf 39} (2008), 116--135.

\bibitem{KMR}C. Klein, B. Muite, and K. Roidot, {\it Numerical study of blowup in the Davey-Stewartson System}. Discrete Contin. Dyn. Syst. Ser. B, {\bf 18} (2013), no. 5, 1361--1387.

\bibitem{KP14} C.~Klein and R.~Peter, {\it Numerical study of blow-up in solutions to generalized Kadomtsev-Petviashvili equations}. Discrete Contin. Dyn. Syst. Ser. B {\bf 19} (2014), no. 6,  1689--1717.

\bibitem{KP16} C.~Klein and R.~Peter, {\it Numerical study of blow-up in solutions to generalized Korteweg-de Vries equations}. Phys. D {\bf 304} (2015), 52--78.

\bibitem{KR11} C.~Klein and K.~Roidot,  {\it Fourth order time-stepping for Kadomtsev-Petviashvili and 
Davey-Stewartson equations}.  SIAM J. Sci. Comput. {\bf 33} (2011), no. 6, 3333--3356.

\bibitem{KS}C. Klein and J.-C. Saut, {\it A numerical approach to blow-up issues for dispersive perturbations of Burgers' equation}. Phys. D {\bf 295} (2015), 46--65.

\bibitem{KSDS}C. Klein and J.-C. Saut, {\it A numerical approach to Blow-up issues for Davey-Stewartson II
type systems}. Comm. Pure Appl. Anal. {\bf 14} (2015), no. 4,  1443--1467.

\bibitem{KMS} C.~Klein, C.~Sparber, and P.~Markowich, {\it Numerical study of fractional Nonlinear Schr\"odinger equations}. Proc. R. Soc. Lond. Ser. A. {\bf 470} (2014) 20140364, 26pp.

\bibitem{KN} C. Klein and N. Stoilov, {\it A numerical study of blow-up mechanisms for Davey-Stewartson II systems}, Stud. Appl. Math. (2018), to appear.

\bibitem{krasny} R. Krasny, {\it A study of singularity formation in a vortex sheet by the point-vortex approximation}, J. Fluid Mech. {\bf167} (1986), 65--93.

\bibitem{LSS1} X. Liu, G. Simpson, and C. Sulem, { \it Stability of solitary waves for a generalized derivative nonlinear Schr\"odinger Equation}. J. Nonlin. Sci. {\bf 23} (2013), no. 4, 557--583.

\bibitem{LSS2} X. Liu, G. Simpson, and C. Sulem, { \it Focusing singularity in a derivative nonlinear Schr\"odinger equation}. Phys. D {\bf 262} (2013), 45--58.

\bibitem{MFP} M. McConnell, A. Fokas, and B. Pelloni, {\it Localised coherent solutions of the DSI and DSII equations a numerical study}. 
Math. Comput. Simul. {\bf 69} (2005), no. 5/6, 424--438.

\bibitem{MR1} F. Merle and P. Rapha\"el, {\it On universality of blow up profile for $L^2$ critical nonlinear Schr\"odinger equation}. Invent. Math. {\bf 156} (2004), 565--672.

\bibitem{MR2} F. Merle and P. Rapha\"el, {\it Profiles and quantization of the blow up mass for critical nonlinear Schr\"odinger equation}. Comm. Math. Phys. {\bf 253} (2005), no. 3, 675--704.

\bibitem{Oz} T. Ozawa, {\it Remarks on proofs of conservation laws for nonlinear Schr\"odinger equations}. Calc. Var. Partial Differ. Equ. {\bf 25} (2006), no. 3, 403--408.

\bibitem{Paz} A. Pazy, Semigroups of Linear Operators and Applications to Partial Differential Equations, Springer Verlag, New York, 1983.

\bibitem{GMRES} Y. Saad and M. Schultz, {\it GMRES: A generalized minimal residual algorithm for solving nonsymmetric linear systems.} SIAM J. Sci. Comput. {\bf 7} (1986), no. 3, 856--869.

\bibitem{Sul} C. Sulem and P.~L. Sulem, {\it The nonlinear Schr\"odinger equation: self-focusing and wave collapse}. Springer Series on Applied Mathematical Sciences vol. 139, Springer, Berlin, 1999. 

\bibitem{TF} M. Tsutsumi and I. Fukuda, {\it On solutions of the derivative nonlinear Schr\"odinger equation. Existence and uniqueness theorem}. 
{Fako l’Funk. Ekvacioj Japana Mat. Societo.} {\bf 23} (1980), no. 3, 259--277. 

\bibitem{Wu} Y. Wu, {\it Global well-posedness on the derivative nonlinear Schr\"odinger equation revisited}.  Anal. PDE {\bf 8} (2015), no. 5, 1101--1112.

\end{thebibliography}

\end{document}